\documentclass[reqno]{amsart}
\usepackage[a4paper]{geometry}

\usepackage[T2A]{fontenc}
\usepackage[utf8]{inputenc}
\usepackage[ukrainian,russian,english]{babel}

\usepackage{comment}

\usepackage{graphicx}
\usepackage{amsmath}

\DeclareFontFamily{U}{mathx}{}
\DeclareFontShape{U}{mathx}{m}{n}{<-> mathx10}{}
\DeclareSymbolFont{mathx}{U}{mathx}{m}{n}

\DeclareMathAccent{\widehat}{0}{mathx}{"70}
\DeclareMathAccent{\widecheck}{0}{mathx}{"71}

\usepackage{amssymb}
\usepackage[mathscr]{eucal}
\usepackage{tikz}
\usetikzlibrary{positioning,calc,decorations.markings}
\usetikzlibrary{calc}
\usepackage{mathtools}
\usepackage{enumerate}
\usepackage{accents}
\usepackage{dsfont}
\usepackage{multicol}
\usepackage{color}
\usepackage[colorlinks=true]{hyperref}
\hypersetup{urlcolor=blue,citecolor=red,linkcolor=blue}
\usepackage[initials]{amsrefs}
\definecolor{darkgreen}{rgb}{0.13, 0.55, 0.13}
\numberwithin{equation}{section}
\theoremstyle{plain}
\newtheorem{theorem}{Theorem}[section]
\newtheorem{proposition}[theorem]{Proposition}
\newtheorem{corollary}[theorem]{Corollary}
\newtheorem{lemma}[theorem]{Lemma}
\theoremstyle{definition}
\newtheorem*{rh-pb*}{Basic RH problem}
\newtheorem*{rh-I*}{RH problem normalized as $I$ at $\infty$}
\newtheorem*{sol-rh-pb*}{Soliton RH problem}
\newtheorem*{data*}{Data of this RH problem associated with $\BS{u_0(x)}$}
\theoremstyle{remark}
\newtheorem{remark}[theorem]{Remark}
\newtheorem*{notations*}{Notations}
\providecommand{\BS}[1]{\boldsymbol{#1}}  
\providecommand{\D}[1]{\mathbb{#1}}
\newcommand{\dd}{\mathrm{d}}
\newcommand{\eul}{\mathrm{e}}
\newcommand{\ii}{\mathrm{i}}
\newlength{\dhatheight}

\renewcommand{\Im}{\operatorname{Im}}

\newcommand{\ord}{\mathrm{O}}
\DeclareMathOperator{\Res}{Res}

\newcommand{\red}{\textcolor{red}}

\newif\ifshort
\shorttrue

\title{The sine Gordon equation in light-cone coordinates on the half lines revisited: a Riemann--Hilbert approach}

\author[Iryna Karpenko]{Iryna Karpenko}
\address{IK: Faculty of Mathematics\\ University of Vienna\\
Oskar-Morgenstern-Platz 1\\ 1090 Wien\\ Austria\\ and B. Verkin Institute for Low Temperature Physics and Engineering\\ 47, Nauky ave\\ 61103 Kharkiv\\ Ukraine}
\email{\href{mailto:iryna.karpenko@univie.ac.at}{iryna.karpenko@univie.ac.at}}

\begin{document}

\begin{abstract} 
In this work, we study the initial boundary value (IBV) problems 
for the sine-Gordon (sG) equation in the light-cone coordinates  $u_{xt}=\sin u$  in the quarter planes $x> 0$, $t>0$ and $x< 0$, $t>0$ assuming a suitable decay as $x\to +\infty$ or as $x\to -\infty$. Employing the Riemann–Hilbert (RH) problem framework, we demonstrate that these two IBV problems differ significantly with respect to the boundary data required for well-posedness. Specifically, the solution of the ``right problem'' ($x\ge 0$)
is uniquely determined by the initial data $u(x,0)$, $x\ge 0$ alone whereas for the ``left problem'' ($x\le 0$), the boundary
data $u(0,t)$ has to be prescribed in addition to the initial data in order to obtain a well-posed problem.

\end{abstract}

\maketitle

\section{Introduction}\label{sec:1}

The sine–Gordon (sG) equation is one of the classical integrable nonlinear equations and has been extensively studied. In the  laboratory coordinates $(y,\tau)$ it takes the form
\[
u_{yy}-u_{\tau\tau}+\sin u=0.
\]
The sG equation arises in differential geometry in the theory of surfaces of constant negative curvature and appears in several areas of physics, including the theory of crystal dislocations and the dynamics of Josephson junctions (see, e.g., \cites{FrenkelKontorova1939,ScottChuMcLaughlin1973,AblowitzClarkson1991,FaddeevTakhtajan1987}).

Introducing the light-cone variables
\[
x=\frac{1}{2}(y+\tau),\quad t=\frac{1}{2}(y-\tau)
\]
the equation transforms into
\begin{equation}\label{sG}
u_{xt} = \sin u.
\end{equation}
In these variables,
$x$ and $t$ correspond to the characteristic directions of a hyperbolic equation.

The Cauchy problem on the whole line for the sG equation was solved using the inverse scattering transform by Ablowitz, Kaup, Newell, and Segur (see \cites{AKNS1973,KaupNewell1978}). This method reveals the rich structure of the sG equation, including the existence of localized solutions such as kinks and breathers, which correspond to discrete eigenvalues of the associated spectral problem (see \cite{KaupNewell1978}).

The analysis becomes
considerably more complicated when the sG equation is posed in domains
with boundaries, such as quarter-planes and semi-strips. To address such problems, Fokas introduced a general and systematic framework known as the Unified Transform Method (UTM), widely referred to as the Fokas method (see \cite{F02}), which has subsequently been developed and extended by many authors \cites{fokas2008unified,BS08,FIS,BFS06,KS25,LF2009}. The method is based on the simultaneous spectral analysis of both equations of the associated Lax pair.

In \cite{FokasIts}, Fokas and Its studied the sG equation in laboratory coordinates in the quarter-plane with prescribed initial data and boundary values (see also \cites{F02,HL18}). In a later work \cite{Fokas1997}, the equation $u_{xt}=-\sin u$ was considered on a finite spatial interval $(0,l)$. In both formulations, the initial profile and the
boundary values are assumed to be prescribed as part of the data of the problem. 

An inverse spectral approach was used by Sakhnovich, who investigated the
sG equation in semi-strip domains using the Weyl function associated
with an auxiliary Dirac-type system arising from the Lax pair
(see \cites{Sakhnovich1992,Sakhnovich2010}). In particular, this approach makes it possible to study the
evolution of the corresponding spectral data and to describe situations in which solutions become unbounded in the quarter-plane. In \cite{Vu2005}, 
the Dirichlet problem on the half-line is considered with prescribed initial data 
$u(x,0)$ and boundary data 
$u(0,t)$. These data are required to satisfy appropriate compatibility conditions, and the solution is reconstructed from the corresponding scattering data.

A different inverse spectral approach to the Dirichlet problem on the half-line was proposed by Leon \cite{Leon2002}. In this work, the sG equation in light-cone coordinates (in the form $u_{xt}=-\sin u$, which can be related to \eqref{sG} by the change of variable $x\mapsto -x$)  is analyzed using an alternative Lax pair associated with a Schr\"odinger-type spectral problem on the half-line. 

Despite the extensive literature on initial-boundary value (IBV) problems for the sine-Gordon equation, several questions concerning the well-posedness of these problems have remained open, particularly regarding the decay conditions as $|x|\to\infty$. In the present paper, we address these questions by studying two initial-boundary value problems for the sG equation in the light-cone coordinates \eqref{sG}, posed on the half-lines $x\ge 0$ (Problem I) and $x\le 0$ (Problem II) building on the Fokas method for half line problems. 

In what follows, we will use standard functional spaces
\begin{align*}
    & H^{2,2}(\Omega):=\{f(x)\in L^2(\Omega):~ x^2f(x), f'(x), x^2f'(x),f''(x), x^2f''(x) \in L^2(\Omega)\},\\
    & C([0,T],X)
:=
\left\{
u:[0,T]\to X:
u \text{ is continuous on } [0,T]
\text{ with values in } X
\right\},\\
& C^{k}([0,T],X)
:=
\{
u:[0,T]\to X:
\partial_t^{,j}u\in C([0,T],X),
\quad j=0,1,\ldots,k
\},
\end{align*}
and 
\[
 \tilde H^{1,2}(\Omega):=\{f(x)\in L^1_{loc}(\Omega): (1+x^2)\sin \frac{f(x)}{2},(1+x^2)f_x(x)\in L^2(\Omega)\}
\]
with a pseudometric 
\[
d_{\tilde H^{1,2}}(f,g)
:=
\left(
\left\|(1+x^2)\sin\!\left(\frac{f-g}{2}\right)\right\|_{L^2(\Omega)}^2
+
\left\|(1+x^2)(f_x-g_x)\right\|_{L^2(\Omega)}^2
\right)^{\frac12}.
\]

For Problem I, we assume that 
the initial condition $u(x,0) = u_0(x)$ for $x \geq 0$
is given and $u_0(x)\in H^{2,2}((0,+\infty))+2\pi k$. We consider the following problem
\begin{equation}\label{pr_1}
    \begin{aligned}
        &u_{xt} = \sin u,\quad x\geq 0,\quad  0\leq t<T<\infty,\\
         &u(x,0) = u_0(x), \quad  x \geq 0
    \end{aligned}
\end{equation}
and seek solutions $u(x,t)$ such that $u(x, t)\in C^1([0,T],\tilde H^{1,2}(0,\infty))$.

For Problem II, we assume that 
the initial condition  $u(x,0) = u_0(x)\in H^{2,2}((-\infty,0))+2\pi k$ for $x \leq 0$, and the boundary condition $ u(0,t)=v_0(t)\in H^{1}(0,T)$
are given. We consider the following IBV problem 
\begin{equation}\label{pr_2}
    \begin{aligned}
        &u_{xt} = \sin u,\quad x\leq 0,\quad  0\leq t<T<\infty,\\
         &u(x,0) = u_0(x), \quad  x \leq 0,\\
        & u(0,t)=v_0(t),\qquad t\in[0,T]
    \end{aligned}
\end{equation}
and seek solutions $u(x,t)$ such that
 $u(x,t)\in C^1([0,T],\tilde H^{1,2}((-\infty,0)))$.

The conditions imposed on the solution of the IBV problems at the spatial infinity
affect crucially the well-posedness of half-line problems.
A striking illustration is provided by Sakhnovich’s treatment of the sine--Gordon equation on the quarter-plane $x\ge 0$, $t\ge 0$ \cite{Sakhnovich2010}. Under suitable assumptions on the boundary data, assuming also that 
the spatial derivative of the solution is bounded in the whole quarter-plane, the boundary data determine uniquely the initial Weyl function 
and thus the initial values; see \cite[Corollary~4.10]{Sakhnovich2010}. In particular, zero boundary data force the initial Weyl function, and therefore the corresponding initial data, to be trivial; see \cite[Example~4.11]{Sakhnovich2010}.  

The outline of the paper is as follows.

In Section \ref{sec:2}, we present the basic tools for our analysis derived from the Lax pair representation of the sG equation, namely the Jost solutions (“eigenfunctions”) and the associated spectral functions (scattering coefficients). 

In Section \ref{sec:3}, we discuss the compatibility conditions for the initial and boundary values.

In Section \ref{sec:4}, we describe the inverse spectral problems for the initial and boundary values in terms of solutions of the associated RH problems.

In Section \ref{sec:5}, the master RH problems are constructed,
 whose solutions provide the solutions to Problem I and Problem II.  

In Section \ref{sec:6}, we show how a local solution of the sG equation can be constructed from the solution of a RH problem parametrized by $x$ and $t$. 

In Section \ref{sec:7}, the differentiability of the solution of the RH problem is discussed.

In Section \ref{sec:8}, we address Problem I, for which we find
that the initial data alone uniquely determine its solution. 

In Section \ref{sec:9}, we focus on Problem II, where we show that in addition to the initial data, the boundary value $u(0,t)$ should be prescribed to uniquely determine the solution.

\begin{notations*}
In what follows, $\sigma_1\coloneqq\left(\begin{smallmatrix}0&1\\1&0\end{smallmatrix}\right)$, $\sigma_2\coloneqq\left(\begin{smallmatrix}0&-\ii\\\ii&0\end{smallmatrix}\right)$, and $\sigma_3\coloneqq\left(\begin{smallmatrix}1&0\\0&-1\end{smallmatrix}\right)$ denote the standard Pauli matrices, 
$e^{\hat\sigma_3}A:=e^{\sigma_3}A e^{-\sigma_3}$,
$\mathbb{C}^+:=\{\lambda\in\mathbb{C}|\Im(\lambda)> 0\}$, and 
$\mathbb{C}^-:=\{\lambda\in\mathbb{C}|\Im(\lambda)< 0\}$. We also let $f^*(k)\coloneqq\overline{f(\bar k)}$ denote the Schwarz 
reflection of a function $f(k)$, $k\in\D{C}$.
\end{notations*}

\section{Eigenfunctions and Spectral Functions}\label{sec:2}

The sine-Gordon equation \eqref{sG} is the compatibility condition for the Lax pair
\begin{subequations}\label{Lax}
\begin{align}\label{Lax-x}
\tilde\Phi_x&=U\tilde\Phi,\\
\label{Lax-t}
\tilde\Phi_t&=V\tilde\Phi,
\end{align}
\end{subequations}
where the coefficients $U\equiv U(x,t,k)$ and $V\equiv V(x,t,k)$ are defined by
\begin{subequations}\label{Lax-UV}
\begin{align}\label{Lax-U}
U&=-\ii k \sigma_3+\frac{u_x}{2}\begin{pmatrix}
        0&1\\-1&0
    \end{pmatrix},\\
\label{Lax-V}
V&=-\frac{1}{4\ii k}\begin{pmatrix}
        \cos u&-\sin u\\-\sin u&-\cos u
    \end{pmatrix}.
\end{align}
\end{subequations}

The first step in the spectral analysis is to investigate the analyticity properties, in the complex spectral parameter $k$, of the eigenfunctions
associated with the Lax pair. Studying their analyticity, symmetry, and
boundedness in appropriate regions of the complex $k$-plane leads to precise
relations among the different eigenfunctions. These relations can be
formulated as a matrix Riemann--Hilbert factorization problem in the spectral
variable.

Moreover, the coefficient matrices $U$ and $V$ in \eqref{Lax-UV} are traceless.
Consequently, for any matrix solution $\Phi(x,t,k)$ of the Lax pair
\eqref{Lax} (constructed from two linearly independent vector solutions), we have
\[
\partial_x \det \Phi = 0, 
\qquad
\partial_t \det \Phi = 0.
\]
Hence $\det \Phi$ is independent of both $x$ and $t$, and therefore
is constant throughout the domain.

\subsection{Problem I}\label{subsec:2.1}

Assume that we are given a solution $u(x,t)$ of the Problem I \eqref{pr_1} such that $u(\cdot, t)\in \tilde H^{1,2}(0,\infty)$
for all $t\in[0,T]$. We study the analytic properties of the associated eigenfunctions of the Lax pair in the complex spectral plane and derive relations among them that can be formulated as a factorization Riemann–Hilbert problem.

\subsubsection{Eigenfunctions near $k=\infty$}\label{sec:2.1}

To have good control as $k\to\infty$, we rewrite the Lax pair \eqref{Lax} as
\begin{subequations}\label{Lax_infty}
\begin{align}\label{Lax-x_infty}
\tilde\Phi_x+\ii k \sigma_3\tilde\Phi&=U_\infty\tilde\Phi,\\
\label{Lax-t_infty}
\tilde\Phi_t+\frac{1}{4\ii k}\sigma_3\tilde\Phi&=V_\infty\tilde\Phi,
\end{align}
\end{subequations}
where the coefficients $U\equiv U(x,t,k)$ and $V\equiv V(x,t,k)$ are defined by
\begin{subequations}\label{Lax-UV_infty}
\begin{align}\label{Lax-U_infty}
U_\infty&=\frac{u_x}{2}\begin{pmatrix}
        0&1\\-1&0
    \end{pmatrix},\\
\label{Lax-V_infty}
V_\infty&=-\frac{1}{4\ii k}\begin{pmatrix}
        \cos u-1&-\sin u\\-\sin u&-\cos u+1
    \end{pmatrix}.
\end{align}
\end{subequations}

Define $Q(x,t,k)$ by 
\begin{subequations}\label{Qp}
\begin{equation}\label{Q}
Q(x,t,k):= \ii k p(x,t,k)\sigma_3, 
\end{equation}
with
\begin{equation}\label{p}
p(x,t,k):=x-\frac{t}{4 k^2}.
\end{equation}
\end{subequations}
Then $p$ satisfies
\[
p_x=1,\qquad
p_t=-\frac{1}{4 k^2},\qquad p(0,0,k)=0,
\]
and the Lax pair \eqref{Lax_infty} can be rewritten in the form
\begin{subequations}\label{phi-tilde}
   \begin{align}
&\tilde \Phi_{ x}+Q_x\tilde \Phi = U_\infty \tilde \Phi,\\
&\tilde \Phi_{ t} +Q_t\tilde \Phi = V_\infty \tilde \Phi.
\end{align} 
\end{subequations}
We will see that this form is convenient for controlling the behavior of the corresponding solutions as $k\to\infty$.

We construct the solutions $\tilde{\Phi}_{ j}$, $j=1,2,3$, of equation~\eqref{phi-tilde} via solutions $\Phi_{ j}$ of the associated Volterra-type integral equations. The functions $\Phi_{ j}$ are uniquely determined by their initial integration points and are related to $\tilde{\Phi}_{ j}$ through
\[
\tilde{\Phi}_{ j}(x,t,\lambda)
= \Phi_{ j}(x,t,\lambda)\, e^{-\ii kp(x,t,\lambda)\sigma_3}.
\]
They satisfy the integral equation
\begin{equation}\label{inteq_inf}
\Phi(x,t,k)=I+\int_{(x^*,t^*)}^{(x,t)}
	\eul^{\ii k(p(y,\tau,k)-p(x,t,k))\hat\sigma_3}( U_\infty\Phi \dd y+V_\infty\Phi \dd \tau)(y,\tau,k),
\end{equation}
where the initial point $(x^*,t^*)$ is chosen as $(0,T)$, $(0,0)$, or $(+\infty,t)$ for $j=1,2,3$, respectively. The corresponding integration contours are shown in Figure~\ref{fig:integration-paths}. Owing to the compatibility of the Lax pair, the value of the integral depends only on its endpoints.

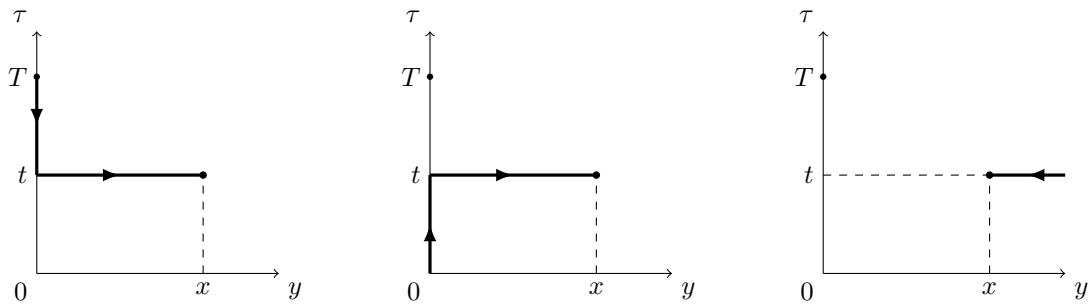
\begin{figure}[ht]
\centering
\begin{tikzpicture}[scale=1]
  \def\x{2.2}  
  \def\t{1.3}  
  \def\T{2.6}  
  \def\W{3.2}  

  \begin{scope}
    \draw[->] (0,0) -- (\W,0) node[below right] {$y$};
    \draw[->] (0,0) -- (0,\W) node[above left] {$\tau$};

    \node[below left] at (0,0) {$0$};
    \node[below]      at (\x,0) {$x$};
    \node[left]       at (0,\t) {$t$};

    \fill (0,\T) circle (1.2pt) node[left] {$T$};

    \draw[dashed] (\x,0) -- (\x,\t);

    \draw[very thick,postaction={decorate},
          decoration={markings, mark=at position 0.5 with {\arrow{latex}}}] 
          (0,\T) -- (0,\t);
    \draw[very thick,postaction={decorate},
          decoration={markings, mark=at position 0.5 with {\arrow{latex}}}] 
          (0,\t) -- (\x,\t);

    \fill (\x,\t) circle (1.4pt);
  \end{scope}

  \begin{scope}[xshift=5.2cm]
    \draw[->] (0,0) -- (\W,0) node[below right] {$y$};
    \draw[->] (0,0) -- (0,\W) node[above left] {$\tau$};

    \node[below left] at (0,0) {$0$};
    \node[below]      at (\x,0) {$x$};
    \node[left]       at (0,\t) {$t$};

    \fill (0,\T) circle (1.2pt) node[left] {$T$};

    \draw[dashed] (\x,0) -- (\x,\t);

    \draw[very thick,postaction={decorate},
          decoration={markings, mark=at position 0.5 with {\arrow{latex}}}] 
          (0,0) -- (0,\t);
    \draw[very thick,postaction={decorate},
          decoration={markings, mark=at position 0.5 with {\arrow{latex}}}] 
          (0,\t) -- (\x,\t);

    \fill (\x,\t) circle (1.4pt);
  \end{scope}

  \begin{scope}[xshift=10.4cm]
    \draw[->] (0,0) -- (\W,0) node[below right] {$y$};
    \draw[->] (0,0) -- (0,\W) node[above left] {$\tau$};

    \node[below left] at (0,0) {$0$};
    \node[below]      at (\x,0) {$x$};
    \node[left]       at (0,\t) {$t$};

    \fill (0,\T) circle (1.2pt) node[left] {$T$};

    \draw[dashed] (\x,0) -- (\x,\t);
    \draw[dashed] (0,\t) -- (\x,\t);

    \draw[very thick,postaction={decorate},
          decoration={markings, mark=at position 0.5 with {\arrow{latex}}}] 
          (\W,\t) -- (\x,\t);

    \fill (\x,\t) circle (1.4pt);
  \end{scope}
\end{tikzpicture}

\caption{Paths of integration for $\Phi_{1}, \Phi_{2},$ and $\Phi_{3}$
($\Phi_{01}, \Phi_{02},$ and $\Phi_{03}$ ).}
\label{fig:integration-paths}
\end{figure}

More precisely,
\begin{equation}\label{inteq_inf1}
    \begin{aligned}
\Phi_{ 1}(x,t,\lambda)=&I+
\int_{0}^{x}
	\eul^{-\ii k(x-y)\hat\sigma_3}( U_\infty\Phi_{1})(y,t,k) \dd y  - \eul^{-\ii kx\hat\sigma_3}
\int_{t}^{T}
	\eul^{ \frac{\tau-t}{4\ii k}\hat\sigma_3}( V_\infty\Phi_{1})(0,\tau,k) \dd \tau,
\end{aligned}
\end{equation}
    
\begin{equation}\label{inteq_inf2}
    \begin{aligned}
\Phi_{ 2}(x,t,k)=&I+
\int_{0}^{x}
	\eul^{-\ii k(x-y)\hat\sigma_3}( U_\infty\Phi_{2})(y,t,k) \dd y  + \eul^{-\ii kx\hat\sigma_3}
\int_{0}^{t}
	\eul^{\frac{\tau-t}{4\ii k}\hat\sigma_3}( V_\infty\Phi_{2})(0,\tau,k) \dd \tau,
\end{aligned}
\end{equation}

    \begin{equation}\label{inteq_inf3}
\Phi_{ 3}(x,t,k)=I-
\int_{x}^{+\infty}
	\eul^{\ii k(y-x)\hat\sigma_3}( U_\infty\Phi_{3})(y,t,k) \dd y .
\end{equation}

The domains where the exponentials are bounded in the complex $k$-plane are separated by the real axis, see 
 Figure \ref{fig:sign}. This partition of the $k$-plane is determined by the signature table of $\Im k$ and $\Im \frac{1}{k}$.

\begin{figure}[ht]
    \centering
\begin{tikzpicture}[scale=1.2]

        \node at (0,-1) {$\mathrm{sign}\Im k$};
    
    \draw[thick] (-2.5,0) -- (2.5,0);
    
    \node at (0,0.5) {$+$};
    \node at (0,-0.5) {$-$};

    \node at (7,-1) {$\mathrm{sign}\Im \frac{1}{k}$};
    
    \draw[thick] (4.5,0) -- (9.5,0);
    
    \node at (7,0.5) {$-$};
    \node at (7,-0.5) {$+$};
\end{tikzpicture}
    \caption{Signature tables for $p(x,t,k)$  }
    \label{fig:sign}
\end{figure}

Let $A^{(1)}$ and $A^{(2)}$ denote the first and second columns of a $2\times 2$ matrix $A = \bigl( A^{(1)}\ \ A^{(2)} \bigr)$ and notice that if $f\in \tilde H^{1,2}(0,\infty)$, then $f_x,~\sin f,~1-\cos f\in L^1(0,\infty)$. Since we assumed that $u(\cdot, t)\in \tilde H^{1,2}(0,\infty)$, it follows that the Neumann series expansions of \eqref{inteq_inf1} -- \eqref{inteq_inf3} converge in the corresponding half-planes of $k$ and yields the following properties of $\Phi_{0 i}^{(j)}(x,t,k)$:

\begin{enumerate}
    \item $\Phi_{ 1}^{(1)}(x,t,k)$ is analytic in $\mathbb{C}\setminus \{0\}$ and $\Phi_{1}^{(1)}(x,t,k)=\begin{pmatrix}
        1\\0
    \end{pmatrix}+O(\frac{1}{k})$ as $k\to \infty$ in $\mathbb{C}^+$. Moreover, $\Phi_{1}^{(1)}(0,t,k)=\begin{pmatrix}
        1\\0
    \end{pmatrix}+O(\frac{1}{k})$ as $k\to \infty$ in $\mathbb{C}$.

    \item $\Phi_{ 1}^{(2)}(x,t,k)$ is analytic in $\mathbb{C}\setminus \{0\}$ and $\Phi_{ 1}^{(2)}(x,t,k)=\begin{pmatrix}
        0\\1
    \end{pmatrix}+O(\frac{1}{k})$ as $k\to \infty$ in $\mathbb{C}^-$. Moreover, $\Phi_{ 1}^{(2)}(0,t,k)=\begin{pmatrix}
        0\\1
    \end{pmatrix}+O(\frac{1}{k})$ as $k\to \infty$ in $\mathbb{C}$.

     \item $\Phi_{ 2}^{(1)}(x,t,k)$ is analytic in $\mathbb{C}\setminus\{0\}$ and $\Phi_{ 2}^{(1)}(x,t,k)=\begin{pmatrix}
        1\\0
    \end{pmatrix}+O(\frac{1}{k})$ as $k\to\infty$ in $\mathbb{C}^+$.

    \item $\Phi_{ 2}^{(2)}(x,t,k)$ is analytic in $\mathbb{C}\setminus\{0\}$ and $\Phi_{ 2}^{(2)}(x,t,k)=\begin{pmatrix}
        0\\1
    \end{pmatrix}+O(\frac{1}{k})$ as  $k\to\infty$ in $\mathbb{C}^-$.

      \item $\Phi_{ 3}^{(1)}(x,t,k)$ is analytic in $\mathbb{C}^-$ and continuous up to real line.   Moreover, $\Phi_{ 3}^{(1)}(x,t,k)$=$\begin{pmatrix}
        1\\0
    \end{pmatrix}+O(\frac{1}{k})$ as $k\to \infty$ in $\mathbb{C}^-$.

    \item $\Phi_{ 3}^{(2)}(x,t,k)$ is analytic in $\mathbb{C}^+$ and continuous up to real line. Moreover, $\Phi_{3}^{(2)}(x,t,k)=\begin{pmatrix}
        0\\1
    \end{pmatrix}+O(\frac{1}{k})$ as $k\to\infty$ in $\mathbb{C}^+$.
\end{enumerate}

Since the coefficient matrices in the Lax pair \eqref{Lax_infty} are traceless, it follows that
\begin{equation}\label{jost_det}
\det \Phi_{ j} \equiv 1, \qquad j=1,2,3.    
\end{equation}

Furthermore, the matrices $U_\infty(k)\equiv U_\infty(x,t,k)$ and 
$V_\infty(k)\equiv V_\infty(x,t,k)$ obey the symmetries
\begin{subequations}\label{sym-UV}
\begin{alignat}{3}\label{sym-U}
U_\infty(- k)&=\overline{U_\infty(\bar k)},&\qquad&
U_\infty(-k)=\sigma_2U_\infty(k)\sigma_2,\\
V_\infty(- k)&=\overline{V_\infty(\bar k)},&&
V_\infty(-k)=\sigma_2V_\infty(k)\sigma_2,
\end{alignat}
\end{subequations}
while the function $p(k)\equiv p(x,t,k)$ satisfies
\begin{equation}\label{sym-p}
p^*(k)=p(k)=p(-k).
\end{equation}
Consequently, the eigenfunctions $\Phi_{ j}$ inherit the same symmetries:
\begin{equation}\label{sym-Phi}
\Phi_{ j}(-k)=\overline{\Phi_{ j}(\bar k)},\qquad
\Phi_{ j}(-k)=\sigma_2\Phi_{ j}(k)\sigma_2.
\end{equation}

Since the functions $\tilde\Phi_{ j}$ solve both equations of \eqref{Lax_infty}, they are related (where defined) by matrices independent of $x$ and $t$. Hence,
\begin{subequations}\label{rel_inf}
    \begin{align}\label{rel_inf_1}
        &\Phi_{ 3}(x,t,k)=\Phi_{ 2}(x,t,k)\eul^{-\ii k p(x,t,k)\sigma_3}s(k)\eul^{\ii k p(x,t,k)\sigma_3},\qquad k\in\mathbb{R}\setminus\{0\}\\\label{rel_inf_2}
        &\Phi_{ 1}(x,t,k)=\Phi_{2}(x,t,k)\eul^{-\ii k p(x,t,k)\sigma_3}S(k)\eul^{\ii k p(x,t,k)\sigma_3}, \qquad k\in\mathbb{C}\setminus\{0\}.
    \end{align}
\end{subequations}
The symmetries \eqref{sym-Phi} imply that the matrices $s(k)$ and $S(k)$ admit the forms
       \begin{equation}
            \label{s}
s(k)=\Phi_{\infty 3}(0,0,k)=\begin{pmatrix}
        a^*(k) & b(k)\\
        -b^*( k) & a(k)
    \end{pmatrix}
    \end{equation}
and
           \begin{equation}
        \label{S}
        S(k)=\Phi_{\infty 1}(0,0,k)=\begin{pmatrix}
        A^*(k) & B(k)\\
        -B^*(k) & A(k)
    \end{pmatrix}\end{equation}
    with some $a(k)$, $b(k)$, $A(k)$, and  $B(k)$.

Relation \eqref{rel_inf_1} constitutes the scattering relation for the
$x$-equation \eqref{Lax-U_infty}, \eqref{Lax-U_infty} (with $t$ being a parameter). Likewise, relation \eqref{rel_inf_2} constitutes the scattering relation for the $t$-equation \eqref{Lax-V_infty},
\eqref{Lax-V_infty} (with $x$ being a parameter). 

\subsubsection{Eigenfunctions near $k=0$}\label{sec:2.2}

In order to have good control as $k \to 0$, we introduce 
\[
\tilde\Phi_0(x,t,k):= P(x,t)\tilde\Phi(x,t,k),
\]
where
\[
P(x,t):=\begin{pmatrix}
\cos \frac{u}{2} & -\sin \frac{u}{2}  \\
\sin \frac{u}{2} & \cos \frac{u}{2} \\
\end{pmatrix}.
\]
This transforms \eqref{Lax} into
\begin{subequations}\label{Lax_0}
\begin{align}\label{Lax_0_x}
    & \tilde\Phi_{0 x}+\ii k\sigma_3\tilde\Phi_0=U_0\tilde\Phi_0,\\\label{Lax_0_t}
   &\tilde\Phi_{0 t}+\frac{1}{4\ii k}\sigma_3\tilde\Phi_0=V_0 \tilde\Phi_0,
\end{align}
\end{subequations}
where
\begin{subequations}\label{Lax-UV_0}
\begin{align}\label{Lax-U_0}
U_0&=-\ii k \begin{pmatrix}
        \cos u-1&\sin u\\\sin u&-\cos u+1
    \end{pmatrix},\\
\label{Lax-V_0}
V_0&=\frac{ u_t}{2}\begin{pmatrix}
        0&-1\\1&0
    \end{pmatrix}.
\end{align}
\end{subequations}

We introduce the solutions $\tilde{\Phi}_{0j}(x,t,k)$, $j=1,2,3$, of 
\eqref{Lax_0} analogously to $\tilde{\Phi}_{ j}(x,t,k)$ by defining
\[
\Phi_{0}(x,t,k):=\tilde \Phi_{0}(x,t,k)\eul^{\ii k p(x,t,k)\sigma_3},
\]
where the functions $\Phi_{0j}(x,t,k)$ solve the associated integral equations
\begin{equation}\label{inteq_0}
\Phi_{0}(x,t,k)=I+\int_{(x^*,t^*)}^{(x,t)}
	\eul^{\ii k(p(y,\tau,k)-p(x,t,k))\hat\sigma_3}( U_0\Phi_0 \dd y+V_0\Phi_0 \dd \tau)(y,\tau,k).
\end{equation}
More precisely,
\begin{equation}\label{inteq_01}
    \begin{aligned}
\Phi_{0 1}(x,t,\lambda)=&I+
\int_{0}^{x}
	\eul^{-\ii k(x-y)\hat\sigma_3}( U_0\Phi_{01})(y,t,k) \dd y  - \eul^{-\ii kx\hat\sigma_3}
\int_{t}^{T}
	\eul^{ \frac{\tau-t}{4\ii k}\hat\sigma_3}( V_0\Phi_{01})(0,\tau,k) \dd \tau,
\end{aligned}
\end{equation}
    
\begin{equation}\label{inteq_02}
    \begin{aligned}
\Phi_{ 02}(x,t,k)=&I+
\int_{0}^{x}
	\eul^{-\ii k(x-y)\hat\sigma_3}( U_0\Phi_{02})(y,t,k) \dd y  + \eul^{-\ii kx\hat\sigma_3}
\int_{0}^{t}
	\eul^{\frac{\tau-t}{4\ii k}\hat\sigma_3}( V_0\Phi_{02})(0,\tau,k) \dd \tau,
\end{aligned}
\end{equation}

    \begin{equation}\label{inteq_03}
\Phi_{0 3}(x,t,k)=I-
\int_{x}^{+\infty}
	\eul^{\ii k(y-x)\hat\sigma_3}( U_0\Phi_{03})(y,t,k) \dd y .
\end{equation}

The following properties of $\Phi_{0 i}^{(j)}(x,t,k)$ follow from the Neumann series for \eqref{inteq_01}--\eqref{inteq_03}:

\begin{enumerate}
    \item $\Phi_{0 1}^{(1)}(x,t,k)$ is analytic in $\mathbb{C}\setminus \{0\}$ and $\Phi_{01}^{(1)}(x,t,k)=\begin{pmatrix}
        1\\0
    \end{pmatrix}+O(k)$ as $k\to 0$ in $\mathbb{C}^-$.

    \item $\Phi_{0 1}^{(2)}(x,t,k)$ is analytic in $\mathbb{C}\setminus \{0\}$ and $\Phi_{0 1}^{(2)}(x,t,k)=\begin{pmatrix}
        0\\1
    \end{pmatrix}+O(k)$ as $k\to 0$ in $\mathbb{C}^+$.

     \item $\Phi_{0 2}^{(1)}(x,t,k)$ is analytic in $\mathbb{C}\setminus\{0\}$ and $\Phi_{0 2}^{(1)}(x,t,k)=\begin{pmatrix}
        1\\0
    \end{pmatrix}+O(k)$ as $k\to0$ in $\mathbb{C}^+$.

    \item $\Phi_{0 2}^{(2)}(x,t,k)$ is analytic in $\mathbb{C}\setminus\{0\}$ and $\Phi_{0 2}^{(2)}(x,t,k)=\begin{pmatrix}
        0\\1
    \end{pmatrix}+O(k)$ as  $k\to0$ in $\mathbb{C}^-$.

      \item $\Phi_{0 3}^{(1)}(x,t,k)$ is analytic in $\mathbb{C}^-$  and continuous up to the boundary.
 Moreover, $\Phi_{0 3}^{(1)}(x,t,k)$= $\begin{pmatrix}
        1\\0
    \end{pmatrix}+O(k)$ as $k\to 0$ in $\mathbb{C}^-$.

    \item $\Phi_{0 3}^{(2)}(x,t,k)$ is analytic in $\mathbb{C}^+$  and continuous up to the boundary.
 Moreover, $\Phi_{03}^{(2)}(x,t,k)=\begin{pmatrix}
        0\\1
    \end{pmatrix}+O(k)$ as $k\to0$ in $\mathbb{C}^+$.
\end{enumerate}

Since the coefficient matrices in the Lax pair \eqref{Lax_0} are traceless, it follows that
\[
\det \Phi_{0 j} \equiv 1, \qquad j=1,2,3.
\]
Moreover, because the matrices $U_0(x,t,k)$, $ V_0(x,t,k)$ satisfy the symmetries \eqref{sym-UV}, the functions $\Phi_{0 j}(x,t,k)$ inherit the symmetries \eqref{sym-Phi}. 

Since the functions $\tilde\Phi_{0 j}$satisfy both equations of the Lax pair \eqref{Lax_0}, they are related (where defined) by matrices independent of $x$ and $t$; hence

\begin{subequations}\label{rel_0}
    \begin{align}\label{rel_0_1}
        &\Phi_{0 3}(x,t,k)=\Phi_{0 2}(x,t,k)\eul^{-\ii kp(x,t,k)\sigma_3}\tilde s(k)\eul^{\ii k p(x,t,k)\sigma_3},\qquad k\in\mathbb{R}\setminus\{0\}\\\label{rel_0_2}
        &\Phi_{0 1}(x,t,k)=\Phi_{0 2}(x,t,k)\eul^{-\ii k p(x,t,k)\sigma_3}\tilde S(k)\eul^{\ii k p(x,t,k)\sigma_3},\qquad k\in\mathbb{C}\setminus\{0\}.
    \end{align}
\end{subequations}

Due to the symmetries of $\Phi_{0 j}$,  $\tilde s(k)$ and $\tilde S(k)$ admit the forms 
 \begin{equation}
        \label{tils}
    \tilde s(k)=\Phi_{0 3}(0,0,k)=\begin{pmatrix}
        \tilde  a^*(k) & \tilde b(k)\\
       -\tilde  b^*(k) & \tilde a(k)
    \end{pmatrix}
    \end{equation}
and
 \begin{equation}
        \label{tilS}\tilde S(k)=\Phi_{0 1}(0,0,k)=\begin{pmatrix}
        \tilde A^*(k)  &\tilde  B(k)\\
        -\tilde B^*(k) &\tilde  A(k)
    \end{pmatrix}\end{equation}
    with some $\tilde a(k)$, $\tilde b(k)$, $\tilde A(k)$, and  $\tilde B(k)$.

\subsubsection{Relations between eigenfunctions of different Lax pairs}
Since $\Phi_0$ and $\Phi$ 
arise from the same system  \eqref{Lax}, they are related as follows:
\begin{subequations}\label{Phi_0_inf_rel}
\begin{equation}\label{Phi_0_inf_rel_}
 \Phi_{ j}(x,t,k)=P^{-1}(x,t)\Phi_{0 j}(x,t,k)\eul^{-\ii kp(x,t,k)\sigma_3}C_j(k)\eul^{\ii k p(x,t,k)\sigma_3}   
\end{equation}
with $C_j(k)$ given by
\begin{align}\label{Phi_0_inf_coeff}
    C_1&=\eul^{-\frac{\ii T}{4 k}\sigma_3}P(0,T)\eul^{\frac{\ii T}{4 k}\sigma_3},\\\label{Phi_0_inf_coeff2}
    C_2&=P(0,0),\\\label{Phi_0_inf_coeff3}
    C_3&=I.
\end{align}
\end{subequations}
Particularly,
\begin{equation}\label{Phi_0_inf_reduced_2}
    \Phi_{ 3}(x,t,k) =P^{-1}(x,t)\Phi_{0 3}(x,t,k).
\end{equation}
Furthermore, setting $(x,t)=(0,0)$ in \eqref{Phi_0_inf_rel} gives 
        \begin{align}\label{s_via_til_s}
          &s(k)=P^{-1}(0,0)\tilde s(k),\\\label{S_via_til_S}
          &S(k)=P^{-1}(0,0)\tilde S(k)C_1=P^{-1}(0,0)\tilde S(k)\eul^{-\frac{\ii T}{4 k}\sigma_3}P(0,T)\eul^{\frac{\ii T}{4 k}\sigma_3}
        \end{align}

\begin{remark}\label{rem:simpl}
In the case $u(0,0)=4\pi k$, we have $P^{-1}(0,0)=I$ 
 and thus \eqref{Phi_0_inf_rel_} for $j=2$
reduces to 
\begin{equation}\label{Phi_0_inf_reduced}
\Phi_{ 2}(x,t,k) =P^{-1}(x,t)\Phi_{0 2}(x,t,k).
\end{equation}
Consequently, in this case
\[
s(k)=\tilde s(k),
\]
or 
\[
a(k) = \tilde{a}(k),\quad
b(k) = \tilde{b}(k),
\]
which 
results in a significant simplification of the subsequent analysis
and constructions.
\end{remark}

In the general case, we set
\[
P^{-1}(0,0)=
\begin{pmatrix}
    \kappa_1^0 & \kappa_2^0 \\
    -\kappa_2^0 & \kappa_1^0
\end{pmatrix},
\]
where $\kappa_1^0$ and $\kappa_2^0$ are determined by
$u(0,0)$ through
\begin{equation}\label{kappas-x}
\kappa_1^0:=\cos \tfrac{u(0,0)}{2},
\qquad
\kappa_2^0:=\sin \tfrac{u(0,0)}{2}.  
\end{equation} 
With this notation, \eqref{s_via_til_s} reads as
\begin{equation}\label{tilde-a--a}
\begin{aligned}
&a(k)=-\kappa_2^0\tilde b(k)+\kappa_1^0\tilde a(k),\\
&b(k)=\kappa_1^0\tilde b(k)+\kappa_2^0\tilde a(k).
\end{aligned}
\end{equation}

Similarly, introducing \[
\kappa_1^T:=\cos \tfrac{u(0,T)}{2},
\qquad
\kappa_2^T:=\sin \tfrac{u(0,T)}{2},
\]
\eqref{S_via_til_S} reads as
\begin{equation}\label{tilde-A--A}
\begin{aligned}
&\tilde A(k)=e^{-\frac{\ii T}{2k}}
\,\kappa_2^T\big(\kappa_2^0 A^*(k)-\kappa_1^0 B^*(k)\big)
+
\kappa_1^T\big(\kappa_2^0 B(k)+\kappa_1^0 A(k)\big),\\
&\tilde B(k)=e^{-\frac{\ii T}{2k}}
\,\kappa_2^T\big(\kappa_1^0 A^*(k)+\kappa_2^0 B^*(k)\big)
+
\kappa_1^T\big(\kappa_1^0 B(k)-\kappa_2^0 A(k)\big),
\end{aligned}
\end{equation}
or, in the reverse order,
\begin{equation}\label{tilde-A--A_}
\begin{aligned}
    &A(k)
=\kappa_1^T\left(
\kappa_1^0\tilde A(k)
-\kappa_2^0\tilde B(k)\right)
+
e^{-\frac{iT}{2k}}\kappa_2^T
\left(\kappa_1^0\tilde B^*(k)+\kappa_2^0\tilde A^*(k)\right),\\
&B(k)
=\kappa_1^T\left(
\kappa_1^0\tilde B(k)
+\kappa_2^0\tilde A(k)\right)
+
e^{\frac{iT}{2k}} \kappa_2^T
\left(\kappa_2^0\tilde B^*(k)-\kappa_1^0\tilde A^*(k)\right).
\end{aligned}
\end{equation}

\subsubsection{The direct $x$-spectral problem
$\{ u_{0}(x) \} \mapsto \{ 
a(k), b(k),\tilde a(k),\tilde b(k) \}$}\label{subsec:2.1.4}

Given a function $u_{0}(x)\in H^{1,1}((0,\infty))+2\pi k$ (not necessarily arising as the initial value of a solution of the sG equation), we define the direct $x$-spectral mapping
\[
\{ u_{0x}(x) \} \longrightarrow \{ {a}(k), {b}(k) \}
\]
by considering 
  Equation \eqref{inteq_inf3} for $t=0$:
\begin{equation}\label{inteq_inf3_t_0}    
\Phi_{ 3}(x,0,k)=I-
\int_{x}^{+\infty}
	\eul^{\ii k (y-x)\hat\sigma_3}( U_\infty\Phi_{3})(y,0,k) \dd y,
\end{equation}
where $U_\infty$ is given by \eqref{Lax-U_infty} with $u_x(x,0)$ replaced by $u_{0x}(x)$.

Then ${a}(k)$ and ${b}(k)$ are defined by (cf. \eqref{s})
\[
\begin{pmatrix}
    {b}(k) \\ {a}(k)
\end{pmatrix} = \Phi_{\infty 3}^{(2)}(0,0,k)
\]
and the analysis of the Volterra integral equation \eqref{inteq_inf3_t_0} yields the following properties:
\begin{enumerate}
    \item $a(k)$ and $ b(k)$ are analytic in $\mathbb{C}^+$ and continuous up to the boundary except at $k=0$. Moreover, $ a(k)=1+O(\frac{1}{k})$ and $ b(k)=\frac{u_{0x}(0)}{4\ii k}+O(\frac{1}{k^2})$ as $k\to\infty$ in $\mathbb{C}^+$.

    \item Determinant relation:
        \begin{equation}\label{detrel_ab}
           a(k) a^*( k)+  b(k) b^*( k)=1,  \qquad k\in \mathbb R.
        \end{equation}

    \item Symmetries:

    \begin{equation}
        \label{sym_a_}
    \overline{a(\bar k)}=a(-k),\qquad
    \overline{b(\bar k)}=b(-k).
    \end{equation}
\end{enumerate}

Similarly, we define the direct $x$-spectral mapping
\[
\{\sin u_{0}(x),\cos u_{0}(x) \} \longrightarrow \{\tilde {a}(k),\tilde {b}(k) \}
\]
by considering Equation 
\eqref{inteq_03} for $t=0$:
\begin{equation}\label{inteq_03_t_0} 
\Phi_{0 3}(x,0,k)=I-
\int_{x}^{+\infty}
	\eul^{\ii k(y-x)\hat\sigma_3}( U_0\Phi_{03})(y,0,k) \dd y,\qquad x\ge 0
\end{equation}
where $U_0$ is given by \eqref{Lax-U_0} with $\sin u(x,0)$ and $\cos u(x,0)$ replaced by $\sin u_0(x)$ and $\cos u_0(x)$, respectively.
Then $\tilde{a}(k)$ and $\tilde{b}(k)$ defined by (cf. \eqref{tils})
\[
\begin{pmatrix}
    \tilde{b}(k) \\ \tilde{a}(k)
\end{pmatrix} = \Phi_{03}^{(2)}(0,0,k)
\]
have the following properties:

\begin{enumerate}
   
    \item $\tilde a(k)$ and $\tilde b(k)$ are analytic in $\mathbb{C}^+$ and continuous up to the boundary. Moreover, $\tilde a(k)=1+O(k)$ and $\tilde b(k)=O(k)$ as $k\to 0$  in $\mathbb{C}^+$.

    \item Determinant relation:
        \begin{equation}\label{detrel_tilab}
          \tilde a(k)\tilde a^*( k)+ \tilde b(k)\tilde b^*( k)=1, \qquad k\in \mathbb R;
        \end{equation}

    \item Symmetries: 
    \begin{equation}\label{sym_tila_}
    \overline{\tilde a(\bar k)}=\tilde a(-k),\qquad
    \overline{\tilde b(\bar k)}=\tilde b(-k).
\end{equation}

\end{enumerate}

Moreover, taking into account \eqref{tilde-a--a}, it follows that we can specify the behavior of $ a(k)$ and $ b(k)$ as $k\to 0$
   \begin{equation}\label{a_at_i}
       a(k)=\kappa_1^0+O(k), \qquad 
    b(k)=\kappa_2^0+O(k), \qquad k\in\mathbb{C}^+.
   \end{equation}

\begin{remark} The spectral functions $a(k)$ and $b(k)$ are determined solely by $u_{0x}(x)$ assuming that $u_{0x}(x)\in L^1(0,+\infty)$, whereas the spectral functions $\tilde a(k)$ and $\tilde b(k)$ are determined by $\sin u_0(x)$ and $\cos u_0(x)$. In particular, it is sufficient to know $u_0(x)$ modulo $2\pi$ in order to define the direct $x$-spectral problems.

Notice that the requirement $u_0(x)\in H^{2,2}((0,+\infty))+2\pi k$ in Problem I will provide additional properties of $a(k)$, $b(k)$,  $\tilde a(k)$, and $\tilde b(k)$ needed for a particular regularity of the RH problem with respect to its parameters (see Sections \ref{sec:6} and \ref{sec:7})

\end{remark}

\subsubsection{The direct $t$-spectral problem
$\{ v_{0}(t) \} \mapsto \{ 
A(k), B(k) \}$}\label{subsec:2.1.5}

Given a function $v_0(t)\in L^1_{loc}(0,T)$ we define the direct $t$-spectral mapping
\[
\{\sin v_0(t), \cos v_0(t) \} \longrightarrow \{ A(k), B(k) \}
\]
by considering Equation \eqref{inteq_inf1} for $x=0$
\begin{equation}\label{inteq_inf1_x_0}    
\Phi_{ 1}(0,t,k)=I-\int_{t}^{T}
	\eul^{ \frac{\tau-t}{4\ii k}\hat\sigma_3}( V_\infty\Phi_{1})(0,\tau,k) \dd \tau,
\end{equation}
where $V_\infty$ is given by \eqref{Lax-V_infty} with $\sin u(0,t)$ and $\cos u(0,t)$ replaced by $\sin v_0(t)$ and $\cos v_0(t)$, respectively. 

Then $A(k)$ and $B(k)$ are defined by (cf. \eqref{S})
\[
\begin{pmatrix}
    {A}(k) \\ B(k)
\end{pmatrix} = \Phi_{\infty 1}^{(2)}(0,0,k)
\]
and satisfy the following properties:

\begin{enumerate}
    \item $ A(k)$ and $ B(k)$ are analytic in $\mathbb{C}\setminus\{0\}$. Moreover, $ A(k)=1+O(\frac{1}{k})$ and $ B(k)=O(\frac{1}{k})$ as $k\to\infty$.

    \item Determinant relation:
        \begin{equation}\label{detrel_AB}
           A(k) A^*( k)+  B(k) B^*(k)=1  
        \end{equation}

    \item Symmetries:
    \begin{equation}\label{sym_A_}
    \overline{A(\bar k)}=A(-k),\qquad
    \overline{B(\bar k)}=B(-k).
\end{equation}

\end{enumerate}

Similarly, we define the direct $t$-spectral mapping
\[
\{ v_{0t}(t) \} \longrightarrow \{\tilde A(k),\tilde B(k) \}
\]
by considering 
\eqref{inteq_01} for $x=0$:
\begin{equation}\label{inteq_01_x_0}    
\Phi_{0 1}(0,t,k)=I-
\int_{t}^{T}
	\eul^{-\frac{\ii}{4k}(\tau-t)\hat\sigma_3}( V_0\Phi_{01})(0,\tau,k ) \dd \tau,
\end{equation}
where $V_0$ is given by \eqref{Lax-V_0} with $u_t(0,t)$ replaced by $v_{0t}(t)$.
Then $\tilde{A}(k)$ and $\tilde{B}(k)$ defined by (cf. \eqref{tils})
\[
\begin{pmatrix}
    \tilde{B}(k) \\ \tilde{A}(k)
\end{pmatrix} = \Phi_{01}^{(2)}(0,0,k)
\]
have the following properties:

\begin{enumerate} 
    \item $\tilde A(k)$ and $\tilde B(k)$ are analytic in $\mathbb{C}\setminus\{0\}$. Moreover, $\tilde A(k)=1+O(k)+O\left(k\eul^{-\frac{\ii T}{2k}}
\right)$ and $\tilde B(k)=O(k)+O\left(k\eul^{-\frac{\ii T}{2k}}\right)$ as $k\to 0$.

    \item Determinant relation:
        \begin{equation}\label{detrel_tilAB}
          \tilde A(k)\tilde A^*(k)+ \tilde B(k)\tilde B^*(k)=1  
        \end{equation}

    \item Symmetries:
    \begin{equation}\label{sym_tilA_}
    \overline{\tilde A(\bar k)}=\tilde A(-k),\qquad 
    \overline{\tilde B(\bar k)}=\tilde B(-k).
\end{equation}

    \item Behavior at $\infty$ (follows from \eqref{tilde-A--A}):

\begin{equation}   \label{tilS_at_infty}
   \tilde A(k)=\kappa_1^T\kappa_1^0+\kappa_2^T\kappa_2^0+O\left(\frac{1}{k}\right),\quad \tilde B(k)=-\kappa_1^T\kappa_2^0+\kappa_2^T\kappa_1^0+O\left(\frac{1}{k}\right),\quad k\to\infty,\quad k\in\mathbb{C}^+.
\end{equation}

\end{enumerate}

Furthermore,  \eqref{tilde-A--A_} allows us to specify the behavior of $A(k)$ and $B(k)$ as $k\to 0$, namely,
     \begin{equation}\label{A_B_at_0}
     A(k)=\kappa_1^T\kappa_1^0 + \kappa_2^T\kappa_2^0 \eul^{-\frac{\ii T}{2k}}+O(k)+O\left(k\eul^{-\frac{\ii T}{2k}}
\right), \qquad B(k)=\kappa_1^T\kappa_2^0 + \kappa_2^T\kappa_1^0 \eul^{-\frac{\ii T}{2k}}+O(k)+O\left(k\eul^{-\frac{\ii T}{2k}}
\right).
 \end{equation}

\subsection{Problem II}\label{subsec:2.2}

As in the case of Problem I, assuming that we are given a solution $u(x,t)$ of Problem II \eqref{pr_2}
such that $u(\cdot, t)\in \tilde H^{1,2}(-\infty,0)$
for all $t\in[0,T]$, we study the analytic properties of the associated eigenfunctions.

We define the eigenfunctions $\Phi$ as in \eqref{inteq_inf}. However, instead of $\Phi_{ 3}$ we introduce $\check \Phi_{3}$, for which the initial point
  $(x^*,t^*)$ is chosen to be $(-\infty,t)$. The integration paths are those shown in Figure \ref{fig:integration-paths_}.

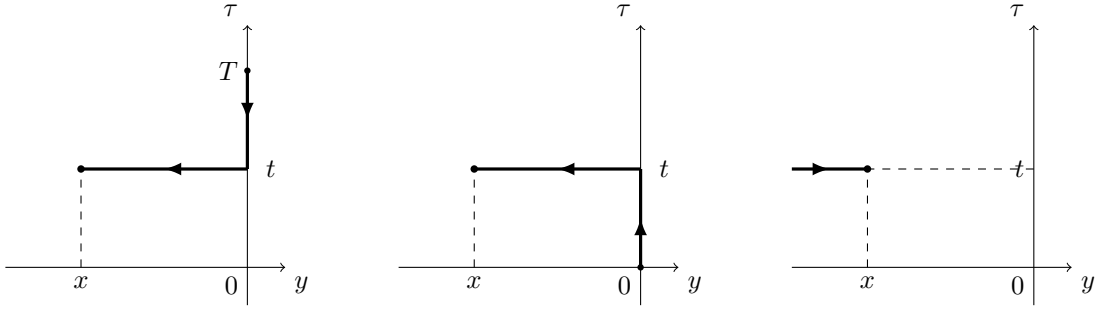
\begin{figure}[ht]
\centering
\begin{tikzpicture}[scale=1]
  \def\x{2.2}  
  \def\t{1.3}  
  \def\T{2.6}  
  \def\W{3.2}  

  \begin{scope}
    \draw[->] (-\W,0) -- (0.5,0) node[below right] {$y$};
    \draw[->] (0,-0.5) -- (0,\W) node[above left] {$\tau$};

    \node[below left] at (0,0) {$0$};
    \node[below]      at (-\x,0) {$x$};
    \node[left]       at (0.5,\t) {$t$};

    \fill (0,\T) circle (1.2pt) node[left] {$T$};

    \draw[dashed] (-\x,0) -- (-\x,\t);

    \draw[very thick,postaction={decorate},
          decoration={markings, mark=at position 0.5 with {\arrow{latex}}}] 
          (0,\T) -- (0,\t);
    \draw[very thick,postaction={decorate},
          decoration={markings, mark=at position 0.5 with {\arrow{latex}}}] 
          (0,\t) -- (-\x,\t);

    \fill (-\x,\t) circle (1.4pt);
  \end{scope}

  \begin{scope}[xshift=5.2cm]
    \draw[->] (-\W,0) -- (0.5,0) node[below right] {$y$};
    \draw[->] (0,-0.5) -- (0,\W) node[above left] {$\tau$};

    \node[below left] at (0,0) {$0$};
    \node[below]      at (-\x,0) {$x$};
    \node[left]       at (0.5,\t) {$t$};

    \fill (0,0) circle (1.2pt);

    \draw[dashed] (-\x,0) -- (-\x,\t);

    \draw[very thick,postaction={decorate},
          decoration={markings, mark=at position 0.5 with {\arrow{latex}}}] 
          (0,0) -- (0,\t);
    \draw[very thick,postaction={decorate},
          decoration={markings, mark=at position 0.5 with {\arrow{latex}}}] 
          (0,\t) -- (-\x,\t);

    \fill (-\x,\t) circle (1.4pt);
  \end{scope}

  \begin{scope}[xshift=10.4cm]
    \draw[->] (-\W,0) -- (0.5,0) node[below right] {$y$};
    \draw[->] (0,-0.5) -- (0,\W) node[above left] {$\tau$};

    \node[below left] at (0,0) {$0$};
    \node[below]      at (-\x,0) {$x$};
    \node[left]       at (0,\t) {$t$};

    \draw[dashed] (-\x,0) -- (-\x,\t);
    \draw[dashed] (0,\t) -- (-\x,\t);

    \draw[very thick,postaction={decorate},
          decoration={markings, mark=at position 0.5 with {\arrow{latex}}}] 
          (-\W,\t) -- (-\x,\t);

    \fill (-\x,\t) circle (1.4pt);
  \end{scope}
\end{tikzpicture}
\caption{Paths of integration for $\Phi_{1}, \Phi_{2},$ and $\check\Phi_{3}$
($\Phi_{01}, \Phi_{02},$ and $\check\Phi_{03}$ ).}
\label{fig:integration-paths_}
\end{figure}

More precisely,
\begin{equation}\label{inteq_inf1_}
    \begin{aligned}
\Phi_{ 1}(x,t,\lambda)=I&-
\int_{x}^{0}
	\eul^{\ii k(y-x)\hat\sigma_3}( U_\infty\Phi_{1})(y,t,k) \dd y   -\eul^{-\ii kx\hat\sigma_3}
\int_{t}^{T}
	\eul^{ \frac{\tau-t}{4\ii k }\hat\sigma_3}( V_\infty\Phi_{1})(0,\tau,k) \dd \tau,
\end{aligned}
\end{equation}
    
\begin{equation}\label{inteq_inf2_}
    \begin{aligned}
\Phi_{ 2}(x,t,k)=I&-
\int_{x}^{0}
	\eul^{\ii k(y-x)\hat\sigma_3}( U_\infty\Phi_{2})(y,t,k) \dd y  +\eul^{-\ii kx\hat\sigma_3}
\int_{0}^{t}
	\eul^{ \frac{\tau-t}{4\ii k }\hat\sigma_3}( V_\infty\Phi_{2})(0,\tau,k) \dd \tau,
\end{aligned}
\end{equation}

    \begin{equation}\label{inteq_inf3_}
\check\Phi_{3}(x,t,k)=I+
\int^{x}_{-\infty}
	\eul^{-\ii k(x-y)\hat\sigma_3}( U_\infty\check\Phi_{3})(y,t,k) \dd y.
\end{equation}
They have the following properties:

\begin{enumerate}
    \item $\Phi_{ 1}^{(1)}(x,t,k)$ is analytic in $\mathbb{C}\setminus \{0\}$ and $\Phi_{ 1}^{(1)}(x,t,k)$=$\begin{pmatrix}
        1\\0
    \end{pmatrix}+O(\frac{1}{k})$ as $k\to\infty$ in $\mathbb{C}^-$. Moreover, $\Phi_{ 1}^{(1)}(0,t,k)$=$\begin{pmatrix}
        1\\0
    \end{pmatrix}+O(\frac{1}{k})$ as $k\to\infty$ in $\mathbb{C}$.

    \item $\Phi_{ 1}^{(2)}(x,t,k)$ is analytic in $\mathbb{C}\setminus \{0\}$ and $\Phi_{ 1}^{(2)}(x,t,k)$=$\begin{pmatrix}
        0\\1
    \end{pmatrix}+O(\frac{1}{k})$ as $k\to\infty$ in $\mathbb{C}^+$. Moreover, $\Phi_{ 1}^{(2)}(0,t,k)$=$\begin{pmatrix}
        0\\1
    \end{pmatrix}+O(\frac{1}{k})$ as $k\to\infty$ in $\mathbb{C}$.

     \item $\Phi_{ 2}^{(1)}(x,t,k)$ is analytic in $\mathbb{C}\setminus\{0\}$ and $\Phi_{\infty 2}^{(1)}(x,t,k)$=$\begin{pmatrix}
        1\\0
    \end{pmatrix}+O(\frac{1}{k})$ as $k\to\infty$ in $\mathbb{C}^-$.

    \item $\Phi_{ 2}^{(2)}(x,t,k)$ is analytic in $\mathbb{C}\setminus\{0\}$ and $\Phi_{ 2}^{(2)}(x,t,k)$=$\begin{pmatrix}
        0\\1
    \end{pmatrix}+O(\frac{1}{k})$ as $k\to\infty$ in $\mathbb{C}^+$.

      \item $\check\Phi_{3}^{(1)}(x,t,k)$ is analytic in $\mathbb{C}^+$ and continuous up to real line. Moreover, $\check\Phi_{3}^{(1)}(x,t,k)$=$\begin{pmatrix}
        1\\0
    \end{pmatrix}+O(\frac{1}{k})$ as $k\to\infty$ in $\mathbb{C}^+$.

    \item $\check\Phi_{3}^{(2)}(x,t,k)$ is analytic in $\mathbb{C}^-$ and continuous up to real line. Moreover,  $\check\Phi_{3}^{(2)}(x,t,k)$=$\begin{pmatrix}
        0\\1
    \end{pmatrix}+O(\frac{1}{k})$ as $k\to\infty$ in $\mathbb{C}^-$.       
\end{enumerate}

Similarly, we define the eigenfunctions $\Phi_0$ as the solutions of the following integral equations:

\begin{equation}\label{inteq_01_}
        \begin{aligned}
\Phi_{0 1}(x,t,k)=I-
\int_{x}^{0}
	\eul^{-\ii k(x-y)\hat\sigma_3}( U_0\Phi_{01})(y,t,k) \dd y   -\eul^{-\ii k x\hat\sigma_3}
\int_{t}^{T}
	\eul^{\frac{i}{4 k}(t-\tau)\hat\sigma_3}( V_0\Phi_{01})(0,\tau,k) \dd \tau,
\end{aligned}
\end{equation}

  \begin{equation} \label{inteq_02_}
    \begin{aligned}
\Phi_{02}(x,t,k)=I-
\int_{x}^{0}
	\eul^{-\ii k(x-y)\hat\sigma_3}( U_0\Phi_{02})(y,t,k) \dd y   +\eul^{-\ii k x\hat\sigma_3}
\int_{0}^{t}
	\eul^{\frac{i}{4 k}(t-\tau)\hat\sigma_3}( V_0\Phi_{02})(0,\tau,k) \dd \tau,
\end{aligned}
\end{equation}

    \begin{equation}\label{inteq_03_}
\check\Phi_{03}(x,t,k)=I+
\int_{-\infty}^{x}
	\eul^{-\ii k (x-y)\hat\sigma_3}( U_0\check\Phi_{03})(y,t,k) \dd y,
    \end{equation}
and using Neumann series, we obtain the following properties of $\Phi_{0 i}^{(j)}(x,t,k)$:

\begin{enumerate}
    \item $\Phi_{0 1}^{(1)}(x,t,k)$ is analytic in $\mathbb{C}\setminus \{0\}$ and $\Phi_{01}^{(1)}(x,t,k)=\begin{pmatrix}
        1\\0
    \end{pmatrix}+O(k)$ as $k\to 0$ in $\mathbb{C}^-$.

    \item $\Phi_{0 1}^{(2)}(x,t,k)$ is analytic in $\mathbb{C}\setminus \{0\}$ and  $\Phi_{0 1}^{(2)}(x,t,k)=\begin{pmatrix}
        0\\1
    \end{pmatrix}+O(k)$ as $k\to 0$ in $\mathbb{C}^+$.

     \item $\Phi_{0 2}^{(1)}(x,t,k)$ is analytic in $\mathbb{C}\setminus\{0\}$ and  $\Phi_{0 2}^{(1)}(x,t,k)=\begin{pmatrix}
        1\\0
    \end{pmatrix}+O(k)$ as $k\to0$ in $\mathbb{C}^+$.

    \item $\Phi_{0 2}^{(2)}(x,t,k)$ is analytic in $\mathbb{C}\setminus\{0\}$ and $\Phi_{0 2}^{(2)}(x,t,k)=\begin{pmatrix}
        0\\1
    \end{pmatrix}+O(k)$ as  $k\to0$ in $\mathbb{C}^-$.

      \item $\check\Phi_{0 3}^{(1)}(x,t,k)$ is analytic in $\mathbb{C}^+$  and continuous up to the boundary.
 Moreover, $\check\Phi_{0 3}^{(1)}(x,t,k)$=$\begin{pmatrix}
        1\\0
    \end{pmatrix}+O(k)$ as $k\to 0$ in $\mathbb{C}^+$.

    \item $\check\Phi_{0 3}^{(2)}(x,t,k)$ is analytic in $\mathbb{C}^-$  and continuous up to the boundary.
 Moreover, $\check\Phi_{03}^{(2)}(x,t,k)=\begin{pmatrix}
        0\\1
    \end{pmatrix}+O(k)$ as $k\to0$ in $\mathbb{C}^-$.
\end{enumerate}
The  properties \eqref{jost_det} and \eqref{sym-Phi} hold as well for $\Phi_{j}(x,t,k)$, $\Phi_{0 j}(x,t,k)$ for $j=1,2$, $\check\Phi_{3}(x,t,k)$ and $\check\Phi_{0 3}(x,t,k)$ .

We define the matrices $s(k)$, $S(k)$, $\tilde s(k)$, and $\tilde S(k)$ as in \eqref{s}, \eqref{S}, \eqref{tils} and \eqref{tilS}, respectively, with $\Phi_{\infty 3}$ ($\Phi_{0 3}$) replaced by  $\check\Phi_{3}$ ($\check\Phi_{0 3}$).

For Problem II, the spectral functions $a(k)$,  $b(k)$, $\tilde a(k)$, and $\tilde b(k)$ possess the following properties:
\begin{enumerate}[(i)]
    \item $\tilde a(k)$ and $\tilde b(k)$ are analytic in $\mathbb{C}^-$. Moreover, $\tilde a(0)=1$ and $\tilde b(0)=0$.

        \item $a(k)$ and $ b(k)$ are analytic in $\mathbb{C}^-$. Moreover, $ a(k)=1+O(\frac{1}{k})$ and $ b(k)=O(\frac{1}{k})$ as $k\to\infty$. Moreover,
   $ a(0)=\kappa_1^0$ and $ b(0)=\kappa_2^0$.
\end{enumerate}
The relations \eqref{detrel_ab}, \eqref{detrel_tilab}, \eqref{sym_a_}, and \eqref{sym_tila_} continue to hold, and the properties of 
$A(k)$ and 
$B(k)$ remain unchanged.

\section{Compatibility of initial and boundary values}\label{sec:3}

In this section, we discuss the relations among the spectral functions. 

\subsection{Problem I} First, we express $\Phi_{3}(0,T,k)$ in terms of $\Phi_{1}(0,T,k)$ using \eqref{rel_inf} and recall that $\Phi_{1}(0,T,k)=I$:
    \begin{equation}\label{for_GR}
        \Phi_{3}(0,T,k)=\eul^{-\ii kp(0,T,k)\sigma_3}S^{-1}(k)s(k)\eul^{\ii kp(0,T,k)\sigma_3}.
    \end{equation}
Taking into account that $\Phi_{3}(0,T,k)$ can be estimated from
\begin{equation*}
  \Phi_{3}(0,T,k)=I-
\int_{0}^{+\infty}
	\eul^{\ii ky\hat\sigma_3}( U_\infty\Phi_{3})(y,T,k) \dd y, 
\end{equation*}
we conclude that
$ \Phi^{(12)}_{\infty3}(0,T,k)=(\frac{1}{k})$ as $k\to\infty$ in $\mathbb{C}^+$.
Then, considering the $(12)$ entry  in \eqref{for_GR}, we arrive at the following asymptotic relationship among the spectral functions:
\begin{equation}\label{relations_Phi_inf3}
A(k)b(k)-B(k)a(k)=O(\frac{1}{k}),\quad k\to\infty, \quad k\in\mathbb{C}^+. 
   \end{equation}

Relationships of this type (known as  ``Global Relations'') arise within analyzing IBV problems by using the RH method, and often they are restrictions
on the spectral functions
that reflect the dependence between initial and boundary problems (see \cites{fokas2008unified,BS08,FIS}).
   With this respect, we notice that the relation \eqref{relations_Phi_inf3} provides no additional restriction on the behavior of the corresponding spectral functions, since it follows from the properties of $a(k)$, $b(k)$,  $A(k)$, and $B(k)$ listed in Section \ref{sec:2}.

Similarly, the relations  \eqref{rel_0} yield
    \begin{equation}\label{for_GR_2}
        \Phi_{03}(0,T,k)=\eul^{-\ii k p(0,T,k)\sigma_3}\tilde S^{-1}(k)\tilde s(k)\eul^{\ii kp(0,T,k)\sigma_3}.
    \end{equation}
In particular, its $(12)$ entry takes the form
\begin{equation*}
    \Phi^{(12)}_{03}(0,T,k)=(\tilde A(k)\tilde b(k)-\tilde B(k)\tilde a(k))\eul^{\frac{\ii}{2k}T},
\end{equation*}
and, recalling that
 $ \Phi^{(12)}_{03}(0,T,k)=O(k)$ as $k\to 0$, we arrive at the following asymptotic relation:
\begin{equation}\label{relations_Phi_03_i}
      (\tilde A(k)\tilde b(k)-\tilde B(k)\tilde a(k))\eul^{\frac{\ii}{2k}T}=O(k),\quad k\to0 ,\quad k\in\mathbb{C}^+.
    \end{equation}

 In contrast with \eqref{relations_Phi_inf3}, the asymptotic formula \eqref{relations_Phi_03_i} does present a restriction on the involved spectral functions, which reflects the compatibility of the initial and boundary data in the spectral terms.

\subsection{Problem II} By arguments analogous to those used for Problem I, we obtain
\begin{equation}\label{relations_Phi_inf3__}
(A(k)b(k)-B(k)a(k))=O(\frac{1}{k}),\quad k\to\infty, \quad k\in\mathbb{C}^- 
   \end{equation}
   and
\begin{equation}\label{relations_Phi_03_i__}
      (\tilde A(k)\tilde b(k)-\tilde B(k)\tilde a(k))\eul^{\frac{\ii}{2k}T}=O(k),\quad k\to0 ,\quad k\in\mathbb{C}^-.
    \end{equation}
In this case, \eqref{relations_Phi_inf3__} and \eqref{relations_Phi_03_i__} present any additional restrictions on the spectral functions: they follow from the properties 
 of $a(k)$, $b(k)$, $\tilde a(k)$, $\tilde b(k)$,  $A(k)$, $B(k)$, $\tilde A(k)$, and $\tilde B(k)$ listed in Section \ref{sec:2}.

\section{Inverse Spectral Mappings}\label{sec:4}

In this section, the inverse spectral mappings are formulated through the solutions of  associated Riemann–Hilbert problems, with jump matrices determined by the corresponding spectral functions.

Let $\epsilon>0$ be sufficiently small so that all 
zeros of $\tilde a(k)$, $\tilde a^*(k)$, $\tilde A(k)$, and $\tilde A^*(k)$ appearing in the denominators below lie outside the disks $\{|k|\leq\epsilon\}$.

\subsection{Problem I}\label{subsec:4.1}

\subsubsection{The inverse $x$-spectral mapping}\label{subsec:4.1.1}

The  inverse
 $x$-spectral mapping
 \[ \{ a(k),b(k) \}\longrightarrow \{ u_0(x) \}\]
is based on solving a Riemann–Hilbert problem whose jump matrix is built from the given spectral functions  $a$ and $b$. The construction of this problem is suggested by the properties of  $\Phi_{2}(x,0,k)$, $\Phi_{3}(x,0,k)$, $\Psi_{02}(x,0,k)$ and $\Psi_{03}(x,0,k)$ described in the direct spectral analysis.

first, notice that $\kappa_1^0$, $\kappa_2^0$ can be obtained by expanding $a(k)$ and $b(k)$ as $k\to 0$. Therefore, having $a(k)$ and $b(k)$, we can determine $\tilde a(k)$ and $\tilde b(k)$ via \eqref{s_via_til_s}.

\begin{figure}[ht]
    \centering

\begin{tikzpicture}[scale=2]

\draw[red] (0,0) circle (1);
\draw[thick] (-2,0) -- (2,0);

\draw[thick, postaction={decorate},
      decoration={markings, mark=at position 0.5 with {\arrow[scale=2]{>}}}]
      (-2,0) -- (-1,0);

\draw[thick, postaction={decorate},
      decoration={markings, mark=at position 0.7 with {\arrow[scale=2]{<}}}]
      (-1,0) -- (0,0);  

\draw[thick, postaction={decorate},
      decoration={markings, mark=at position 0.5 with {\arrow[scale=2]{>}}}]
      (1,0) -- (2,0);

\draw[thick, red, postaction={decorate},
      decoration={markings, mark=at position 0.25 with {\arrow[scale=2]{<}}}] 
      (0,0) circle (1);

\draw[thick, red, postaction={decorate},
      decoration={markings, mark=at position 0.75 with {\arrow[scale=2]{>}}}] 
      (0,0) circle (1);

\node at (0,0.1) {\textcolor{red}{$0$}};
\fill[red] (0,0) circle(0.02);

\node at (1.1,0.1) {\textcolor{blue}{$\epsilon$}};
\fill[blue] (1,0) circle(0.02);

\end{tikzpicture}

    \caption{The oriented contour for the Riemann--Hilbert problems \textbf{RH$^{(x)}$}, \textbf{RH$^{(t)}$}, and \textbf{RH$^{(xt)}$}}
    \label{fig:contour_RH_x}
\end{figure}
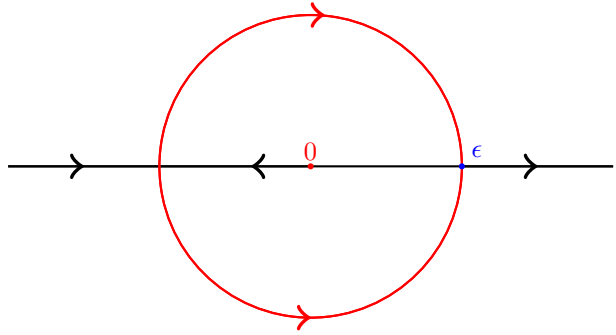

Given $u_{0}(x)$, consider the  matrix-valued function $M^{(x)}(x,k)$ defined in the domains separated by the contour shown in Figure \ref{fig:contour_RH_x}:

\begin{equation}\label{M_(x)}
M^{(x)}(x,k)=\begin{cases}

\left( \frac{\Phi_{ 2}^{(1)}(x,0,k)}{a(k)},\Phi_{ 3}^{(2)}(x,0,k)\right),\quad k\in\mathbb{C}^+\cap\{|k| >\epsilon\},\\

\left( \frac{\Psi_{0 2}^{(1)}(x,0,k)}{\tilde a(k)},\Psi_{0 3}^{(2)}(x,0,k)\right),\quad k\in\mathbb{C}^+\cap\{|k| <\epsilon\},\\

\left( \Phi_{ 3}^{(1)}(x,0,k),\frac{\Phi_{ 2}^{(2)}(x,0,k)}{a^*( k)}\right),\quad k\in\mathbb{C}^-\cap\{|k| >\epsilon\},\\

\left( \Psi_{0 3}^{(1)}(x,0,k),\frac{\Psi_{0 2}^{(2)}(x,0,k)}{\tilde a^*( k)}\right),\quad k\in\mathbb{C}^-\cap\{|k| <\epsilon\},
\end{cases}
\end{equation}
where the functions $\Psi_{0j}$
are defined by 
\begin{equation}\label{psi_0}
    \Psi_{0j}(x,t,k)\coloneqq P^{-1}(x,t)\Phi_{0j}(x,t,k). 
\end{equation}

Here, the functions $\Phi_{2}(x,0,k)$, $\Phi_{3}(x,0,k)$, $\Phi_{02}(x,0,k)$, $\Phi_{03}(x,0,k)$ are to be understood as defined by \eqref{inteq_inf2}, \eqref{inteq_inf3}, \eqref{inteq_02}, and \eqref{inteq_03},respectively,  for $t=0$, where $u(x,0)$ in $U_\infty(x,0)$ and $U_0(x,0)$ is replaced by $u_0(x)$.

Notice the following properties of $M^{(x)}(x,k)$:

\begin{enumerate}
    \item Jump relation across $\mathbb{R}\cup\{|k|=\epsilon\}$:
    \begin{subequations}
        \label{jump_M_(x)}
        \begin{equation}
           M_-^{(x)}(x,k)=M_+^{(x)}(x,k)J(x,k),\quad k\in \mathbb{R}\cup\{|k|=\epsilon\},
        \end{equation}
        where
         \begin{equation}
          J(x,k)=\eul^{-\ii k x\sigma_3} J_0(k)\eul^{\ii k x\sigma_3}
        \end{equation} 
         with
\begin{equation}\label{J0}
   J_0(k)=\begin{cases}
       \begin{pmatrix}
       1&-r^*(k)\\-r(k)&1+r(k)r^*(k)
   \end{pmatrix},\quad k\in\mathbb{R}\cap\{|k| >\epsilon\},\\
   \begin{pmatrix}
       1+\tilde r(k)\tilde r^*(k)&\tilde r^*(k)\\\tilde r(k)&1
   \end{pmatrix},\quad k\in\mathbb{R}\cap\{|k| <\epsilon\},\\
   \begin{pmatrix}
       1&0\\
       -\frac{\kappa_2^0}{a(k)\tilde a(k)}&1
   \end{pmatrix},\quad k\in\mathbb{C}^+\cap\{|k| =\epsilon\},\\
   \begin{pmatrix}
       1&-\frac{\kappa_2^0}{a^*(k)\tilde a^*(k)}\\
       0&1
   \end{pmatrix},\quad k\in\mathbb{C}^-\cap\{|k| =\epsilon\},
   \end{cases} 
\end{equation}
where $r(k)=\frac{b^*(k)}{a(k)}$ and $\tilde r(k)=\frac{\tilde b^*(k)}{\tilde a(k)}$. 
    \end{subequations}

\item Behavior at $\infty$:
\begin{subequations}\label{inf_M_(x)}
\begin{equation}\label{inf_M_exp(x)}
     M^{(x)}(x,k)=I+\frac{1}{4\ii k}M^{\infty}(x)+O(\frac{1}{k^2}),\quad k\to\infty
\end{equation}
with
\begin{equation}\label{inf_M_inf(x)}
M^{\infty}(x)=\begin{pmatrix}
    -\frac{1}{2}\int_x^\infty
u_{0x}^2(y)\dd y&u_{0x}(x)\\
    u_{0x}(x)&\frac{1}{2}\int_x^\infty
u_{0x}^2(y)\dd y\end{pmatrix}.
\end{equation}
\end{subequations}
\begin{proof} First, consider $k\in\mathbb{C}^+$.
Expanding $\Phi_{2}^{(1)}(x,0, k)$ as $ k\to\infty$ via the Neumann series and using integration by parts yields
\[
\Phi_{2}^{(1)}(x,0, k)=\begin{pmatrix}
    1\\0
\end{pmatrix}+\frac{1}{4\ii k}\begin{pmatrix}
    \frac{1}{2}\int_0^x
u_{0x}^2(y)\dd y\\u_{0x}(x)
\end{pmatrix}+O(\frac{1}{k^2}).
\]
Similarly, expanding $\Phi_{03}^{(2)}(x,0, k)$ as $ k\to \infty$, we obtain
\[
\Phi_{3}^{(2)}(x,0, k)=\begin{pmatrix}
    0\\1
\end{pmatrix}+\frac{1}{4\ii k}\begin{pmatrix}u_{0x}(x)
    \\\frac{1}{2}\int_x^\infty
u_{0x}^2(y)\dd y
\end{pmatrix}+O(\frac{1}{k^2}),
\]
and, in particular,
\[
a(k)=1+\frac{1}{8\ii k}\int_0^\infty
u_{0x}^2(y)\dd y+O(\frac{1}{k^2})
\]
Therefore, \eqref{inf_M_(x)} follows for $k\in\mathbb{C}^+$.
For $k\in\mathbb{C}^-$, this follows expanding $\Phi_{03}^{(1)}(x,0, k)$ and $\Phi_{2}^{(2)}(x,0, k)$. 
\end{proof}

\item $\det M^{(x)}(x,k)\equiv1$.

\item Symmetry properties:
\begin{equation}\label{sym-M_(x)}
M^{(x)}(- k)=\overline{M^{(x)}(\bar k)},\qquad M^{(x)}(-k)=\sigma_2M^{(x)}(k)\sigma_2.
\end{equation}

\item Residue properties. Let $\{k_j\}$ be zeros of $a(k)$ in $\mathbb{C}^+$. Assume that they are all simple. Then
\begin{subequations}\label{res-M_(x)}
\begin{align}\label{res-M+_(x)}
\Res_{k_j}M^{(x)(1)}(x,k)&=\frac{e^{2\ii k_j x}}{\dot a(k_j)b(k_j)}M^{(x)(2)}(x,k_j),\\
\label{res-M-_(x)}
\Res_{- k_j}M^{(x)(2)}(x,k)&=\frac{e^{2\ii k_j x}}{\dot a( k_j)b( k_j)}M^{(x)(1)}(x,- k_j)
\end{align}
\end{subequations}
 due to the symmetry  \eqref{sym_a_}.

\item Behavior at $0$:
\begin{equation}\label{i_beh-M_(x)}
M^{(x)}(x,k)=P_0^{-1}(x)+O(k), \quad k\to 0, 
\end{equation}
where 
\[
P_0(x)=\begin{pmatrix}
\cos \frac{u_0}{2} &\sin \frac{u_0}{2}  \\
-\sin \frac{u_0}{2} & \cos \frac{u_0}{2} \\
\end{pmatrix}.
\]

\end{enumerate}

\begin{remark}\label{rem:x-RH_reduced}
    In the case 
    $u_{0}(0)=4\pi k$, the 
    definition of $M^{(x)}(x,k)$ and the corresponding 
    jump conditions simplify,
    in view of \eqref{Phi_0_inf_reduced}, to the following:
     \begin{equation*}
           M_-^{(x)}(x,k)=M_+^{(x)}(x,k)J(x,k),\quad k\in \mathbb{R} ,
        \end{equation*}
        where
\[
M^{(x)}(x,k)=\begin{cases}
\left( \frac{\Phi_{\infty 2}^{(1)}(x,0,k)}{a(k)},\Phi_{\infty 3}^{(2)}(x,0,k)\right),\quad k\in\mathbb{C}^+,\\
\left( \Phi_{\infty 3}^{(1)}(x,0,k),\frac{\Phi_{\infty 2}^{(2)}(x,0,k)}{a^*( k)}\right),\quad k\in\mathbb{C}^-,
\end{cases}
\]
\begin{equation*}
          J(x,t,k)=\eul^{-\ii kx\sigma_3} J_0(k)\eul^{\ii kx\sigma_3}
        \end{equation*} 
        with  
\begin{equation*}
   J_0( k)=
      \begin{pmatrix}
       1&-r^*(k)\\-r(k)&1+r(k)r^*(k)
   \end{pmatrix},\quad k\in\mathbb{R}.
\end{equation*}
Here $\mathbb{R}$ is oriented from left to right.
    
\end{remark}

Now we notice that the properties of $M^{(x)}$ stated above suggest that $M^{(x)}(x,k)$  can be 
uniquely characterized 
as a solution of the RH problem (parametrized by $x$).

\textbf{The Riemann--Hilbert problem RH$^{(x)}$ for $M^{(x)}(x, k)$:}
Given $a( k)$, $b( k)$ for $ k\in\mathbb{C}^+$ and the set $\{k_j\}$ in $\mathbb{C}^+$, find a  piece-wise meromorphic $2\times 2$ matrix valued function $M^{(x)}(x, k)$ that satisfies the following conditions:

\begin{enumerate}
    \item Jump condition \eqref{jump_M_(x)}.

\item Behavior at $\infty$:
\begin{equation}\label{inf_hatM_(x)}
    M^{(x)}(x, k)=I+O\left(\frac{1}{ k}\right),\quad k\to\infty.
\end{equation}

\item Residue conditions \eqref{res-M_(x)}.
\end{enumerate}

\begin{proposition}\label{prop:Mx}
If a solution $M^{(x)}$ of \textbf{RH$^{(x)}$} exists, then:
    \begin{enumerate}
    \item it is unique;
    
        \item $\det  M^{(x)}\equiv1$;

        \item 
        \begin{equation}\label{sym-hatM_(x)}
 M^{(x)}( k)=\overline{ M^{(x)}(-\bar k)},\qquad  M^{(x)}(k)=\sigma_2 \overline{ M^{(x)}(\bar k)}\sigma_2.
\end{equation}

    \end{enumerate}
\end{proposition}

\begin{proof} \noindent (2) The conditions for $ M^{(x)}$ ensure that $\det M^{(x)}$ has neither a jump across $\mathbb{R}\cup\{|k|=\epsilon\}$ no singularities at $k_j$. As $\det M^{(x)}$ tends to $1$ as $k \to\infty$, Liouville’s theorem implies that $\det  M^{(x)}\equiv1$.

\noindent (1) The uniqueness of the solution follows by applying the Liouville type arguments to the ratio $ M_1^{(x)}( M_2^{(x)})^{-1}$ of two potential solutions, $ M_1^{(x)}$ and $ M_2^{(x)}$.

\noindent (3) The symmetries are inherited from the corresponding symmetries of the jump matrix and the uniqueness of the solution.

\end{proof}

The uniqueness of the solution to the Riemann–Hilbert problem \textbf{RH$^{(x)}$}, together with the properties of $ M^{(x)}(x,k)$, yields the following procedure for the inverse mapping
\[
    \{a(k),\, b(k)\} \longrightarrow \{u_0(x)\}
\]
for the $x$-problem:

\begin{enumerate}[Step 1.]
    \item Given $a(k)$ and $b(k)$, compute $\kappa_1^0$ and $\kappa_2^0$ through \eqref{a_at_i}, and determine $\tilde a(k)$ and $\tilde b(k)$ through \eqref{s_via_til_s}. With these data, formulate the Riemann–Hilbert problem \textbf{RH$^{(x)}$};

    \item Solve the RH problem constructed in the previous step and evaluate its solution $  M^{(x)}(x,k)$  at $ k=\infty$ (cf. \eqref{inf_M_(x)}):

    \begin{equation*}
       M^{(x)}(x, k)=I+\frac{1}{4\ii k}\begin{pmatrix}
           -\eta(x) & \xi(x)\\
           \xi(x) & \eta(x) 
       \end{pmatrix}+O\left(\frac{1}{k^2}\right).
    \end{equation*}

    \item Determine $u_{0x}(x)$  as follows:
    
\begin{equation}\label{m_0_via_RH_x}
    u_{0x}(x)=\xi(x). 
    \end{equation}

    \item  Having $u_{0x}(x)$, reconstruct $u_0(x)$ by
\begin{equation}\label{Fund_u0}
   u_0(x) = -\int_x^\infty u_{0x}(y)\dd y +2\pi k
\end{equation}

\end{enumerate}

\begin{remark}\label{rem:rec_u0}
    Alternatively, we can use \eqref{i_beh-M_(x)} in order to reconstruct $u_0(x)$.
More precisely, evaluate the solution $ M^{(x)}(x, k)$ of this RH problem at $ k=0$
\begin{equation}
      M^{(x)}(x, k)   =\begin{pmatrix}
            \alpha(x)& \beta(x)\\
           - \beta(x)& \alpha(x)
       \end{pmatrix}+O(k), 
    \end{equation}
and define $\sin \tfrac{u_0(x)}{2}$ and $\cos \tfrac{u_0(x)}{2}$ by
\[
\sin \tfrac{u_0(x)}{2}=\beta(x),\qquad \cos \tfrac{u_0(x)}{2}=\alpha(x).
\]
Then
\[
\sin u_0(x)=2\alpha(x)\beta(x),\qquad \cos u_0(x)=\alpha^2(x)-\beta^2(x)
\]
and
\[
u_{0x}=\cos u_0 (\sin u_0)_x-\sin u_0(\cos u_0)_x=-2\alpha\beta(\alpha+\beta)(\alpha_x+\beta_x).
\]

\end{remark}

\subsubsection{The Inverse $t$-Spectral Mapping}

The inverse $t$-spectral mapping
 \[ \{ A(k), B(k) \}\longrightarrow \{ v_0(t)  \}\]
can be expressed in terms of the solution of a Riemann–Hilbert problem with a jump matrix determined by $ A$ and $ B$. This RH problem is constructed using the relations among the eigenfunctions $\Phi_{ 1}(0,t,k)$, $\Phi_{ 2}(0,t, k)$, $\Psi_{01}(0,t, k)$, and $\Psi_{02}(0,t, k)$ arising in the direct spectral analysis.

Notice that having $A( k)$ and $ B( k)$, we can obtain $\kappa_i^0$, and $\kappa_i^T$, $i=1,2$, and determine $\tilde A( k)$ and $\tilde B(k)$ (see \eqref{A_B_at_0} and \eqref{tilde-A--A}).

We introduce the matrix-valued function $M^{(t)}(t,k)$ defined in the domains separated by the contour shown in Figure \ref{fig:contour_RH_x}:

\begin{equation}\label{M_(t)}
M^{(t)}(t, k)=\begin{cases}

\left( \Phi_{2}^{(1)}(0,t, k),\frac{\Phi_{ 1}^{(2)}(0,t, k)}{A( k)}\right),\quad  k\in\mathbb{C}^+\cap\{| k| >\epsilon\},\\

\left(\Psi_{0 2}^{(1)}(0,t, k),\frac{\Psi_{0 1}^{(2)}(0,t, k)}{\tilde A( k)}\right),\quad  k\in\mathbb{C}^+\cap\{| k| <\epsilon\},\\

\left(\frac{\Phi_{ 1}^{(1)}(0,t, k)}{A^*( k)},\Phi_{ 2}^{(2)}(0,t, k)\right),\quad  k\in\mathbb{C}^-\cap\{| k| >\epsilon\}\\

\left(\frac{\Psi_{0 1}^{(1)}(0,t, k)}{\tilde A^*( k)},\Psi_{0 2}^{(2)}(0,t, k)\right),\quad  k\in\mathbb{C}^-\cap\{| k| <\epsilon\},
\end{cases}
\end{equation}
where the functions $\Psi_{0j}$
are defined by \eqref{psi_0}.

Then the function $M^{(t)}(t, k)$ has the following properties:

\begin{enumerate}
    \item  Jump relation across $\mathbb{R}\cup\{|k|=\epsilon\}$
    \begin{subequations}
        \label{jump_M_(t)}
        \begin{equation}
           M_-^{(t)}(t, k)=M_+^{(t)}(t, k)J^{(t)}(t, k),\quad k\in\mathbb{R}\cup\{|k|=\epsilon\}
        \end{equation}
        where
         \begin{equation}
          J^{(t)}(t, k)=\eul^{\tfrac{\ii t}{4 k}\sigma_3} J^{(t)}_0( k)\eul^{-\tfrac{\ii t}{4 k}\sigma_3}
        \end{equation} 
        with 
\begin{equation}\label{J_0_jump_M_(t)}
   J^{(t)}_0( k)=\begin{cases}
       \begin{pmatrix}
          1+R( k)R^*( k)&-R( k)\\
          -R^*( k)&1
       \end{pmatrix},\quad  k\in\mathbb{R}\cap\{| k|>\epsilon\},\\
       \begin{pmatrix}
          1&\tilde R( k)\\
        \tilde R^*( k)&1+\tilde R( k)\tilde R^*( k)
       \end{pmatrix},\quad  k\in\mathbb{R}\cap\{| k|<\epsilon\},\\
       
       \begin{pmatrix}
        \kappa_1^0+R(k)\kappa_2^0&\tilde R(k)\kappa_1^0+\tilde R(k)R(k)\kappa_2^0+\kappa_2^0-R(k)\kappa_1^0\\
          -\kappa_2^0 & \kappa_1^0-\tilde R(k)\kappa_2^0
       \end{pmatrix},\quad  k\in\mathbb{C}^+\cap\{| k|=\epsilon\},\\

              \begin{pmatrix}
          \kappa_1^0+\tilde R^*(k)\kappa_2^0&-\kappa_2^0\\
          \tilde R^*(k)\kappa_1^0+\tilde R^*(k)R^*(k)\kappa_2^0+\kappa_2^0-R^*(k)\kappa_1^0&\kappa_1^0-R^*(k)\kappa_2^0
       \end{pmatrix},\quad  k\in\mathbb{C}^-\cap\{| k|=\epsilon\}
       
   \end{cases}
\end{equation}
and $R( k):=\frac{B( k)}{A( k)}$ and $\tilde R( k):=\frac{\tilde B( k)}{\tilde A( k)}$.
    \end{subequations}

\item Behavior at $\infty$:
\begin{equation}\label{inf_M_(t)}
     M^{(t)}(t, k)=I+O\left(\frac{1}{ k}\right),\quad k\to\infty.
\end{equation}

\item $\det M^{(t)}(t, k)\equiv1$

\item Symmetry properties:
\begin{equation}\label{sym-M_(t)}
M^{(t)}( k)=\overline{M^{(t)}(-\bar k)},\qquad M^{(t)}( k)=\sigma_2M^{(t)}( -k)\sigma_2
\end{equation}

\item Residue properties: Let $\{\mu_j\}$ be zeros of $A( k)$ in $\mathbb{C}^+$. Assume that they are all simple. Then, using the symmetries \eqref{sym_A_}, it follows that 
\begin{subequations}
    \label{res-M_(t)}
\begin{align}\label{res-M+_(t)}
\Res_{\mu_j}M^{(t)(2)}(t, k)&=\frac{B(\mu_j)}{\dot A(\mu_j)}e^{\frac{\ii t}{2\mu_j}}M^{(t)(1)}(t,\mu_j),\\
\label{res-M-_(t)}
\Res_{-\mu_j}M^{(t)(1)}(t, k)&=\frac{B(\mu_j)}{\dot A(\mu_j)}e^{\frac{\ii t}{2\mu_j}}M^{(t)(2)}(t,-\mu_j).
\end{align}
\end{subequations}

\item Assume that $v_0(t):=u(0,\cdot)\in W^{2,1}(0,T)$. Then expanding $\Phi_{02}^{(1)}(0,t, k)$ and $\Phi_{01}^{(2)}(0,t, k)$ at $ k=0$ in $\mathbb{C}^+$ via the Neumann series  using integration by parts yields
\[
\Phi_{02}^{(1)}(0,t, k)=\begin{pmatrix}
    1\\0
\end{pmatrix}+\ii k\begin{pmatrix} \frac{1}{2}\int_0^t
u_{t}^2(0,\tau)\dd \tau
    \\ -u_t(0,t)
\end{pmatrix}+O(k^2),
\]
and
\[
\Phi_{01}^{(2)}(0,t, k)=\begin{pmatrix}
    0\\1
\end{pmatrix}+\ii k\begin{pmatrix} -u_t(0,t)
    \\\frac{1}{2}\int_t^T
u_{t}^2(0,\tau)\dd \tau
\end{pmatrix}+O(k^2).
\]
In particular,
\[
A(k)=1+\frac{1}{2}\int_0^T
u_{t}^2(0,\tau)\dd \tau +O(k^2).
\]
Therefore,
\begin{subequations}\label{i_beh-M_(t)}
\begin{equation}\label{i_M_exp(x)}
     M^{(t)}(t,k)=P^{-1}(0,t)\left(I+\ii kM^{0}(t)+O(k^2)\right),\quad k\to 0, \quad k\in\mathbb{C}^+
\end{equation}
with
\begin{equation}\label{i_M_i(x)}
M^{0}(t)=\begin{pmatrix}
    \frac{1}{2}\int_0^t
u_{t}^2(0,\tau)\dd \tau&-u_t(0,t)\\
    -u_t(0,t)&-\frac{1}{2}\int_0^t
u_{t}^2(0,\tau)\dd \tau \end{pmatrix}.
\end{equation}
\end{subequations}

Similarly, expanding $\Phi_{01}^{(1)}(0,t, k)$ and $\Phi_{02}^{(2)}(0,t, k)$ at $ k=0$ in $\mathbb{C}^-$ yields \eqref{i_M_exp(x)}--\eqref{i_M_i(x)} in $\mathbb{C}^-$.
\end{enumerate}

\begin{remark}
     In the case 
    $u_{0}(0)=4\pi k$, the jump matrix simplifies to
\begin{equation}
   J^{(t)}_0( k)=\begin{cases}
       \begin{pmatrix}
          1+R( k)R^*( k)&-R( k)\\
          -R^*( k)&1
       \end{pmatrix},\quad  k\in\mathbb{R}\cap\{| k|>\epsilon\},\\
       \begin{pmatrix}
          1&\tilde R( k)\\
        \tilde R^*( k)&1+\tilde R( k)\tilde R^*( k)
       \end{pmatrix},\quad  k\in\mathbb{R}\cap\{| k|<\epsilon\},\\
       
       \begin{pmatrix}
        1&\tilde R(k)-R(k)\\
          0 & 1
       \end{pmatrix},\quad  k\in\mathbb{C}^+\cap\{| k|=\epsilon\},\\

              \begin{pmatrix}
          1&0\\
          \tilde R^*(k)-R^*(k)&1
       \end{pmatrix},\quad  k\in\mathbb{C}^-\cap\{| k|=\epsilon\}
       
   \end{cases}
\end{equation}
\end{remark}

Now we notice that the properties of $M^{(t)}$ stated above suggest that $M^{(t)}(t,k)$  can be 
uniquely characterized 
as a solution of the RH problem (parametrized by $t$).

\textbf{The Riemann--Hilbert problem RH $^{(t)}$ for $M^{(t)}(t, k)$:}
Given $A( k)$, $B( k)$ for $k\in\mathbb{C}\setminus\{0\}$ and the set $\{\mu_j\}$ in $\mathbb{C}^+$,
find  a piece-wise meromorphic $2\times 2$ matrix valued function $M^{(t)}(t, k)$ that satisfies the following conditions:

\begin{enumerate}
    \item Jump condition \eqref{jump_M_(t)}.

\item Behavior at $\infty$:
\begin{equation}\label{inf_hatM_(t)}
    M^{(t)}(t, k)=I+O\left(\frac{1}{ k}\right),\quad k\to\infty.
\end{equation}

\item Residue conditions \eqref{res-M_(t)}.
\end{enumerate}

Similarly to the $x$-problem, the following proposition holds:

\begin{proposition} If a solution $ M^{(t)}$ of \textbf{RH$^{(t)}$} exists, then 
    \begin{enumerate}
      \item  it is unique. 
      
        \item $\det M^{(t)}\equiv1$.

        \item 
        \begin{equation}\label{sym-hatM_(t)}
 M^{(t)}( k)=\overline{ M^{(t)}(-\bar k)},\qquad  M^{(t)}(k)=\sigma_2\overline{ M^{(t)}(\bar k)}\sigma_2.
\end{equation}
 
    \end{enumerate}
\end{proposition}

The uniqueness of the solution of the Riemann--Hilbert problem \textbf{RH$^{(t)}$} justifies the following procedure for the inverse mapping
\[
    \{ A(k),\, B(k) \} \longrightarrow \{v_0(t)\}
\]
for the $t$-problem:

\begin{enumerate}[Step 1.]
    \item Given $ A(k)$ and $ B(k)$, compute $\kappa_i^0$, $\kappa_i^T$, $i=1,2$, and determine $\tilde A(k)$ and $\tilde B(k)$ via  \eqref{S_via_til_S}. Having these data, formulate the Riemann–Hilbert problem \textbf{RH$^{(t)}$};

    \item Solve the RH problem constructed in the previous step
    and evaluate its solution $M^{(t)}(t, k)$ at $ k=0$:

    \begin{equation}\label{Mt_exp_0}
        M^{(t)}(t, k)=\begin{pmatrix}
            \alpha(t)& \beta(t)\\
           - \beta(t)& \alpha(t)
       \end{pmatrix}\left( I+\ii k\begin{pmatrix}
            f_1(t)& f_2(t)\\ f_2(t)&-f_1(t)
       \end{pmatrix}+O(k^2)\right). 
    \end{equation}

    \item Define $ v_{0t}(t)$  from this expansions in the following way (cf. \eqref{i_beh-M_(t)}):
    \begin{equation}
        \label{hat_v_j_via_RH_t}
      v_{0t}(t)=- f_2(t).
      \end{equation}
\item     Reconstruct $v_0(t)$ by
\begin{equation}\label{Fund_v0}
   v_0(t) = v_0(0)+ \int_0^t v_{0t}(\tau)\dd \tau .
\end{equation}

\end{enumerate}

\begin{remark}\label{rec:sin}
    We could also define $\sin \tfrac{v_0(t)}{2}$ and $\cos \tfrac{v_0(t)}{2}$ by
\[
\sin \tfrac{v_0(t)}{2}=\beta(t),\qquad \cos \tfrac{v_0(t)}{2}=\alpha(t).
\]
Then 
\[
\sin v_0(t)=2\alpha(t)\beta(t),\qquad \cos v_0(t)=\alpha^2(t)-\beta^2(t).
\]
and
\[
v_{0t}=\cos v_0 (\sin v_0)_t-\sin v_0(\cos v_0)_t=-2\alpha\beta(\alpha+\beta)(\alpha_t+\beta_t).
\]

\end{remark}

\subsection{Problem II}\label{subsec:4.2}

\subsubsection{The Inverse $x$-Spectral Mapping}
Taking into account the analytic properties of the eigenfunctions and the scattering coefficients, we introduce the matrix-valued function 
 $\check M^{(x)}(x,k)$ defined in the domains separated by the contour shown in Figure \ref{fig:contour_RH_x}:

\begin{equation}\label{M_(x)_}
\check M^{(x)}(x,k)=\begin{cases}

\left(\check\Phi_{3}^{(1)}(x,0,k), \frac{\Phi_{ 2}^{(2)}(x,0,k)}{a^*(k)}\right),\quad k\in\mathbb{C}^+\cap\{| k| >\epsilon\},\\

\left(\check\Psi_{0 3}^{(1)}(x,0,k), \frac{\Psi_{0 2}^{(2)}(x,0,k)}{\tilde a^*(k)}\right),\quad k\in\mathbb{C}^+\cap\{| k| <\epsilon\},\\

\left( \frac{\Psi_{0 2}^{(1)}(x,0,k)}{\tilde a( k)},\check\Psi_{0 3}^{(2)}(x,0,k)\right),\quad k\in\mathbb{C}^-\cap\{| k| <\epsilon\},\\

\left( \frac{\Phi_{ 2}^{(1)}(x,0,k)}{a( k)},\check\Phi_{3}^{(2)}(x,0,k)\right),\quad k\in\mathbb{C}^-\cap\{| k| >\epsilon\}.

\end{cases}
\end{equation}

Then the function $\check M^{(x)}(x,k)$ has the following properties:

\begin{enumerate}
    \item Jump relation across $\mathbb{R}\cup\{|k|=\epsilon\}$:
    \begin{subequations}
        \label{jump_M_(x)_left}
        \begin{equation}
           \check M_-^{(x)}(x,k)=\check M_+^{(x)}(x,k)\check J(x,k),\quad k\in \mathbb{R}\cup\{|k|=\epsilon\} 
        \end{equation}
        where
         \begin{equation}
          \check J(x,k)=\eul^{-\ii kx\sigma_3} \check J_0(k)\eul^{\ii kx\sigma_3}
        \end{equation} 
        with 
\begin{equation}\label{J_0__}
   \check J_0(k)=\begin{cases}
       \begin{pmatrix}
       1+r(k)r^*(k)&r^*(k)\\r(k)&1
   \end{pmatrix},\quad k\in\mathbb{R}\cap\{|k| >\epsilon\},\\
   \begin{pmatrix}
     1 &-\tilde r^*(k)\\-\tilde r(k)& 1+\tilde r(k)\tilde r^*(k)
   \end{pmatrix},\quad k\in\mathbb{R}\cap\{|k| <\epsilon\},\\
\begin{pmatrix}
       1&\frac{\kappa_2^0}{a^*(k)\tilde a^*(k)}\\
       0&1
   \end{pmatrix}
,\quad k\in\mathbb{C}^+\cap\{|k| =\epsilon\},\\
    \begin{pmatrix}
       1&0\\
    \frac{\kappa_2^0 }{a(k)\tilde a(k)}&1
   \end{pmatrix},\quad k\in\mathbb{C}^-\cap\{|k| =\epsilon\},
   \end{cases} 
\end{equation}
where $r(k)=\frac{b^*(k)}{a(k)}$ and $\tilde r(k)=\frac{\tilde b^*(k)}{\tilde a(k)}$. 
    \end{subequations}

\item Behavior at $\infty$:
\begin{subequations}\label{inf_M_(x)__}
\begin{equation}\label{inf_M_exp(x)__}
    \check  M^{(x)}(x,k)=I+\frac{1}{4\ii k}\check M^{\infty}(x)+O(\frac{1}{k^2}),\quad k\to\infty
\end{equation}
with
\begin{equation}\label{inf_M_inf(x)__}
\check M^{\infty}(x)=\begin{pmatrix}
    \frac{1}{2}\int_{-\infty}^x
u_{0x}^2(y)\dd y&u_{0x}(x)\\
    u_{0x}(x)&- \frac{1}{2}\int_{-\infty}^x
u_{0x}^2(y)\dd y\end{pmatrix}.
\end{equation}
\end{subequations}

\item $\det \check M^{(x)}(x,k)\equiv1$

\item Symmetry properties:
\begin{equation}\label{sym-M_(x)__}
\check M^{(x)}(- k)=\overline{\check M^{(x)}(\bar k)},\qquad \check M^{(x)}(-k)=\sigma_2\check M^{(x)}(k)\sigma_2
\end{equation}

\item Residue properties. Let $\{k_j\}$ be zeros of $a(k)$ in $\mathbb{C}^-$. Assume that they are all simple. Then
\begin{subequations}
    \label{res-M_(x)__}
\begin{align}\label{res-M+_(x)__}
\Res_{k_j}\check M^{(x)(1)}(x,k)&=\frac{e^{2\ii k_jx}}{\dot a(k_j)b(k_j)}\check M^{(x)(2)}(x,k_j),\\
\label{res-M-_(x)__}
\Res_{- k_j}\check M^{(x)(2)}(x,k)&=\frac{e^{2\ii k_j x}}{\dot a( k_j)b( k_j)}\check M^{(x)(1)}(x,- k_j).
\end{align}
\end{subequations}

\item Behavior at $0$:
\begin{equation}\label{i_beh-M_(x)__}
\check M^{(x)}(x,k)=\check P_0(x)+O(k) , \quad k\to 0,  
\end{equation}
where
\[
\check P_0(x)=\begin{pmatrix}
\cos \frac{u_0}{2} & \sin \frac{u_0}{2}  \\
-\sin \frac{u_0}{2} & \cos \frac{u_0}{2} \\
\end{pmatrix}.
\]

\end{enumerate}

    \begin{remark}\label{rem:x-RH_reduced_left}
    In the case 
    $u_{0}(0)=4\pi k$, the 
    definition of $M^{(x)}(x,k)$ and the corresponding 
    jump conditions simplify,
    in view of \eqref{Phi_0_inf_reduced}, to the following:
     \begin{equation*}
           \check M_-^{(x)}(x,k)=\check M_+^{(x)}(x,k)\check J(x,k),\quad k\in \mathbb{R} 
        \end{equation*}
        where
\[
\check M^{(x)}(x,k)=\begin{cases}
\left(\check\Phi_{3}^{(1)}(x,0,k), \frac{\Phi_{ 2}^{(2)}(x,0,k)}{a^*(k)}\right),\quad k\in\mathbb{C}^+,\\
\left( \frac{\Phi_{ 2}^{(1)}(x,0,k)}{a( k)},\check \Phi_{3}^{(2)}(x,0,k),\right),\quad k\in\mathbb{C}^-.
\end{cases}
\]
\begin{equation*}
         \check  J(x,k)=\eul^{-\ii kx\sigma_3} J_0^{(x)}(k)\eul^{\ii kx\sigma_3}
        \end{equation*} 
        with  
\begin{equation}\label{J_0___simpl}
   \check J^{(x)}_0( k)=
      \begin{pmatrix}
       1+r(k)r^*(k)&r^*(k)\\r(k)&1
   \end{pmatrix},\quad k\in\mathbb{R}.
\end{equation}
Here $\mathbb{R}$ is oriented from left to right.
\end{remark}

Now we notice that the properties of $\check M^{(x)}$ stated above suggest that $\check  M^{(x)}(x,k)$  can be 
uniquely characterized 
as a solution of the RH problem (parametrized by $x$).

\textbf{The Riemann--Hilbert problem $\widecheck{\text{RH}}^{(x)}$ for $\check M^{(x)}(x, k)$:}
Given $a( k)$, $b( k)$ for $ k\in\mathbb{C}^+$ and the set $\{k_j\}$ in $\mathbb{C}^+$, find a  piece-wise meromorphic $2\times 2$ matrix valued function $\check M^{(x)}(x, k)$ that satisfies the following conditions:

\begin{enumerate}
    \item Jump condition \eqref{jump_M_(x)_left}.

\item Behavior at $\infty$:
\begin{equation}\label{inf_hatM_(x)__}
    \check M^{(x)}(x, k)=I+O\left(\frac{1}{ k}\right),\quad k\to\infty.
\end{equation}

\item Residue conditions \eqref{res-M_(x)__}.
\end{enumerate}

Similarly to the Problem I, the proposition similar to Proposition \ref{prop:Mx} holds for \textbf{$\widecheck{\text{RH}}^{(x)}$}, which justifies the following procedure of the inverse mapping: 

\begin{enumerate}[Step 1.]
    \item Given $a(k)$ and $b(k)$, compute $\kappa_1^0$ and $\kappa_2^0$ through \eqref{a_at_i}, and determine $\tilde a(k)$ and $\tilde b(k)$ through \eqref{s_via_til_s}. With these data, formulate the Riemann–Hilbert problem \textbf{$\widecheck{\text{RH}}^{(x)}$};

    \item Solve the RH problem constructed in the previous step and evaluate its solution $ \check M^{(x)}(x,k)$ at $ k=\infty$ (c.f. \eqref{inf_M_(x)__}):

    \begin{equation*}
      \check M^{(x)}(x, k)=I+\frac{1}{4\ii k}\begin{pmatrix}
           -\eta(x) & \xi(x)\\
           \xi(x) & \eta(x) 
       \end{pmatrix}+O\left(\frac{1}{k^2}\right).
    \end{equation*}

    \item Determine $u_{0x}(x)$  as follows:
    
\begin{equation}\label{m_0_via_RH_x__}
    u_{0x}(x)=\xi(x). 
    \end{equation}

    \item  Having $u_{0x}(x)$, reconstruct $u_0(x)$ by
\begin{equation}\label{Fund_u0__}
   u_0(x) = \int_{-\infty}^x u_{0x}(y)\dd y +2\pi k.
\end{equation}

\end{enumerate}

For an alternative way to reconstruct $u_0(x)$ see Remark \ref{rem:rec_u0}.

\subsubsection{The Inverse $t$-Spectral Mapping}

Naturally, Problems I and II share the same $t$-spectral mapping.

\section{The master Riemann-Hilbert problems}\label{sec:5}

\subsection{Problem I}\label{subsec:5.1} 

Assume that Problem I \eqref{pr_1} admits a solution $u(x,t)$. Using suitable solutions of the Lax pair equations, we define the piecewise (with respect to the  contour shown in Figure \ref{fig:contour_RH_x}) meromorphic matrix-valued function $M(x,t,k)$, depending on $(x,t)$ as parameters:

\begin{equation}\label{M_(xt)}
M^{(xt)}(x,t,k)=\begin{cases}

\left( \frac{\Phi_{ 2}^{(1)}(x,t,k)}{a(k)},\Phi_{ 3}^{(2)}(x,t,k)\right),\quad k\in\mathbb{C}^+\cap\{|k| >\epsilon\},\\

\left( \frac{\Psi_{0 2}^{(1)}(x,t,k)}{\tilde a(k)},\Psi_{0 3 }^{(2)}(x,t,k)\right),\quad k\in\mathbb{C}^+\cap\{|k| <\epsilon\},\\

\left( \Phi_{ 3}^{(1)}(x,t,k),\frac{\Phi_{ 2}^{(2)}(x,t,k)}{a^*( k)}\right),\quad k\in\mathbb{C}^-\cap\{|k| >\epsilon\},\\

\left( \Psi_{0 3}^{(1)}(x,t,k),\frac{\Psi_{0 2}^{(2)}(x,t,k)}{\tilde a^*( k)}\right),\quad k\in\mathbb{C}^-\cap\{|k| <\epsilon\},

\end{cases}
\end{equation}
where the functions $\Psi_{0j}$
are defined by \eqref{psi_0}.

The function $M^{(xt)}(x,t,k)$ possesses the following properties:

\begin{enumerate}
    \item Jump relation across $\mathbb{R}\cup\{|k|=\epsilon\}$
    \begin{subequations}
        \label{jump_M_(xt)}
        \begin{equation}
           M_-^{(xt)}(x,t,k)=M_+^{(xt)}(x,t,k)J^{(xt)}(x,t,k),\quad k\in \mathbb{R}\cup\{|k|=\epsilon\} 
        \end{equation}
        where,
         \begin{equation} \label{jump_M_(xt)_matr}
          J^{(xt)}(x,t,k)=\eul^{-\ii kp(x,t,k)\sigma_3} J^{(xt)}_0(k)\eul^{\ii kp(x,t,k)\sigma_3}
        \end{equation} 
        with $p(x,t,k)=x-\frac{t}{4 k^2}$ and
\begin{equation}\label{jump_M_(xt)_matr_0}
   J^{(xt)}_0(k)=\begin{cases}
       \begin{pmatrix}
       1&-r^*(k)\\-r(k)&1+r(k)r^*(k)
   \end{pmatrix},\quad k\in\mathbb{R}\cap\{|k| >\epsilon\},\\
   \begin{pmatrix}
       1+\tilde r(k)\tilde r^*(k)&\tilde r^*(k)\\\tilde r(k)&1
   \end{pmatrix},\quad k\in\mathbb{R}\cap\{|k| <\epsilon\},\\
   \begin{pmatrix}
       1&0\\
       -\frac{\kappa_2^0}{a(k)\tilde a(k)}&1
   \end{pmatrix},\quad k\in\mathbb{C}^+\cap\{|k| =\epsilon\},\\
   \begin{pmatrix}
       1&-\frac{\kappa_2^0}{a^*(k)\tilde a^*(k)}\\
       0&1
   \end{pmatrix},\quad k\in\mathbb{C}^-\cap\{|k| =\epsilon\}
   \end{cases} 
\end{equation}
and $r(k)=\frac{b^*(k)}{a(k)}$, $\tilde r(k)=\frac{\tilde b^*(k)}{\tilde a(k)}$.
Notice that \eqref{s_via_til_s} implies
\[
\tilde r(k)=r(k)-\frac{\kappa_2 }{a(k)\tilde a(k)}, \quad k\in\mathbb{R}.\]
 \end{subequations}

\item Behavior at $\infty$:
\begin{equation}\label{inf_M_(xt)}
     M^{(xt)}(x,t,k)=I+\frac{1}{4\ii k}\begin{pmatrix}
    *&u_{x}(x,t)\\
    u_{x}(x,t)&*\end{pmatrix}+O(\frac{1}{k^2}),\quad k\to\infty.
\end{equation}

\item $\det M^{(xt)}(x,t,k)\equiv1$

\item Symmetry properties:
\begin{equation}\label{sym-M_(xt)}
M^{(xt)}( k)=\overline{M^{(xt)}(-\bar k)},\qquad M^{(xt)}(k)=\sigma_2 \overline{M^{(xt)}(\bar k)}\sigma_2
\end{equation}

\item Residue properties. Let $\{k_j\}$ be zeros of $a(k)$ in $\mathbb{C}^+$. Assume that they are all simple. Then

\begin{subequations}\label{res-M_(xt)}
\begin{align}\label{res-M+_(xt)}
\Res_{k_j}M^{(xt)(1)}(x,t,k)&=\frac{e^{2\ii k_jp(x,t,k_j)}}{\dot a(k_j)b(k_j)}M^{(xt)(2)}(x,t,k_j),\\
\label{res-M-_(xt)}
\Res_{- k_j}M^{(xt)(2)}(x,t,k)&=\frac{e^{2\ii k_jp(x,t, k_j)}}{\dot a( k_j)b( k_j)}M^{(xt)(1)}(x,t,- k_j).
\end{align}
\end{subequations}

\item Behavior at $0$:
\begin{equation}\label{i_beh-M_(xt)}
M^{(xt)}(x,t,k)=P^{-1}(x,t)\left(I-\ii k\begin{pmatrix}
     *&u_t(x,t)\\
     u_t(x,t)& *
\end{pmatrix}+O(k^2)\right).\end{equation}
    
\end{enumerate}

The properties of $M^{(xt)}$ stated above (that follow from the analysis of the direct problem) suggest formulating a Riemann--Hilbert factorization problem parametrized by $(x,t)$.

\textbf{The Riemann--Hilbert problem RH$^{(xt)}$ for $M^{(xt)}(x,t,k)$:} Given $a(k)$, $b(k)$ for $k\in\mathbb{C}^+$, and the set $\{k_j\}_1^N\subset \mathbb{C}^+$, find a piece-wise meromorphic $2\times 2$  matrix valued function $M^{(xt)}(x,t,\mu)$ that satisfies the following conditions:

\begin{enumerate}
    \item Jump condition \eqref{jump_M_(xt)} across $\mathbb{R}\cup|k|=\epsilon$.

\item Behavior at $\infty$:
\begin{equation}\label{inf_hatM_(xt)}
      M^{(xt)}(x, t,k)=I+O\left(\frac{1}{ k}\right),\quad k\to\infty.
\end{equation}

\item Residue conditions \eqref{res-M_(xt)}.
\end{enumerate}

\begin{proposition}\label{prop:5.2}
   There exists a unique solution $M^{(xt)}(x,t,k)$ of the RH problem \textbf{RH$^{(xt)}$ } for all $x$ and $t$. Moreover,
        $\det  M^{(xt)}\equiv1$
    and
        \begin{equation}\label{sym-hatM_(xt)}
M^{(xt)}( k)=\overline{ M^{(xt)}(-\bar k)},\qquad  M^{(xt)}(k)=\sigma_2\overline{ M^{(xt)}(\bar k)}\sigma_2
\end{equation}
  
\end{proposition}

Concerning the proof, 
    notice that the jump matrix $ J^{(xt)}(x,t, k)$ defined in  
\eqref{jump_M_(xt)_matr}--\eqref{jump_M_(xt)_matr_0} is positively defined on $\mathbb R$ and satisfies the symmetry \[ J^{(xt)}(x, t,k)=\overline{ J^{(xt)T}(x,t,\bar  k)}, \qquad k\in\{|k|=\epsilon\}.\] 
These properties allow us to invoke Zhou’s vanishing lemma (see \cite{Z89}), which ensures the solvability of the Riemann–Hilbert problem \textbf{RH$^{(xt)}$}.

\begin{remark}
    The construction of the master RH problem \textbf{RH$^{(xt)}$} uses the same 
data as \textbf{RH$^{(x)}$}, i.e., the 
scattering functions associated with the initial condition alone.
\end{remark}

\begin{remark}\label{rem:jump_coins}
   In the case 
    $u_{0}(0)=4\pi k$,  \textbf{RH$^{(xt)}$} simplifies to the following:
    \begin{enumerate}
    \item Jump condition  \begin{equation*}
           M_-^{(xt)}(x,t,k)=M_+^{(xt)}(x,t,k)J(x,t,k),\quad k\in \mathbb{R} 
        \end{equation*}
        where
        \begin{equation*}
          J(x,t,k)=\eul^{-\ii kp(x,t,k)\sigma_3} J_0(k)\eul^{\ii kp(x,t,k)\sigma_3}
        \end{equation*} 
        with  
\begin{equation*}
   J_0( k)=
      \begin{pmatrix}
       1&-r^*(k)\\-r(k)&1+r(k)r^*(k)
   \end{pmatrix},\quad k\in\mathbb{R}.
\end{equation*}
Here $\mathbb{R}$ is oriented from left to right.
\item Behavior at $\infty$:
\begin{equation*}
      M^{(xt)}(x, t,k)=I+O\left(\frac{1}{ k}\right),\quad k\to\infty.
\end{equation*}
\item Residue conditions \eqref{res-M_(xt)}.
\end{enumerate}   
\end{remark}

\subsection{Problem II}\label{subsec:5.2}
Assume that Problem II \eqref{pr_2} admits a solution $u(x,t)\in C^1([0,T],\tilde H^{1,2}((-\infty,0)))$.

Introduce 
\begin{align*}
    d(k)=a(k)A^*(k)+b(k)B^*(k), \quad \tilde d(k)=\tilde a(k)\tilde A^*(k)+\tilde b(k)\tilde B^*(k),
\end{align*}
and define the piecewise meromorphic matrix-valued function $M(x,t,k)$, depending on $(x,t)$ as parameters, in the domains separated by the contour shown in Figure \ref{fig:contour_RH_x}:

\begin{equation}\label{M_(xt)__}
\check M^{(xt)}(x,t,k)=\begin{cases}

\left(\check \Phi_{3}^{(1)}(x,t,k),\frac{\Phi_{ 1}^{(2)}(x,t,k)}{d^*(k)}\right),\quad k\in\mathbb{C}^+\cap\{|k| >\epsilon\},\\

\left(\check \Psi_{0 3}^{(1)}(x,t,k),\frac{\Psi_{0 1}^{(2)}(x,t,k)}{\tilde d^*(k)}\right),\quad k\in\mathbb{C}^+\cap\{|k| <\epsilon\},\\

\left( \frac{\Phi_{ 1}^{(1)}(x,t,k)}{d(k)},\check \Phi_{3}^{(2)}(x,t,k)\right),\quad k\in\mathbb{C}^-\cap\{|k| >\epsilon\},\\

\left( \frac{\Psi_{0 1}^{(1)}(x,t,k)}{\tilde d(k)},\check \Psi_{0 3}^{(2)}(x,t,k)\right),\quad k\in\mathbb{C}^-\cap\{|k| <\epsilon\},

\end{cases}
\end{equation}
where the functions $\Psi_{0j}$
are defined by \eqref{psi_0}.

The function $\check M^{(xt)}(x,t,k)$ possesses the following properties:
\begin{enumerate}
    \item Jump relation across $\mathbb{R}\cup\{|k|=\epsilon\}$:
    \begin{subequations}
        \label{jump_M_(xt)__}
        \begin{equation}
          \check  M_-^{(xt)}(x,t,k)=\check M_+^{(xt)}(x,t,k)\check J^{(xt)}(x,t,k),\quad k\in\Sigma,
        \end{equation}
        where
         \begin{equation}
         \check  J^{(xt)}(x,t,k)=\eul^{-p(x,t,k)\sigma_3} \check J^{(xt)}_0(k)\eul^{p(x,t,k)\sigma_3}
        \end{equation} 
        with       
\begin{equation}\label{jump_M_(xt)_0__}
  \check J^{(xt)}_0(k)=\begin{cases}      \begin{pmatrix}
        1+\rho(k)\rho^*(k)&-\rho(k)\\
          -\rho^*(k)&1
       \end{pmatrix},\quad \mu\in\mathbb{R}\setminus\{|k|<\epsilon\},\\
\begin{pmatrix}
          1&\tilde \rho(k)\\
          \tilde \rho^*(k)& 1+\tilde\rho(k)\tilde\rho^*(k)
       \end{pmatrix},\quad \mu\in\mathbb{R}\cap\{|k|<\epsilon\},\\
      \begin{pmatrix}
{1}&\Xi(k)\\
          0&{1}
       \end{pmatrix},\quad \mu\in\mathbb{C}^+\cap\{|k|=\epsilon\},\\
            \begin{pmatrix}
          {1}&0\\
         \Xi^*(k)&{1}
       \end{pmatrix},\quad \mu\in\mathbb{C}^-\cap\{|k|=\epsilon\},  
   \end{cases}
\end{equation}
    \end{subequations}
where
\begin{align}\label{rho}
&\rho(k):= \frac{a(k)B(k)-b(k)B(k)}{d^*(k)},\\
\label{til_rho}
&\tilde \rho(k):= \frac{\tilde a(k)\tilde B(k)-\tilde b(k)\tilde B(k)}{\tilde d^*(k)},\\
&\Xi(k):=\frac{\kappa_1^0(A(k)\tilde B(k)-B(k)\tilde A(k))+\kappa_2^0(B(k) \tilde B(k) +A(k) \tilde A(k))}{\tilde d^*(k) d^*(k)}.
\end{align}
Notice that
\[
\rho(k)-\tilde \rho(k)+\Xi(k)=0.
\]

\item Behavior at $\infty$:
\begin{equation}\label{inf_M_(xt)__}
    \check  M^{(xt)}(x,t,k)=I+\frac{1}{4\ii k}\begin{pmatrix}
    *&u_{x}(x,t)\\
    u_{x}(x,t)&*\end{pmatrix}+O(\frac{1}{k^2}),\quad k\to\infty.
\end{equation}

\item $\det \check M^{(xt)}(x,t,k)\equiv1$

\item Symmetry properties:
\begin{equation}\label{sym-M_(xt)__}
\check M^{(xt)}(- k)=\overline{\check M^{(xt)}(\bar k)},\qquad \check M^{(xt)}(k)=\sigma_2\overline{\check M^{(xt)}(\bar k)}\sigma_2
\end{equation}

\item Residue properties. Let $\{\eta_j\}$ be zeros of $d(k)$ in $\mathbb{C}^-$. Assume that they are all simple. Then

\begin{subequations}\label{res-M_(xt)__}
\begin{align}\label{res-M+_(xt)__}
\Res_{\eta_j}\check M^{(xt)(1)}(x,t,k)&=-\frac{e^{2\ii \eta_jp(x,t,\eta_j)}B^*(\eta_j)}{\dot d(\eta_j)a(\eta_j)}\check M^{(xt)(2)}(x,t,\eta_j),\\
\label{res-M-_(xt)__}
\Res_{- \eta_j}\check M^{(xt)(2)}(x,t,k)&=\frac{e^{2\ii \eta_jp(x,t, \eta_j)}B^*(\eta_j)}{\dot d( \eta_j)a( \eta_j)}\check M^{(xt)(1)}(x,t,- \eta_j).
\end{align}
\end{subequations}

\item Behavior at $0$:
\begin{equation}\label{i_beh-M_(xt)__}
\check M^{(xt)}(x,t,k)=P^{-1}(x,t)\left(I-\ii k\begin{pmatrix}
     *&u_t(x,t)\\
     u_t(x,t)& *
\end{pmatrix}+O(k^2)\right) .\end{equation}

\end{enumerate}

\begin{remark}
    In the case $u(0,0)=4\pi k$ the matrix $\check J_0^{(xt)}(k)$ simplifies to the following:

\begin{equation}\label{jump_M_(xt)_0__simpl}
\check J^{(xt)}_0(k)=\begin{cases}      \begin{pmatrix}
          \frac{1}{dd^*}&-\frac{aB-bA}{d^*}\\
          -\frac{a^*B^*-b^*A^*}{d}&1
       \end{pmatrix},\quad \mu\in\mathbb{R}\setminus\{|k|<\epsilon\},\\
\begin{pmatrix}
          1&\frac{\tilde a\tilde B-\tilde b\tilde A}{\tilde d^*}\\
          \frac{\tilde a^*\tilde B^*-\tilde b^*\tilde A^*}{\tilde d}&\frac{1}{\tilde d\tilde d^*}
       \end{pmatrix},\quad \mu\in\mathbb{R}\cap\{|k|<\epsilon\},\\
      \begin{pmatrix}
          {1}&\frac{A\tilde B-B\tilde A}{\tilde d^* d^*}\\
          0&{1}
       \end{pmatrix},\quad \mu\in\mathbb{C}^+\cap\{|k|=\epsilon\},\\
            \begin{pmatrix}
          {1}&0\\
         \frac{A^*\tilde B^*-B^*\tilde A^*}{\tilde d d}&{1}
       \end{pmatrix},\quad \mu\in\mathbb{C}^-\cap\{|k|=\epsilon\}. 
   \end{cases}
\end{equation}
    
\end{remark}

The properties of $\check M^{(xt)}$ stated above give rise to a 
family of Riemann--Hilbert factorization problems parametrized by $(x,t)$.

\textbf{The Riemann--Hilbert problem \textbf{$\widecheck{\text{RH}}^{(xt)}$} for $\check M^{(xt)}(x,t,k)$:} Given $a(k)$, $b(k)$ for $k\in\mathbb{C}^+$, and the set $\{k_j\}_1^N\subset \mathbb{C}^+$, find a piece-wise meromorphic $2\times 2$  matrix valued function $M^{(xt)}(x,t,\mu)$ that satisfies the following conditions:

\begin{enumerate}
    \item Jump condition \eqref{jump_M_(xt)__} across $\mathbb{R}\cup|k|=\epsilon$.

\item Behavior at $\infty$:
\begin{equation}\label{inf_hatM_(xt)__}
     \check M^{(xt)}(x, t,k)=I+O\left(\frac{1}{ k}\right),\quad k\to\infty.
\end{equation}

\item Residue conditions \eqref{res-M_(xt)__}.
\end{enumerate}

  Similarly to Problem I, the following proposition holds:

\begin{proposition} There exists a unique solution $\check M^{(xt)}(x,t,k)$ of the RH problem \textbf{$\widecheck{\text{RH}}^{(x)}$} for all $x$ and $t$. Moreover,
        $\det \check M^{(xt)}\equiv1$
    and
        \begin{equation}\label{sym-hatM_(xt)__}
\check M^{(xt)}( k)=\overline{\check M^{(xt)}(-\bar k)},\qquad \check M^{(xt)}(k)=\sigma_2\overline{\check M^{(xt)}(\bar k)}\sigma_2
\end{equation}

\end{proposition}

\section{Recovering $u(x,t)$ from the solution of the RH problem}\label{sec:6}

In this subsection, we describe the derivation of a (local) solution to the sG equation (together with the derivation of the corresponding Lax pair equations), starting from a RH problem parametrized by $x$ and $t$, motivated by the considerations in Section \ref{sec:5}.

Let $\Gamma$ denote an oriented contour in the complex plane, invariant under the symmetries $k\mapsto-k$ and $k\mapsto\bar k$. Suppose that ${\mathbb C}\setminus\Gamma = D_1\cup D_2$ so that $\Gamma$ is the counterclockwise boundary of $D_1$ (the clockwise boundary of $D_2$).

Consider the following \textbf{RH problem parametrized by $x$ and $t$}: Given the $2\times 2$-matrix valued function $J_0(k)\in L^2(\Gamma)\cap L^{\infty}(\Gamma)$ for $k\in\Gamma$ such that 
\begin{subequations}\label{jump_cond}
    \begin{align}\label{det_jump}
    &\det J_0(k)\equiv 1,\\
    \label{jump_at_inf}
    &J_0(k)=I+O(\frac{1}{k}),\quad k\to\infty,\\
    \label{jump_at_0}
    &J_0(k)=I+O(k),\quad k\to0,
    \end{align}
\end{subequations}
(whenever the corresponding points lie on $\Gamma$)
and the sets $\{k_j,c_j\}_1^N$ with $k_j\in D_1$, and $\{\eta_j,d_j\}_1^M$ with $\eta_j\in D_2$,
find a piece-wise meromorphic (w.r.t. $\Gamma$) $2\times 2$-matrix valued function $\hat M(x,t,k)$ satisfying:
\begin{enumerate}[\textbullet]
\item
\emph{Jump} condition
\begin{equation}\label{jump-y_loc}
 M_+(x,t,k)=M_-(x,t,\lambda) J(x,t,k),\qquad k\in\Sigma,
\end{equation}
where  
$J(x,t,k)=\eul^{-ik(x-\frac{t}{4k^2})\sigma_3}J_0(k)\eul^{ik(x-\frac{t}{4k^2})\sigma_3}$.

\item  \emph{Normalization} condition:
\begin{equation}\label{norm-m-hat_loc}
\hat M(x,t,k)=I+\ord(\frac{1}{k}), \quad k\to\infty.
\end{equation}

\item \emph{Residue} conditions:
\begin{align}\label{res_hatM_loc}
\Res_{ k_j} M^{(1)}(x,t,k)&= c_j  M^{(2)}(x,t, k_j) \eul^{-2\ii  k_j (x-\frac{t}{4k_j^2})}, ~  k_j\in D_1, \\\label{res_hatM__loc}
\Res_{\eta_j} M^{(2)}(x,t,k)&=d_j  M^{(1)}(x,t,\eta_j)\eul^{2\ii \eta_j (x-\frac{t}{4\eta_j^2})} , ~ \eta_j\in D_2.
\end{align}
\end{enumerate}

The following Proposition holds:
\begin{proposition}\label{prop:det}
\begin{enumerate}
    \item  If $M$ is a solution of RH problem \eqref{jump-y_loc}--\eqref{res_hatM__loc}, then $\det  M=1.$

    \item  If a solution of the RH problem \eqref{jump-y_loc}--\eqref{res_hatM__loc} exists, it is unique.
\end{enumerate}

\end{proposition}

Now, assume that the RH problem \eqref{jump-y_loc}--\eqref{res_hatM__loc} 
has a solution $M(x,t,k)$ that satisfies the \emph{symmetries}
\begin{equation}
 \label{sym_1}
 M(k)=\overline{ M(-\bar k)}
 \end{equation}
and 
 \begin{equation}
 \label{sym_2}
 M(k)=\begin{pmatrix}
    0&-\ii\\\ii&0
\end{pmatrix}\overline{ M(\bar k)}\begin{pmatrix}
    0&-\ii\\\ii&0
\end{pmatrix}
 \end{equation} 
and  is differentiable w.r.t. $x$ and $t$.

Now, our goal is to show that, by evaluating the solution $ M(x,t,k)$ of the RH problem \eqref{jump-y_loc}--\eqref{res_hatM__loc} at appropriate points of the extended complex plane $\overline{\mathbb C}$, one can obtain a solution of the sG equation.

We proceed as follows:
\begin{enumerate}[(a)]
\item 
Starting from $M(x,t,k)$, define $2\times 2$-matrix valued functions
\begin{equation}
    \label{psi}
    \Psi (x,t,k):=  M(x,t,\mu)\eul^{-ik(x-\frac{t}{4k^2})\sigma_3}
\end{equation}
 and show that $\Psi(x,t,k)$ satisfies the system of differential equations:
\begin{equation}\label{Lax-hat-hat}
\begin{split}
\Psi_x&=\hat{U}\Psi, \\
 \Psi_t&=\hat{V}\Psi,
\end{split}
\end{equation}
where $\hat{U}$ and $\hat{V}$ have the same (rational) dependence on $k$ as in \eqref{Lax-UV}, with coefficients given in terms of $ M(x,t,k)$ evaluated at appropriate values of $k$.
\item
Show that the compatibility condition for \eqref{Lax-hat-hat}, i.e., the equality $\hat{U}_t - \hat{V}_x + [\hat{U},\hat{V}]=0$, reduces to 
\eqref{sG}.
\end{enumerate}

\begin{proposition}\label{Prop_Lax_y}
    Let $M(x,t,k)$ be the solution of the RH problem \eqref{jump-y_loc}--\eqref{res_hatM__loc} that satisfies symmetries \eqref{sym_1} and \eqref{sym_2}. 
    Then $\Psi(x,t,k)$ defined by \eqref{psi} satisfies the differential equation
\begin{equation}\label{Lax_y_Psi}
    \hat\Psi_x=\hat{U}\hat\Psi
\end{equation}
with 
\begin{equation}\label{hat_hat_U}
    \hat{U}(x,t,k)=-\ii k \sigma_3 + \frac{\xi(x,t)}2{}\begin{pmatrix}
    0&1\\-1&0
\end{pmatrix},
\end{equation}
where 
$\xi(x,t)\in\mathbb{R}$ can be obtained from the large $k$ expansion of $M(x,t,k)$: 
\begin{equation}\label{xi-M}
 M(x,t,k)=I+\frac{1}{4\ii k}\begin{pmatrix}
    -\eta(x,t) & \xi(x,t)\\
   \xi(x,t)& \eta(x,t)
\end{pmatrix}+O(\frac{1}{k^2}),\qquad k\to\infty.
\end{equation}

\end{proposition}

\begin{proof}
     First, notice that $\Psi(x,t,k)$ satisfies the jump condition
\[
\Psi^+(x,t,k)=\Psi^-(x,t,k)J_0(k)
\]
with the jump matrix $J_0(k)$ independent of $x$. Hence, $\Psi_x(x,t,k)$ satisfies the same jump condition.

As for the residue conditions, we notice that the symmetry assumptions  
\eqref{sym_1} and \eqref{sym_2} imply that if $ k_j$ is pole of $\hat M^{(1)}$ then $-\bar k_j$ is also pole of $\hat M^{(1)}$ with $c_{ k_j}=-\overline{c_{-\bar k_j}}$. Analogously, if $\eta_j$ is pole of $\hat M^{(2)}$ then $-\bar\eta_j$ is also pole of $\hat M^{(2)}$ with $d_{\eta_j}=-\overline{d_{-\bar\eta_j}}$. Moreover,
\begin{equation}\label{res-sym}
    \eta_j=- k_j, \quad d_j=c_j.
\end{equation}
Since $\{c_j\}_1^N$ are independent of $x$, 
$\Psi_x(x,t,k)$ satisfies the same residue conditions as $\Psi(x,t,k)$ does:

\begin{align*}
\Res_{ k_j}\Psi^{(1)}&= c_j  \Psi^{(2)}( k_j) , ~  k_j\in D_1, \\
\Res_{- k_j} \Psi^{(2)}&= c_j\Psi^{(1)}(- k_j) , ~ - k_j\in D_2.
\end{align*}

Therefore,  $ \Psi_x  \Psi^{-1}= M_x M^{-1}-\ii k  M \sigma_3  M ^{-1}$   has neither jump nor singularities at $ k_j$, and thus it is a meromorphic function, with possible singularity at $k=\infty$.

Let us analyze the behavior of $ \Psi_x \Psi^{-1}$ as $k\to\infty$. First of all, notice that the symmetry \eqref{sym_1} implies
    \begin{equation}
    M = I+\frac{\ii}{k}\begin{pmatrix}
     m_{11}^\infty &  m_{12}^\infty\\
     m_{21}^\infty&  m_{22}^\infty
\end{pmatrix}+O(\frac{1}{k^2}), \quad k\in\mathbb{C}_+
   \end{equation}
with $\hat m_{+ij}^\infty \in\mathbb{R}$. Moreover, the symmetry \eqref{sym_2} implies
    \begin{equation}
    M = I-\frac{\ii}{k}\begin{pmatrix}
     m_{22}^\infty & - m_{21}^\infty\\
    - m_{12}^\infty&  m_{11}^\infty
\end{pmatrix}+O(\frac{1}{k^2}), \quad k\in\mathbb{C}_-.
   \end{equation}
Therefore, $
 M_x=O(\frac{1}{k})$  and thus 
 $M_x M^{-1}=O(\frac{1}{k})$, which leads to 
the following expansion for $\Psi_x\Psi^{-1}$:
\begin{equation}\label{Psi_y_Psi_inf_C+}
  \Psi_x \Psi^{-1}=-\ii k \sigma_3 + \begin{pmatrix}
     m_{11}^\infty+ m_{22}^\infty&-2 m_{12}^\infty\\2 m_{21}^\infty&- m_{11}^\infty- m_{22}^\infty
\end{pmatrix}+O(\frac{1}{k}),  \quad k\to \infty, \quad k\in\mathbb{C}_+,
\end{equation}
\begin{equation}\label{Psi_y_Psi_inf_C-}
  \Psi_x  \Psi^{-1}=-\ii k \sigma_3 - \begin{pmatrix}
     m_{11}^\infty+ m_{22}^\infty&2 m_{21}^\infty\\-2 m_{12}^\infty&- m_{11}^\infty- m_{22}^\infty
\end{pmatrix}+O(\frac{1}{k}),  \quad k\to \infty, \quad k\in\mathbb{C}_-.
\end{equation}
Since $\Psi_x  \Psi^{-1}$ is meromorphic in $\mathbb{C}$ these expansions should coincide, and thus, 
$ m_{11}^\infty=- m_{22}^\infty$ and $ m_{12}^\infty= m_{21}^\infty$.

Denoting $\eta=4 m_{11}$ and $\xi=-4m_{12}$, we get
\eqref{xi-M}. Therefore, Liouville's theorem yields \eqref{Lax_y_Psi}--\eqref{hat_hat_U}.
\end{proof}

\begin{proposition}\label{Prop_Lax_t}
$\Psi(x,t,k)$ defined in \eqref{psi} satisfies the differential equation
\begin{equation}\label{Lax_t_Psi}
    \Psi_t=\hat{V}\Psi
\end{equation}
with 
\begin{equation}\label{hat_hat_V}
    \hat{V}=-\frac{1}{4\ii k} \begin{pmatrix}
        \alpha^2-\beta^2&-2\alpha  \beta \\
        -2 \alpha \beta&-\alpha^2+\beta^2
    \end{pmatrix},
\end{equation}
where 
$ \alpha(x,t)\in\mathbb{R}$ and $\beta(x,t)\in\mathbb{R}$ can be obtained from the from the expansion of $ M(x,t,k)$ as $k\to 0$: 
\begin{equation}\label{alpha_beta-M}
     M(x,t,k)=\begin{pmatrix}
     \alpha(x,t) & \beta(x,t)\\
   -\beta(x,t)& \alpha(x,t)
\end{pmatrix}+
O(k),\quad k\to 0
\end{equation}
with 
\begin{equation}\label{det_alpha_beta}
    \alpha^2(x,t)+\beta^2(x,t)=1.
\end{equation}
\end{proposition}

\begin{proof}
    Similarly to Proposition \ref{Prop_Lax_y}, we notice that $\Psi_t \Psi^{-1}=M_t M^{-1}+\frac{\ii }{4k} M\sigma_3M^{-1}$ has neither jump nor poles at $ k_j$, and thus it is a meromorphic function, with possible singularities at $k=\infty$ and $k=0$:

Evaluating $\Psi_t\Psi^{-1}$ near these points, we have the following.
\begin{enumerate}[(i)]

\item As $k\to\infty$, we have $ M M_t^{-1}=O(\frac{1}{k})$, $p_t(k)=O(\frac{1}{k})$ and thus
\begin{equation}\label{psi-inf-t}
\hat\Psi_t\hat\Psi^{-1}(k)=\ord(k^{-1}),\qquad k\to\infty.
\end{equation}

\item As $k\to 0$, first notice that \eqref{jump_at_0} together with the  symmetries \eqref{sym_1} and  \eqref{sym_2} implies \eqref{alpha_beta-M}, and the Proposition \ref{prop:det} yields \eqref{det_alpha_beta}.
Therefore,
\begin{equation}\label{Psi_t_Psi_0}
  \hat \Psi_t \hat \Psi^{-1}=-\frac{1}{4\ii k} \begin{pmatrix}
\hat \alpha^2- \hat \beta^2&-2 \hat \alpha \hat \beta\\-2 \hat \alpha \hat \beta&-\hat \alpha^2+\hat \beta^2
\end{pmatrix}+O(1),  \quad k\to 0.
\end{equation}
Combining \eqref{psi-inf-t} and \eqref{Psi_t_Psi_0} with the Liouville's theorem yields \eqref{Lax_t_Psi}--\eqref{alpha_beta-M}.
\end{enumerate}

\end{proof}

We now show that the compatibility condition
\begin{equation}\label{compat}
\hat{U}_t - \hat{V}_x + [\hat{U},\hat{V}]=0
\end{equation}
yields the sG equation.
For this purpose, substituting into \eqref{compat} the expressions for $\hat{U}$ and 
$\hat{V}$ from \eqref{hat_hat_U} and \eqref{hat_hat_V}, respectively, and
equating to $0$ the coefficients  at various expressions involving $k$ we get algebraic and differential equations amongst $ \alpha$, $ \beta$ and $ \xi$:
\begin{proposition}
Let $\alpha(x,t)$, $ \beta(x,t)$ and $\xi(x,t)$ be the functions determined in terms of $ M(x,t,k)$ as in Propositions \ref{Prop_Lax_y} and \ref{Prop_Lax_t}. Then they satisfy the following equations:
\begin{subequations}\label{rel}
\begin{align}\label{sys_rec1_}
&\xi_t- 2\alpha \beta=0,\\\label{sys_rec3_}
&\alpha \alpha_x-\beta \beta_x+ \xi \beta  \alpha=0,\\\label{sys_rec4_}
&2\alpha_x \beta+2\alpha \beta_x- \xi( \alpha^2- \beta^2)=0.
\end{align}
\end{subequations}
\end{proposition}

\begin{proposition}
    \label{reduce} 
Introducing $u(x,t)$ up to $2\pi k$ such that
\begin{equation}\label{hmbbgg}
 \sin \frac{u}{2} = \beta, \qquad \cos \frac{u}{2} =\alpha
\end{equation}
or
\begin{equation}\label{hmbbgg_}
 \sin \frac{u}{2} = -\beta, \qquad \cos \frac{u}{2} =-\alpha
\end{equation}
 equations  \eqref{sys_rec1_}-- \eqref{sys_rec4_} reduce to \eqref{sG}. 
\end{proposition}

\begin{proof}
First notice that \eqref{hmbbgg} implies
\[
\sin u=2\alpha\beta,\qquad \cos u=\alpha^2-\beta^2.
\]
Moreover, since $\alpha^2+\beta^2=1$, $\alpha$ and $\beta$ cannot be simultaneously $0$.

Combining these facts with \eqref{sys_rec4_} (in case $\alpha^2-\beta^2\neq 0$) or  \eqref{sys_rec3_} (in case $2\alpha\beta\neq 0$), we conclude that 
\[
\xi=u_x.
\]

Consequently, \eqref{sys_rec1_} implies \eqref{sG}.
\end{proof}

\section{Differentiability of solution of master RH problems}\label{sec:7}

\subsection{Basic facts about $L^2$-RH problems.}

For our purposes, we will need several results concerning the equivalence between
$L^2$ Riemann--Hilbert problems (whose jump matrices may be discontinuous at
intersection points of the contour) and singular integral equations; see, for
example, \cites{L18,ZH89}. We collect the necessary definitions and results
below.

Throughout this section, $\Gamma$ denotes an oriented contour consisting of a
finite union of piecewise $C^1$ arcs. The arcs may intersect only at their
endpoints, and all intersections are assumed to be transversal. In particular,
this class of contours includes configurations such as the union of the real
line with finitely many circles.

The orientation of $\Gamma$ is chosen so that
\[
\Gamma=\partial D_+=-\partial D_-,
\]
where $\mathbb{C}\setminus\Gamma=D_+\cup D_-$.

For a bounded component $D$ of $\mathbb{C}\setminus\Gamma$, denote by $E^2(D)$
the class of analytic functions $f$ in $D$ for which there exists a sequence of
rectifiable Jordan curves $\{C_n\}\subset D$ such that for every compact set
$D_c\subset D$ there exists $N$ so that $C_n$ surrounds $D_c$ for all $n>N$ and
\[
\sup_{n\ge1}\int_{C_n}|f(z)|^2|dz|<\infty .
\]

If $D$ is unbounded, define $f\in E^2(D)$ if
$f\circ\phi^{-1}\in E^2(\phi(D))$, where
\[
\phi(z)=\frac{1}{z-z_0}
\]
for some $z_0\in\mathbb{C}\setminus\overline{D}$. We also write
\[
\dot E^2(D)=\{f\in E^2(D):zf(z)\in E^2(D)\},
\]
that is, $f(z)\to0$ as $z\to\infty$.

Denote by $\mathcal{B}(L^2(\Gamma))$ the space of bounded linear operators on
$L^2(\Gamma)$ and by $\mathcal{F}(L^2(\Gamma))$ the set of Fredholm operators
on $L^2(\Gamma)$. The index map
\[
\mathrm{Ind}:\mathcal{F}(L^2(\Gamma))\to\mathbb{Z}
\]
is constant on the connected components of $\mathcal{F}(L^2(\Gamma))$.

For $h\in L^2(\Gamma)$ the Cauchy transform is defined by
\begin{equation}\label{Cauchy}
(Ch)(\lambda)=\frac{1}{2\pi i}\int_\Gamma \frac{h(z)}{z-\lambda}\,dz .
\end{equation}
The non-tangential limits of $Ch$ as $\lambda$ approaches $\Gamma$ from the
left and right exist and will be denoted by $C_+h$ and $C_-h$, respectively.
The operators $C_\pm$ are bounded on $L^2(\Gamma)$ and satisfy
\[
C_+-C_-=\mathds{1}, \quad  (C_+)^2=C_+,\quad  (C_-)^2=-C_-, \quad C_+C_-=C_-C_+=0.
\]

For functions $w^\pm\in L^2(\Gamma)\cap L^\infty(\Gamma)$ define the operator
$C_w:L^2(\Gamma)+L^\infty(\Gamma)\to L^2(\Gamma)$ by
\begin{equation}\label{C-w}
C_w(f)=C_+(fw^-)+C_-(fw^+).
\end{equation}
Moreover,
\begin{subequations}
\begin{alignat}{4}\label{C_w_ess_1}
||C_w||_{\mathcal{B}(L^2(\Gamma))}&\leq C \max\{||w^+||_{L^\infty}(\Gamma),||w^-||_{L^\infty}(\Gamma)\},\\\label{C_w_ess_2}
||C_w h||_{L^2(\Gamma)}&\leq C ||h||_{L^\infty(\Gamma)} \max\{||w^+||_{L^2(\Gamma)},||w^-||_{L^2(\Gamma)}\}
\end{alignat}
with $C=2\max \{||C_+||_{\mathcal{B}(L^2(\Gamma))},||C_-||_{\mathcal{B}(L^2(\Gamma))}\}$.
\end{subequations}
where
\[
C=2\max\{\|C_+\|_{\mathcal{B}(L^2(\Gamma))},\|C_-\|_{\mathcal{B}(L^2(\Gamma))}\}.
\]

The following lemma (\cites{L18}) establishes the equivalence between the
$L^2$ Riemann--Hilbert problem determined by $(\Gamma,v)$ and the singular
integral equation for $\mu\in I+L^2(\Gamma)$
\begin{equation}\label{mu_int_eq}
\mu-I=C_w(\mu).
\end{equation}

\begin{lemma}\label{Lem 5.2}
Given $v^\pm:\Gamma\to GL(n,\mathbb{C})$. Let $v=(v^-)^{-1}v^+$, $w^+=v^+-I$, $w^-=I-v^-$. Suppose $v^\pm$, $(v^\pm)^{-1}\in I + L^2(\Gamma)\cap L^\infty(\Gamma)$. If $m\in I + \dot E^2(D)$ satisfies $L^2-$RH problem determined by $(\Gamma,v(\lambda))$, then $\mu = m_+(v^+)^{-1}= m_-(v^-)^{-1}\in I+L^2(\Gamma)$ satisfies \eqref{mu_int_eq}. Conversely, if $\mu \in I+L^2(\Gamma)$ satisfies \eqref{mu_int_eq}, then $m=I+C(\mu(w^+ + w^-))\in I + \dot E^2(D)$ satisfies $L^2-$RH problem determined by $(\Gamma,v(\lambda))$.
\end{lemma}

The following lemma (\cites{L18}) shows that operator $\mathds{1}-C_w$ is Fredholm.

\begin{lemma}\label{Lem 5.3}
Given $v^\pm:\Gamma\to GL(n,\mathbb{C})$. Let $v=(v^-)^{-1}v^+$, $w^+=v^+-I$, $w^-=I-v^-$. Suppose $v^\pm$, $(v^\pm)^{-1}\in I + L^2(\Gamma)\cap L^\infty(\Gamma)$, and $v^\pm$ is piecewise continuous on $\Gamma$. Then
\begin{enumerate}[(1)]
    \item The operator $\mathds{1}-C_w:L^2(\Gamma)\to L^2(\Gamma)$ is Fredholm.
    \item If $w^\pm$ are nilpotent, then $\mathds{1}-C_w$ has Fredholm index $0$; in this case TFAE:
    \begin{enumerate}[(a)]
        \item  The map $\mathds{1}-C_w:L^2(\Gamma)\to L^2(\Gamma)$ is bijective.
        \item The $L^2-$RH problem determined by $(\Gamma,v(\lambda))$ has a unique solution.
        \item The homogeneous $L^2-$RH problem determined by $(\Gamma,v(\lambda))$ has only the zero solution.
        \item The map $\mathds{1}-C_w:L^2(\Gamma)\to L^2(\Gamma)$ is injective.
    \end{enumerate}
\end{enumerate}
\end{lemma}

\subsection{Problem I}\label{sec:7.2}
We assume that $a(k)$ has no zeros in $\overline{\mathbb{C}^+}$. Zeros in the open upper half-plane can be eliminated by introducing small circular contours and converting the residue conditions into jump conditions; zeros on $\mathbb{R}$ could be eliminated by introducing small half-circular contours.

First, notice that the jump matrix \eqref{jump_M_(xt)_matr} can be written in the form suitable for Lemma \ref{Lem 5.2} and \ref{Lem 5.3}. Indeed, $[J^{(xt)}]^{(-1)}(x,t,k)=(v^-)^{-1}v^+$ with
\begin{equation}\label{v+}
v^-(x,t,k)=\begin{cases}
       \begin{pmatrix}
       1&-r^*(k)\eul^{-2\ii kp(x,t,k)}\\0&1
   \end{pmatrix},\quad k\in\mathbb{R}\cap\{|k| >\epsilon\},\\
   \begin{pmatrix}
       1&0\\\tilde r(k)\eul^{2\ii kp(x,t,k)}&1
   \end{pmatrix},\quad k\in\mathbb{R}\cap\{|k| <\epsilon\},\\
   I,\quad k\in\mathbb{C}^+\cap\{|k| =\epsilon\},\\
   \begin{pmatrix}
       1&-\frac{\kappa_2^0\eul^{-2\ii kp(x,t,k)}}{a^*(k)\tilde a^*(k)}\\
       0&1
   \end{pmatrix},\quad k\in\mathbb{C}^-\cap\{|k| =\epsilon\}
   \end{cases} 
\end{equation}
and
\begin{equation}\label{v-}
v^+(x,t,k)=\begin{cases}
       \begin{pmatrix}
       1&0\\r(k)\eul^{2\ii kp(x,t,k)}&1
   \end{pmatrix},\quad k\in\mathbb{R}\cap\{|k| >\epsilon\},\\
   \begin{pmatrix}
       1&-\tilde r^*(k)\eul^{-2\ii kp(x,t,k)}\\0&1
   \end{pmatrix},\quad k\in\mathbb{R}\cap\{|k| <\epsilon\},\\
   \begin{pmatrix}
1&0\\\frac{\kappa_2^0\eul^{2\ii kp(x,t,k)}}{a(k)\tilde a(k)}&1
   \end{pmatrix},\quad k\in\mathbb{C}^+\cap\{|k| =\epsilon\},\\
   I,\quad k\in\mathbb{C}^-\cap\{|k| =\epsilon\}.
   \end{cases} 
\end{equation}
Moreover, from the considerations in Section \ref{sec:2.1}, it follows that $r(k)$,  $\tilde r(k)$ and $\frac{1}{a(k)\tilde a(k)}$ satisfy the following properties:
\begin{align}\label{prop_r_1}
  &r(k)=O\left(\frac{1}{k}\right),\quad k\in\mathbb{R},~  k\to\infty;\quad  r(k)\in C(\mathbb{R}),\\
  \label{prop_r_2}
  &\tilde r(k)=O(k) ,\quad k\in\mathbb{R},~ k\to 0;\quad  r(k)\in C(\mathbb{R}),\\
\label{prop_r_3}
  &\frac{1}{a(k)\tilde a(k)}\in C(\{|k|=\epsilon\}),\quad k\in\mathbb{C}^+.
\end{align}
Therefore, using $\det v^\pm=1$, we conclude that
\[
v^\pm,v_\pm^{-1}\in I + L^2(\Sigma)\cap L^\infty(\Sigma).
\]
Hence, Lemma \ref{Lem 5.2} applies and yields the equivalence between the
RH problem \textbf{RH$^{(xt)}$} and the associated singular integral equation.

Accordingly, $w^\pm$ are given by
\begin{equation}\label{w+}
w^-(x,t,k)=\begin{cases}
       \begin{pmatrix}
       0&-r^*(k)\eul^{-2\ii kp(x,t,k)}\\0&0
   \end{pmatrix},\quad k\in\mathbb{R}\cap\{|k| >\epsilon\},\\
   \begin{pmatrix}
       0&0\\\tilde r(k)\eul^{2\ii kp(x,t,k)}&0
   \end{pmatrix},\quad k\in\mathbb{R}\cap\{|k| <\epsilon\},\\
   0,\quad k\in\mathbb{C}^+\cap\{|k| =\epsilon\},\\
   \begin{pmatrix}
       0&-\frac{\kappa_2^0\eul^{-2\ii kp(x,t,k)}}{a^*(k)\tilde a^*(k)}\\
       0&0
   \end{pmatrix},\quad k\in\mathbb{C}^-\cap\{|k| =\epsilon\}
   \end{cases} 
\end{equation}
and
\begin{equation}\label{w-}
w^+(x,t,k)=\begin{cases}
       \begin{pmatrix}
       0&0\\r(k)\eul^{2\ii kp(x,t,k)}&0
   \end{pmatrix},\quad k\in\mathbb{R}\cap\{|k| >\epsilon\},\\
   \begin{pmatrix}
       0&-\tilde r^*(k)\eul^{2\ii kp(x,t,k)}\\0&0
   \end{pmatrix},\quad k\in\mathbb{R}\cap\{|k| <\epsilon\},\\
   \begin{pmatrix}
0&0\\\frac{\kappa_2^0\eul^{2\ii kp(x,t,k)}}{a(k)\tilde a(k)}&0
   \end{pmatrix},\quad k\in\mathbb{C}^+\cap\{|k| =\epsilon\},\\
   0,\quad k\in\mathbb{C}^-\cap\{|k| =\epsilon\},   \end{cases} 
\end{equation}
and since $w^\pm$ are nilpotent, we can apply Lemma \ref{Lem 5.3}, which together with Zhou's vanishing lemma ensures the solvability of the
RH problem \textbf{RH$^{(xt)}$}. In particular,
\begin{equation}\label{mu_sol}
 \mu^{(xt)}(x,t,k)=I+(\mathds{1}-C_w)^{-1}(x,t)C_w(x,t)I
\end{equation}
and the solution $M^{(xt)}(x,t,k)$ of the Riemann--Hilbert problem \textbf{RH$^{(xt)}$} can be written in terms of $\mu^{(xt)}(x,t,k)$ as 
\begin{equation}\label{M_via_mu}
    {M^{(xt)}}(x,t,k)=I+\frac{1}{2\pi\ii}\int_\Sigma \frac{\mu^{(xt)}(x,t,z)(w^+ + w^-)(x,t,z)}{z-k}dz.
\end{equation}

Notice that
\begin{equation}\label{C_w}
   C_w(x,t)= C_{w(x,t)}
\end{equation}
and $C_w$ is linear in $w$.

To study the differentiability with respect to $x$ and $t$, it is convenient to
modify the RH problem so that its jump matrix approaches the identity
sufficiently fast as $k\to\infty$, at least as $O\left(\frac{1}{k^2}\right)$. The jump matrix of the original problem decays only as $r(k)=O\left(\frac{1}{k}\right)$, and is of order $O(k)$ as $k\to 0$. To improve this behavior, we introduce two functions (c.f. \cite{L12})
\[
h(k):=\frac{r_1}{4\ii k},\quad \tilde h(k)=\ii\tilde r_1 k
\]
where $r_1$ and $\tilde r_1$ are determined by the expansions $r(k)=\frac{r_1}{4\ii k}+O\left(\frac{1}{k^2}\right)$ and $\tilde r(k)=\ii \tilde r_1 k+O(k^2)$. Notice that
\begin{equation}\label{prop_h_1}
 \overline{h(-\bar k)}=h(k),\quad h(k)\in C(\{|k|=\epsilon\}), \quad h(k)=O\left(\frac{1}{k}\right),\quad k\to\infty.   
\end{equation}
\begin{equation}\label{prop_tilh_1}
 \overline{\tilde h(-\bar k)}=\tilde h(k),\quad \tilde h(k)\in C(\{|k|=\epsilon\}), \quad \tilde h(k)=O\left(k\right),\quad k\to 0.   
\end{equation}

Moreover, by the definition of $h(k)$, we have
\begin{equation}\label{prop_h_2}
 h(k)-r(k)=O\left(\frac{1}{k^2}\right), \quad k\to \infty;\quad  \tilde h(k)-\tilde r(k)=O\left(k^2\right), \quad k\to 0.   
\end{equation}

We then introduce a matrix-valued function $\hat M(x,t,k)$ related to
$M(x,t,k)$ by
\begin{equation}\label{hat_M}
   \hat M(x,t,k)=\begin{cases}
        M(x,t,k)\eul^{-\ii k p(x,t,k)\sigma_3}\begin{pmatrix}
            1&0\\
        -h(k)&1
   \end{pmatrix}\eul^{\ii k p(x,t,k)\sigma_3},\quad \{|k|>\epsilon\}\cap\mathbb{C}_+,\\
         M(x,t,k)\eul^{-\ii k p(x,t,k)\sigma_3}\begin{pmatrix}
            1&h^*(k)\\
        0&1
\end{pmatrix}\eul^{\ii k p(x,t,k)\sigma_3},\quad \{|k|>\epsilon\}\cap\mathbb{C}_-,\\
             M(x,t,k)\eul^{-\ii k p(x,t,k)\sigma_3}\begin{pmatrix}
            1&0\\
        -\tilde h(k)&1
   \end{pmatrix}\eul^{\ii k p(x,t,k)\sigma_3},\quad \{|k|<\epsilon\}\cap\mathbb{C}_+,\\
         M(x,t,k)\eul^{-\ii k p(x,t,k)\sigma_3}\begin{pmatrix}
            1&\tilde h^*(k)\\
        0&1
\end{pmatrix}\eul^{\ii k p(x,t,k)\sigma_3},\quad \{|k|<\epsilon\}\cap\mathbb{C}_-
    \end{cases}  
    \end{equation}
and conclude that it satisfies the following RH problem:
\begin{enumerate}

    \item Jump condition across $\mathbb{R}\cup\{|k|=\epsilon\}$
    \begin{subequations}
        \label{jump_M_(xt)_hat}
        \begin{equation}
          \hat M_-^{(xt)}(x,t,k)=\hat M_+^{(xt)}(x,t,k)\hat J^{(xt)}(x,t,k),\quad k\in \mathbb{R}\cup\{|k|=\epsilon\} 
        \end{equation}
        where
         \begin{equation} \label{jump_M_(xt)_matr_hat}
        \hat  J^{(xt)}(x,t,k)=\eul^{-\ii kp(x,t,k)\sigma_3} \hat J^{(xt)}_0(k)\eul^{\ii kp(x,t,k)\sigma_3}
        \end{equation} 
        with $p(x,t,k)=x-\frac{t}{4 k^2}$ and
\begin{equation}\label{jump_M_(xt)_matr_0_hat}
\hat J^{(xt)}_0(k)=\begin{cases}
      \begin{pmatrix}
       1&0\\h(k)-r(k)&1
   \end{pmatrix} \begin{pmatrix}
       1&h^*(k)-r^*(k)\\0&1
   \end{pmatrix},\quad k\in\mathbb{R}\cap\{|k| >\epsilon\},\\
   \begin{pmatrix}
       1&\tilde r^*(k)-\tilde h^*(k)\\0&1
\end{pmatrix}\begin{pmatrix}
       1&0\\\tilde r(k)-\tilde h(k)&1
   \end{pmatrix},\quad k\in\mathbb{R}\cap\{|k| <\epsilon\},\\
   \begin{pmatrix}
       1&0\\
       h(k)-\tilde h(k)-\frac{\kappa_2^0}{a(k)\tilde a(k)}&1
   \end{pmatrix},\quad k\in\mathbb{C}^+\cap\{|k| =\epsilon\},\\
   \begin{pmatrix}
       1&h^*(k)-\tilde h^*(k)-\frac{\kappa_2^0}{a^*(k)\tilde a^*(k)}\\
       0&1
\end{pmatrix},\quad k\in\mathbb{C}^-\cap\{|k| =\epsilon\}
   \end{cases} 
\end{equation}
and $r(k)=\frac{b^*(k)}{a(k)}$, $\tilde r(k)=\frac{\tilde b^*(k)}{\tilde a(k)}$.
    \end{subequations}

\item Behavior at $\infty$:
\begin{equation}\label{inf_hatM_(xt)_hat}
      \hat M^{(xt)}(x, t,k)=I+O\left(\frac{1}{ k}\right),\quad k\to\infty.
\end{equation}
\end{enumerate}
All the considerations above apply to this RH problem as well. In particular,

\begin{equation}\label{mu_sol_hat}
 \hat \mu^{(xt)}(x,t,k)=I+(\mathds{1}-C_{\hat w})^{-1}(x,t)C_{\hat w}(x,t)I
\end{equation}
and
\begin{equation}\label{hat_M_via_mu}
    {\hat M^{(xt)}}(x,t,k)=I+\frac{1}{2\pi\ii}\int_\Sigma \frac{\hat\mu^{(xt)}(x,t,z)(\hat w^+ + \hat w^-)(x,t,z)}{z-k}dz.
\end{equation}
Here $\hat w_\pm$ are defined as in \eqref{w-} and \ref{w+}, with $r(k)$, $\tilde r(k)$, and $\frac{\kappa_2^0}{a(k)\tilde a(k)}$ replaced by $r(k)-h(k)$, $\tilde r(k)-\tilde h(k)$, and $\frac{\kappa_2^0}{a(k)\tilde a(k)}-h(k)+\tilde h(k)$, respectively.

The differentiability of $M^{(xt)}(x,t,k)$ then follows from the
differentiability of $\hat M^{(xt)}(x,t,k)$ established in

\begin{proposition}\label{Prop 5.6}
 The function $ \mu^{(xt)}(x,t,k)$ is $C^1$ with respect to $x$ and $t$. 
\end{proposition}

\begin{proof} Notice that  $\hat \mu^{(xt)}(x,t,k)$ is a composition of the following maps:
\begin{align*}
&(x,t)\mapsto C_{\hat w(x,t)}I: (\mathbb{R}^+,[0,T))\to L^2(\Sigma);\\
&(x,t)\mapsto (\mathds{1}-C_{\hat w(x,t)})^{-1}:(\mathbb{R}^+,[0,T))\to \mathcal{B}(L^2(\Sigma)).
\end{align*}
In turn, $(x,t)\mapsto C_{\hat w(x,t)}I: (\mathbb{R},[0,\infty)\to L^2(\Sigma)$ is a composition of 
\begin{align*}
    &(x,t)\mapsto \hat w^\pm(x,t):(\mathbb{R}^+,[0,T))\to L^2(\Sigma);\\
    &Q:(\hat w^+,\hat w^-)\mapsto C_{\hat w }I:L^2(\Sigma)\times L^2(\Sigma) \to L^2(\Sigma)
\end{align*}
and $(x,t)\mapsto (\mathds{1}-C_{\hat w(x,t)})^{-1}:(\mathbb{R}^+,[0,T))\to \mathcal{B}(L^2(\Sigma))$ is a composition of
\begin{align*}
    &(x,t)\mapsto \hat w^\pm(x,t):(\mathbb{R}^+,[0,T))\to L^\infty(\Sigma);\\
    &F:(\hat w^+,\hat w^-)\mapsto  (\mathds{1}-C_{\hat w})^{-1}:L^\infty(\Sigma)\times L^\infty(\Sigma)\to \mathcal{B}(L^2(\Sigma)).
\end{align*}

Differentiating $\eul^{\ii kp(x,t,k)}$ we get
\[
\frac{d}{dx} (\eul^{2 \ii kp(x,t,k)})=2\ii k\eul^{2 \ii kp(x,t,k)}, \quad \frac{d}{dt} (\eul^{2 \ii kp(x,t,k)})=\frac{1}{2\ii k}\eul^{2 \ii kp(x,t,k)}. 
\]
Combining \eqref{prop_r_1}--\eqref{prop_r_3} with \eqref{prop_h_1}--\eqref{prop_h_2}, we conclude that map $(x,t)\mapsto \hat w^\pm(x,t):(\mathbb{R}^+,[0,T))\to L^\infty(\Sigma)$ is differentiable w.r.t. both $x$ and $t$. On the other hand, these properties together with the dominated convergence theorem yield the differentiability of $(x,t)\mapsto \hat w^\pm(x,t):(\mathbb{R}^+,[0,T))\to L^2(\Sigma)$.

The maps $F$ and $Q$ are linear, and estimates \eqref{C_w_ess_1} and \eqref{C_w_ess_2} imply their differentiability. Moreover, $dF:(u^+,u^-)\mapsto -C_u:L^\infty(\Sigma)\times L^\infty(\Sigma)\to \mathcal{B}(L^2(\Sigma))$ and $dQ:(u^+,u^-)\mapsto C_u I:L^2(\Sigma)\times L^2(\Sigma)\to L^2(\Sigma)$. Hence, by the chain rule (see \cites{TFA}), the map $(x,t)\mapsto (\mathds{1}-C_{\hat w(x,t)})^{-1}:(\mathbb{R}^+,[0,T))\to \mathcal{B}(L^2(\Sigma))$ is differentiable together with its inverse. The chain rule again implies differentiability of the map $(x,t)\mapsto C_{\hat w(x,t)}I: (\mathbb{R}^+,[0,T))\to L^2(\Sigma)$. Consequently, $\hat \mu^{(xt)}(x,t,k)$ is differentiable as the composition of differentiable maps.
\end{proof}

Notice that properties \eqref{prop_h_1} and \eqref{prop_tilh_1} imply that the first term in the expansion of $\hat M(x,t,k)$ as $k \to 0$, as well as the second term in the expansion as $k \to \infty$, coincide with those of $M(x,t,k)$. 
Thus, Proposition~\ref{Prop 5.6} together with the dominated convergence theorem implies the following corollary:

\begin{corollary}
    The functions $\alpha(x,t)$, $\beta(x,t)$ and $\xi(x,t)$ are $C^1$ with respect to $x$ and $t$.
\end{corollary}

\subsection{Problem II}\label{sec:7.3}
The idea of the proof that $\check M^{(xt)}(x,t,k)$ is differentiable with respect to $x$ and $t$ is essentially the same as in the case of Problem~I. Therefore, we only outline the necessary modifications.

In the case of Problem II, we assume that $d^*(k)$ has no zeros in $\overline{\mathbb{C}^+}$:  zeros in the open upper half-plane can be eliminated by introducing small circular contours and converting the residue conditions into jump conditions, zeros on $\mathbb{R}$ can be eliminated by introducing small half-circular contours.

We then introduce the matrices $v_\pm$ and rewrite the jump matrix in a form suitable for Lemmas~\ref{Lem 5.2} and~\ref{Lem 5.3}. Namely,  $[\check J^{(xt)}]^{(-1)}(x,t,k)=(v^-)^{-1}v^+$ with

\begin{equation}\label{v+_check}
v^-(x,t,k)=\begin{cases}
       \begin{pmatrix}
       1&0\\-\rho^*(k)\eul^{2\ii kp(x,t,k)}&1
   \end{pmatrix},\quad k\in\mathbb{R}\cap\{|k| >\epsilon\},\\
   \begin{pmatrix}
       1&\tilde\rho(k)\eul^{-2\ii kp(x,t,k)}\\0&1
   \end{pmatrix},\quad k\in\mathbb{R}\cap\{|k| <\epsilon\},\\
   I,\quad k\in\mathbb{C}^+\cap\{|k| =\epsilon\},\\
  \begin{pmatrix}
          {1}&0\\
\Xi^*(k)\eul^{2\ii kp(x,t,k)}&{1}
       \end{pmatrix},\quad k\in\mathbb{C}^-\cap\{|k| =\epsilon\}
   \end{cases} 
\end{equation}
and
\begin{equation}\label{v-_check}
v^+(x,t,k)=\begin{cases}
       \begin{pmatrix}
       1&\rho(k)\eul^{-2\ii kp(x,t,k)}\\0&1
   \end{pmatrix},\quad k\in\mathbb{R}\cap\{|k| >\epsilon\},\\
   \begin{pmatrix}
       1&0\\-\rho^*(k)\eul^{2\ii kp(x,t,k)}&1
   \end{pmatrix},\quad k\in\mathbb{R}\cap\{|k| <\epsilon\},\\
  \begin{pmatrix}
{1}&\Xi(k)\eul^{-2\ii kp(x,t,k)}\\
          0&{1}
       \end{pmatrix},\quad k\in\mathbb{C}^+\cap\{|k| =\epsilon\},\\
   I,\quad k\in\mathbb{C}^-\cap\{|k| =\epsilon\}.
   \end{cases} 
\end{equation}
Moreover, from the considerations in Section \ref{sec:2.2}, it follows that
\begin{align}\label{prop_r_1_check}
  &\rho(k)=O\left(\frac{1}{k}\right),\quad k\in\mathbb{R},~   k\to\infty;\quad  \rho(k)\in C(\mathbb{R}),\\
  \label{prop_r_2_check}
  &\tilde \rho(k)=O(k),\quad k\in\mathbb{R},~   k\to 0;\quad  \tilde \rho(k)\in C(\mathbb{R}),\\
\label{prop_r_3_check}
  &\Xi(k)\in C(\{|k|=\epsilon\}), \quad  k\in\mathbb{C}^+.
\end{align}
Further, we proceed in the same way as in the case of Problem I.

\section{The uniqueness and existence results for Problem I}\label{sec:8}

\subsection{Uniqueness result for Problem I}\label{sec:8.1}

In this section, assuming the existence of a solution to Problem~I (i.e., with an appropriate behavior as $x \to +\infty$), we address its uniqueness.

\begin{proposition}\label{uniq_I}
Let $ u(x,t)$ be a solution of the sG equation in the domain $x>0$, $0<t<T$
such that $u(\cdot,t)\in C^1([0,T],\tilde H^{1,2}((0,\infty)))$. 
Then $u(x,t)$ is uniquely determined by $ u(x,0)$, $x\in[0,\infty)$ through a unique solution of the Riemann--Hilbert problem, for which the data (jump matrix and residue condition) are given in terms of the spectral functions associated with $ u(x,0)$. 
    
\end{proposition}

\begin{proof} The uniqueness of the solution to the RH problem implies, using \eqref{i_beh-M_(xt)}, the following procedure for representing
$u(x,t)$ in terms of $ M^{(xt)}(x,t,k)$:

  \begin{enumerate}[Step 1.]

  \item Given $u(x,0)$ construct $a(k)$, $b(k)$ (see Section \ref{subsec:2.1.4}).
  
    \item Having $a(k)$ and $b(k)$, compute $\kappa_1^0$ and $\kappa_2^0$, and determine $\tilde a(k)$ and $\tilde b(k)$ via  \eqref{s_via_til_s}. With these data, formulate the Riemann--Hilbert problem \textbf{RH$^{(xt)}$}, see Section \ref{subsec:5.1}.
    \item 
    Solve this RH problem for all $x\ge 0$ and $t\in[0,T]$
    and 
evaluate its  solution $ M^{(xt)}(x,t,k)$  at $k=\infty$:
\begin{equation}\label{hatM_xt__inf}
        M^{(xt)}(x,t,k)= I+\frac{1}{4\ii k}\begin{pmatrix}
           -\eta(x,t)&\xi(x,t)\\\xi(x,t)&\eta(x,t)
       \end{pmatrix}+O(\frac{1}{k^2}). 
    \end{equation}
    \item 
    Define $ u_x(x,t)$, $x\ge 0$ and $t\in[0,T]$, from this expansion as follows (cf. \eqref{inf_M_(xt)}):
    \begin{equation}\label{sol_via_RH_xt}
      u_x(x,t)=\xi(x,t).
    \end{equation}
 \item  Having $u_{x}(x,t)$, construct $u(x,t)$ by
\begin{equation}\label{Fund_u_xt}
   u(x,t) = -\int_x^{+\infty} u_{x}(y,t)\dd y+2\pi k,
\end{equation}
where $k$ is defined by the behavior of $u(x,0)$ as $x\to+\infty$.  
    
\end{enumerate}

\end{proof}

\subsection{Well-posedness of Problem I: mapping of RH problems}\label{sec:8.2}
The fact that the initial data on a half-line $x\in [l,+\infty)$ determine uniquely the solution of the sG equation in the corresponding
quarter-plane $x\ge l$, $t\ge 0$ suggests that the RH problems associated with two initial data, 
$u_1(x,0)$ for $x\in [l_1,+\infty)$ and $u_2(x,0)$ for 
$x\in [l_2,+\infty)$ with $l_2<l_1$ such that 
$u_2(x,0)=u_1(x,0)$ for all $x\in [l_1,+\infty)$,
have to be related for $x\in [l_1,+\infty)$, $t\ge 0$.

First, assume that $u_1(x,t)$ is a solution of the sG equation on $x\in[l_1,+\infty)$ and $u_2(x,t)$ is a solution of the sG equation on $x\in[l_2,+\infty)$ and establish a useful relation amongst the spectral functions of the problems, in the framework of the direct problems. In analogy with \eqref{p}, 
introduce 
\begin{equation}\label{p-j}
p_j(x,t,k):=x-l_j-\frac{t}{4 k^2},\qquad j=1,2
\end{equation}
and notice that
\begin{equation}\label{p2-p1}
p_2(x,t,k)-p_1(x,t,k) = 
l_2-l_1=:\mu>0.
\end{equation}
Introduce the Jost solutions $\Phi_{ 2}$, $\Phi_{ 3}$, $\hat \Phi_{ 2}$ and $\hat \Phi_{ 3}$
as follows: $\Phi_{ 3}$ and $\hat\Phi_{ 3}$ are introduced by \eqref{inteq_inf3}, where $U_\infty$ is given by \eqref{Lax-U_infty} with $u_x$ replaced by $u_{1x}$ and $u_{2x}$,respectively,
whereas 
let $\Phi_{ 2}$ and $\check \Phi_{2}$ be the analogues of 
$\Phi_{ 2}$ introduced by \eqref{inteq_inf2}, taking into account the left ends of the spatial intervals:
\begin{subequations}\label{inteq_inf2-4}
    \begin{align}
\Phi_{ 2}(x,t,k)=I & +
\int_{l_1}^{x}
	\eul^{-\ii k(x-y)\hat\sigma_3}( U_\infty\Phi_{2})(y,t,k) \dd y  \nonumber \\
&+\eul^{-\ii k(x-l_1)\hat\sigma_3}
\int_{0}^{t}
	\eul^{\frac{\tau-t}{4 \ii k}\hat\sigma_3}( V_\infty\Phi_{2})(l_1,\tau,k) \dd \tau,\\
\hat\Phi_{ 2}(x,t,k)=I & +
\int_{l_2}^{x}
	\eul^{-\ii k(x-y)\hat\sigma_3}( U_\infty\hat\Phi_{2})(y,t,k) \dd y  \nonumber \\
&+\eul^{-\ii k(x-l_2)\hat\sigma_3}
\int_{0}^{t}
	\eul^{\frac{\tau-t}{4 \ii k}\hat\sigma_3}( V_\infty\hat\Phi_{2})(l_2,\tau,k) \dd \tau,    
\end{align}
\end{subequations}
where $U_\infty$ and $V_\infty$ are given by \eqref{Lax-U_infty} and \eqref{Lax-V_infty} with $u$ replaced by $u_{1}$ and $u_{2}$, respectively.

Introduce the scattering matrices $s^{(j)}(k)$, $j=1,2$
(with the corresponding entries $a^{(j)}(k)$ and $b^{(j)}(k)$)
associated with the initial data 
$\{u_j(x,0), x\in [l_j,+\infty)\}$, which are determined by the scattering 
relations similar to \eqref{rel_inf_1}:
\begin{subequations}\label{rel_inf-234}
    \begin{align}\label{rel_inf_32}
        &\Phi_{3}(x,t,k)=\Phi_{2}(x,t,k)
        \eul^{-\ii k p_1(x,t,k)\sigma_3}s^{(1)}(k)\eul^{\ii k p_1(x,t,k)\sigma_3},\\\label{rel_inf_34}
        &\hat\Phi_{ 3}(x,t,k)=\hat\Phi_{ 2}(x,t,k)\eul^{-\ii k p_2(x,t,k)\sigma_3}s^{(2)}(k)\eul^{\ii k p_2(x,t,k)\sigma_3},
    \end{align}
\end{subequations}
where 
\begin{subequations}
            \label{s1}
            \begin{equation}
s^{(1)}(k)=\Phi_{ 3}(l_1,0,k)=:\begin{pmatrix}
        a^{(1)*}(k) & b^{(1)}(k)\\
        -b^{(1)*}( k) & a^{(1)}(k)
    \end{pmatrix}
    \end{equation}
        and 
    \begin{equation}
            \label{s2}
s^{(2)}(k)=\hat\Phi_{ 3}(l_2,0,k)=:\begin{pmatrix}
        a^{(2)*}(k) & b^{(2)}(k)\\
        -b^{(2)*}( k) & a^{(2)}(k)
    \end{pmatrix}.
    \end{equation}
    \end{subequations}
It follows that for $t=0$ and $x\geq l_1$, $\Phi_{ 2}$ and $\hat\Phi_{ 2}$ can be related 
(taking into account \eqref{p2-p1}) by
\begin{equation}\label{rel-42}
\hat\Phi_{ 2}(x,0,k) = \Phi_{2}(x,0,k)
\eul^{-\ii k p_1(x,0,k)\sigma_3}s^{(1)}(k)
\eul^{-\ii k \mu \sigma_3}[s^{(2)}]^{-1}(k)
\eul^{\ii k p_2(x,0,k)\sigma_3}.
\end{equation}

Now, we notice that by construction, the first column of 
$\hat\Phi_{\infty 2}(x,t,k)$ is analytic in $\mathbb {C}^+$ and, moreover, its (21) entry
is $O(1/k)$ as $k\to \infty$. On the other hand, the (21) entry of the
r.h.s. of \eqref{rel-42} evaluated at $x=l_1$ reads in terms 
$a^{(j)}$ and $b^{(j)}$ as 
$-b^{(1)*}(k)a^{(2)}(k) + b^{(2)*}(k)a^{(1)}(k)\eul^{2\ii k \mu }$. Consequently, we have 
\begin{equation}\label{gr-ini}
    -b^{(1)*}(k)a^{(2)}(k) + b^{(2)*}(k)a^{(1)}(k)\eul^{2\ii k\mu } = 
    O\left(\frac{1}{k}\right), \qquad k\to \infty, \quad 
    k\in {\mathbb C}^+.
\end{equation}
\begin{remark}\label{rem:gr-ini}
    In general, the spectral functions $b^{(j)*}(k)$ are determined  
    for $k\in {\mathbb C}^-$, but their combination as in the l.h.s.
    of \eqref{gr-ini} turns out to be analytically continued into 
    ${\mathbb C}^+$.
\end{remark}

\begin{proposition}
    Let $M^{(2)}(x,t,k)$, $x\in [l_2, +\infty)$, $t\ge 0$ satisfy the RH problem \textbf{RH$^{(xt)}$}
    with $p(x,t,k)$ replaced by $p_2(x,t,k)$ 
    and the spectral functions $a^{(2)}(k)$, $b^{(2)}(k)$
associated with the initial data $u(x,0)$, $x\in [l_2, +\infty)$. For $l_1>l_2$,
    define $\tilde M^{(1)}(x,t,k)$ for $x\in [l_1, +\infty)$, $t\ge 0$ as 
    follows:
    \begin{equation}
     \tilde M^{(1)}(x,t,k) := M^{(2)}(x,t,k) \begin{cases}
         \begin{pmatrix}
             1 &   0 \\
             X^{(1)}(k)\eul^{2\ii k  p_1(x,t,k)} - 
             X^{(2)}(k)\eul^{2\ii k p_2(x,t,k)} & 1
         \end{pmatrix}, & k\in {\mathbb C}^+, \\
         \begin{pmatrix}
             1 &  X^{(2)*}(k)\eul^{-2\ii k p_2(x,t,k)} - 
             X^{(1)*}(k)\eul^{-2\ii k  p_1(x,t,k)} \\
             0 & 1
         \end{pmatrix}, & k\in {\mathbb C}^-, \\
              \end{cases}
    \end{equation}
    where 
    \begin{equation}
       X^{(j)}(k) = \begin{cases}
        r^{(j)}(k), &   k\in {\mathbb C}^+\cap \{|k|>\epsilon\} \\
        r^{(j)}(k) - \frac{\kappa_2^{(0j)}}{a^{(j)}(k) \tilde a^{(j)}(k)}, & k\in {\mathbb C}^+\cap \{|k|<\epsilon\}
       \end{cases}
    \end{equation}
    with $r^{j}(k)=\frac{b^{(j)*}(k)}{a^{(j)}(k)}$ and 
    $\kappa_2^{(0j)}=b^{(j)}(0)$.
    
    Then for $x\in [l_1, +\infty)$ and $t\ge 0$,
    $\tilde M^{(1)}(x,t,k) = M^{(1)}(x,t,k)$, where $M^{(1)}(x,t,k)$  satisfies the RH problem \textbf{RH$^{(xt)}$} associated with the initial data $u(x,0)$ restricted to $x\in [l_1, +\infty)$. 
   
\end{proposition}

\begin{proof}
    In view of the construction of RH problems and the 
    way of extracting the solution of the sG equation from the 
    solution of the associated RH problem, it is sufficient to show that for $k\in {\mathbb C}^+$ and for all $x\ge l_1$, $t\ge 0$,
    \begin{enumerate}[(i)]
        \item 
    $X^{(1)}(k)\eul^{2\ii k  p_1(x,t,k)} - 
             X^{(2)}(k)\eul^{2\ii k p_2(x,t,k)} = O(1/k) $ as $k\to \infty$;
\item      $X^{(1)}(k)\eul^{2\ii k p_1(x,t,k)} - 
             X^{(2)}(k)\eul^{2\ii k p_2(x,t,k)} = O(k) $ as $k\to 0$.
\end{enumerate}

Regarding item (i), notice that for $k\in {\mathbb C}^+\cap \{|k|>\epsilon\}$,
\[
X^{(1)}(k)\eul^{2\ii k  p_1(x,t,k)} - 
             X^{(2)}(k)\eul^{2\ii k p_2(x,t,k)} = 
             \frac{e^{2\ii k  p_1(x,t,k)}}{a^{(1)}(k)a^{(2)}(k)}
             \left(b^{(1)*}(k)a^{(2)}(k) - b^{(2)*}(k)a^{(1)}(k)\eul^{2\ii k\mu }\right)
\]
and thus, taking into account  \eqref{gr-ini},
the analyticity of 
$a^{(j)}(k)$ in ${\mathbb C}^+$ with $a^{(j)}(k)=1+O\left(\frac{1}{k}\right)$, and that $\Im p_1(x,t,k)\ge 0$ for all $x\ge l_1$, $t\ge 0$, the statement of (i) follows.

Regarding item (ii), we notice that the estimate $X^{(1)}(k)\eul^{2\ii k  p_1(x,t,k)} - 
             X^{(2)}(k)\eul^{2\ii k p_2(x,t,k)}=O(k)$
follow from \eqref{a_at_i}, the estimates $\tilde a^{(j)}(k)
= 1 + O(k)$
as $k\to 0$, and the fact that $X^{(1)}(k)\eul^{2\ii k p_1(x,t,k)} - 
             X^{(2)}(k)\eul^{2\ii k p_2(x,t,k)}$ is analytic in ${\mathbb C}^+$ including near $k=0$ 
(see Remark \ref{rem:gr-ini}).

\end{proof}

\subsection{Behavior as $x\to+\infty$}\label{sec:8.3}

 In Sections~\ref{sec:8.1} and \ref{sec:8.2}, we assumed the existence of a solution to Problem~I for the sine--Gordon equation (with prescribed behavior as $x \to +\infty$) and constructed the associated Riemann--Hilbert problem from the corresponding Jost solutions. In the present and the following section, we drop this assumption and instead analyze the Riemann--Hilbert problem constructed from initial data belonging to a suitable class. The differentiability of the ingredients extracted from this Riemann--Hilbert problem, which are used in the reconstruction of the solution to the sine--Gordon equation, was established in Section~\ref{sec:7.2}. 
In this section, we focus on the behavior as $x \to +\infty$.
To fix ideas,
we assume that $a(k)$ has no zeros in $\overline{\mathbb{C}^+}$.   .

First, notice that the considerations in Section \ref{sec:8.2} imply that extending the initial data to the left does not affect the solution of the problem on the common half-line. In particular, one may extend the domain $(0,+\infty)$ to $(-\infty,+\infty)$, thus reducing Problem I to a problem on the whole real line. Consequently, without loss of generality, we can assume that $\kappa_2 = 0$, $r(k)=O\left(\frac{1}{k^2}\right)$ as $k\to\infty$ (i.e.  $r(k), ~kr(k) \in L^{2}((0,\infty))$), and $r(k) = O(k^3)$ as $k \to 0$. To ensure the last condition, we extend the initial function $u_0(x)$ from $(0,+\infty)$ to $(-\infty,+\infty)$ in such a way that $\int_{-\infty}^{\infty} \sin u_0(y)\dd y=0$ and $\int_{-\infty}^{\infty} y\sin u_0(y)\dd y=0$.

The dominated convergence theorem, together with the properties of $r(k)$, implies that we can pass to the (non-tangential) limits $k\to \infty$ and $k\to 0$ in \eqref{M_via_mu}. In particular, we have
\begin{equation}\label{xi_via_mu}
\begin{aligned}
 \xi(x,t)=&\frac{2}{\pi}\int_{\mathbb{R}} \mu_{11}^{(xt)}(x,t,z) r^*(z)\eul^{-2\ii z x+\frac{\ii t}{2z}} \dd z,
\end{aligned}   
\end{equation}
\\
\begin{equation}\label{beta_via_mu}
\begin{aligned}
 \beta(x,t)=&\frac{1}{2\pi\ii}\int_{\mathbb{R}} \mu_{11}^{(xt)}(x,t,z) \frac{r^*(z)}{z}\eul^{-2\ii z x+\frac{\ii t}{2z}} \dd z.
\end{aligned}   
\end{equation}

Notice that the following analogs of Lemma 2.3 and Theorem 2.1 from 
\cite{Zhou98} hold:

\begin{lemma}\label{lem_est}
    Assume that $r(k)\in H^n(\mathbb{R})$ and $r(k)$ vanishes to order at least $2n$ at $k=0$. Then 
\begin{equation}
    \label{est_1}
    \|C_- w^+\|_{L^2}, ~\|C_+ w^-\|_{L^2}\leq \frac{\|r\|_{H^n}}{(1+x^2)^{\frac{n}{2}}}.
\end{equation}
\end{lemma}

\begin{proof}
   We prove the estimate for $C_- w^+$; the argument for $C_+ w^-$ is analogous.

   First, notice that under the assumptions on $r(k)$, it follows that 
   \[g(k):=r(k)\eul^{-\frac{\ii t}{2 k}}\in H^n(\mathbb{R})\]  

   Observe that for a sufficiently nice function $f$ (i.e. $f\in L^2(\mathbb{R})$), we have
   \[
   (C_- f)(k)=\lim_{\epsilon \downarrow 0} \frac{1}{2\pi\ii}\int_{\mathbb{R}} \frac{f(s)}{s-(k-\ii\epsilon)}\dd s.
   \]
Furthermore, for $f\in L^2(\mathbb{R})$, the following holds
\[
f(k)=\frac{1}{\sqrt{2\pi}}\int_{\mathbb{R}} \eul^{\ii k\xi}\hat f(\xi)\dd \xi,\quad \hat f(\xi)=\frac{1}{\sqrt{2\pi}}\int_{\mathbb{R}} \eul^{-\ii k\xi} f(k)\dd k
\]
and
\[
(C f)(z)=\frac{1}{2\pi\ii}\int_{\mathbb{R}} \frac{f(s)}{s-z}\dd s=\frac{1}{\sqrt{2\pi}}\int_{\mathbb{R}}\left(\frac{1}{2\pi\ii}\int_{\mathbb{R}} \frac{\eul^{\ii s\xi}}{s-z}\dd s\right)\hat f(\xi)\dd \xi.
\]
For $z\in \mathbb{C}^-$, the Jordan’s lemma implies
\[
\frac{1}{2\pi\ii}\int_{\mathbb{R}} \frac{\eul^{\ii s\xi}}{s-z}\dd s=\begin{cases}
    \eul^{\ii z\xi},\quad \xi<0,\\
    0,\quad \xi>0,
\end{cases}
\]
Thus,
\[
(C f)(z)=\frac{1}{\sqrt{2\pi}}\int_{-\infty}^0 \eul^{\ii z\xi}\hat f(\xi)\dd \xi,
\]
and
\begin{equation}\label{C_-F}
(C_- f)(k)=\frac{1}{\sqrt{2\pi}}\int_{-\infty}^0 \eul^{\ii k\xi}\hat f(\xi)\dd \xi,\quad \widehat{(C_- f)}(\xi)=\chi_{(-\infty,0)}\hat f(\xi).
\end{equation}

Applying \eqref{C_-F} to $f(k):=g(k)\eul^{2\ii k x}$ and using $\hat f(\xi)=\hat g(\xi-2x)$, we obtain
\[
(C_- f)(k)=\frac{1}{\sqrt{2\pi}}\int_{-\infty}^0 \eul^{\ii k\xi}\hat r(\xi-2x)\dd \xi=\frac{\eul^{2\ii k x}}{\sqrt{2\pi}}\int_{-\infty}^{-2x} \eul^{\ii k\eta}\hat r(\eta)\dd \eta,
\]
and, using Plancherel’s theorem and \eqref{C_-F}, we obtain
\begin{equation}\label{norm_Pl}
\begin{aligned}
      &\|C_- f\|_{L^2}^2= \|\widehat{C_- f}\|_{L^2}^2=\int_{-\infty}^0 |\hat g(\xi-2x)|^2\dd \xi=\int_{-\infty}^{-2x}|\hat g(\eta)|^2\dd \eta\\
      &=\int_{-\infty}^{-2x} \frac{1}{(1+\eta^2)^n}(1+\eta^2)^n|\hat g(\eta)|^2\dd \eta\leq \frac{\|r\|_{H^n}}{(1+ x^2)^n}
\end{aligned} 
\end{equation}
    
\end{proof}

\begin{lemma}\label{lem_est_2}
    Assume that $r(k)\in H^n(\mathbb{R})$ and $r(k)$ vanishes to order at least $2n$ at $k=0$. Then for all fixed $t$
\begin{equation}
    \label{est_2}
    \xi(x,t)\in L^2((1+x^{2n})\dd x,(0,+\infty)).
\end{equation}
Assume additionally that $r(k)$ vanishes to order at least $2n+1$ at $k=0$.
then
\begin{equation}
    \label{est_3}
    \beta(x,t)\in L^2((1+x^{2n})\dd x,(0,+\infty)).
\end{equation}

\end{lemma}

\begin{proof} First, consider \eqref{xi_via_mu}. 

Notice that 
\[
\int_{\mathbb{R}} \mu_{11}^{(xt)}(x,t,z) r^*(z)\eul^{-2\ii z x+\frac{\ii t}{2z}} \dd z=\int_{\mathbb{R}} \left((\mathds{1}-C_{w(x,t,z)})^{-1}I\right)_{11} r^*(z)\eul^{-2\ii z x+\frac{\ii t}{2z}} \dd z,
\]
and
\[
(\mathds{1}-C_{w})^{-1}I=I+C_{w}I+C_{w}(\mathds{1}-C_{w})^{-1}C_{w}I,\quad C_{w}(\mathds{1}-C_{w})^{-1}C_{w}I=C_{w}(\mu-I).
\]

Thus, we can split the integral into three parts:
\[
\int_{\mathbb{R}} \left((\mathds{1}-C_{w(x,t,z)})^{-1}I\right)_{11} r^*(z)\eul^{-2\ii z x+\frac{\ii t}{2z}} \dd z=I_1+I_2+I_3
\]
with
\begin{align*}
    &I_1=\int_{\mathbb{R}} r^*(z)\eul^{-2\ii z x+\frac{\ii t}{2z}} \dd z,\\
   &I_2=\int_{\mathbb{R}} \left(C_{w}I\right)_{11}r^*(z)\eul^{-2\ii z x+\frac{\ii t}{2z}} \dd z,\\
   &I_3=\int_{\mathbb{R}} \left(C_{w}(\mu-I)\right)_{11}r^*(z)\eul^{-2\ii z x+\frac{\ii t}{2z}} \dd z
\end{align*}

Consider $I_1$. Notice that under the assumptions on $r(k)$, it follows that 
   \[g(k):=r(k)\eul^{-\frac{\ii t}{2 k}}\in H^n(\mathbb{R}),\]
   and thus, 
   \[
   I_1\in L^2((1+x^{2n})\dd x,(0,+\infty))
   \]
by the properties of the Fourier transform.

Consider $I_2$. First, notice that the fact that $w_\pm$ are nilpotent implies
\[
C_{w}I(w^++w^-)=(C_+w^-)w^++(C_-w^+)w^-.
\]
Furthermore, $C_+-C_-=\mathds{1}$ yields
\[
(C_+w^-)w^++(C_-w^+)w^-=-(C_+w^-)(C_-w^+)+(C_-w^+)(C_+w^-)+(C_+w^-)(C_+w^+)-(C_-w^+)(C_-w^-).
\]
Since $(C_+w^-)(C_+w^+)$ is analytic in the upper half-plane, while  $(C_-w^-)(C_-w^+)$ is analytic in the lower half-plane, the Cauchy’s theorem implies that
\[\int_{\mathbb{R}}(C_+w^-)(C_+w^+) \dd z=0,\quad \int_{\mathbb{R}}(C_-w^+)(C_-w^-)\dd z=0.
\]
Now, applying  Lemma \ref{lem_est} together with Cauchy-Schwarz’s inequality, we get
\[
|I_2|\leq \int_{\mathbb{R}} |((C_+w^-)(C_-w^+))_{11}|+|((C_-w^+)(C_+w^-))_{11}|\dd z\leq \frac{C}{1+x^{2n}}
\]
Therefore, 
\[
\xi(x,t)\in L^2((1+x^{2n})\dd x,(0,+\infty)).
\]

Notice that under the assumptions on $r(k)$, it follows that 
   \[\frac{r(k)}{k}\eul^{-\frac{\ii t}{2 k}}\in H^n(\mathbb{R}),\]
thus \eqref{est_3} follows.
    
\end{proof}

Applying Lemma~\ref{lem_est} and Lemma~\ref{lem_est_2} in the case $n=0$, we obtain the following

\begin{corollary}
    Assume that $u_{0}(x)\in H^{1,2}((0,\infty))+2\pi k$. Then for all fixed $t\geq 0$
    \begin{equation}\label{beh_xi_alp_beta_1}
         \xi(x,t),~\beta(x,t)\in L^2((0,+\infty))
    \end{equation}
\end{corollary}

Moreover, assuming $u_{0x}(x)\in H^{1,2}((0,+\infty))$ and proceeding as in \cite{L12}, we can prove that $\frac{\dd}{\dd k}r(k)\in L^2(\mathbb{R})$. Therefore, applying Lemma~\ref{lem_est} and Lemma~\ref{lem_est_2} in the case $n=1$, we have

\begin{corollary}
    Assume that $u_{0}(x)\in H^{2,2}((0,+\infty))+2\pi k$. Then for all fixed $t\geq 0$
    \begin{equation}\label{beh_xi_alp_beta_2}
         \xi(x,t),~\beta(x,t)\in L^2((1+x^{2})\dd x,(0,+\infty)).
    \end{equation}
\end{corollary}

\subsection{Existence result for Problem I}\label{sec:8.4}

\begin{proposition}
    Let $u_0(x)\in H^{2,2}((0,+\infty))+2\pi k$ with some $k\in \mathbb{Z}$. Then the problem  
\begin{subequations}\label{SP_IBVP}
\begin{align}\label{SP-2_IBVP}
&u_{xt} = \sin u,\quad \quad x \geq 0,\quad
     0\leq t < T<\infty;\\
     \label{ic_IBVP}
     &u(x,0) = u_0(x), \quad x \geq 0;\\
     \label{as_IBVP}
     &u(x,t)\in \tilde H^{1,2}((0,+\infty)), \quad
     0\leq t < T<\infty
\end{align}
\end{subequations}
has a unique solution, $u(x,t)$, which can be represented in terms of the solution of the associated RH problem \textbf{RH$^{(xt)}$} via \eqref{sol_via_RH_xt}--\eqref{Fund_u_xt}.

\end{proposition} 

\begin{proof} As it was shown above, the existence of a solution of the RH problem \textbf{RH$^{(xt)}$} follows from Zhou's vanishing lemma. Next, we construct $u(x,t)$ by following Steps~2--4 of Proposition \ref{uniq_I}. As shown in Section
 \ref{sec:6}, the function $u(x,t)$ constructed in this way indeed satisfies the sG equation. Moreover, the considerations of Section \ref{sec:7.2} imply that $u(x,t)$ is $C^1$ with respect to both variables $x$ and $t$, whereas Section \ref{sec:8.3} establishes that $u(x,t)\in \tilde H^{1,2}((0,+\infty))$.

Therefore, it remains to prove that $u(x,0)=u_0(x)$ for all $x\geq 0$.
Notice that ${ M}^{(xt)}(x,0,k)$ satisfies the RH problem \textbf{RH$^{(x)}$}, see Section \ref{subsec:4.1.1}. Thus
 by uniqueness of solution of RH problem \textbf{RH$^{(x)}$}, we conclude that $\hat{ M}^{(xt)}(y,0,k)=\hat{ M}^{(x)}(y,k)$. In particular, 
\[
 {u_x}(x,0)={u}_{0x}(x).
\]

From the solution of the Riemann--Hilbert problem, we extract the quantities $\sin u(x,t)$, $\cos u(x,t)$, and $u_x(x,t)$. Hence, the function $u(x,t)$ can be recovered up to an additive constant of the form $2\pi k$, $k \in \mathbb{Z}$. 
The constant is then uniquely determined by imposing the initial condition, together with continuity of the reconstructed solution.
\end{proof}

\section{The uniqueness and existence results for Problem II}\label{sec:9}

\subsection{Uniqueness result for Problem II}\label{sec:9.1}

Similarly to Problem I, in this section, assuming the existence of a solution to Problem~II (with prescribed behavior as $x \to +\infty$), we address its uniqueness.

\begin{proposition}
Let $u(x,t)$ be a solution of the sG equation in the domain $x<0$, $0<t<T$ such that $u(x,t)\in C^1([0,T],\tilde H^{1,2}((-\infty,0)))$. Then $u(x,t)$ is uniquely determined by $u(x,0)$, $x\in(-\infty,0]$, and $u(0,t)$, $0<t<T$, through a unique solution of the Riemann--Hilbert problem, for which the data (jump matrix and residue condition) are given in terms of spectral functions associated with $u(x,0)$ and $u(0,t)$.
 
\end{proposition}

\begin{proof} The uniqueness of the solution to the RH problem implies, using \eqref{i_beh-M_(xt)__}, the following procedure for representing
$u(x,t)$ in terms of $\check M^{(xt)}(x,t,k)$:

  \begin{enumerate}[Step 1.]
\item Given $u(x,0)$ construct $a(k)$, $b(k)$ (see Section \ref{subsec:2.1.4}); given $u(0,t)$ construct $A(k)$ and $B(k)$ (see Section \ref{subsec:2.1.5}).
  
    \item Having $a(k)$, $b(k)$ compute $\kappa_1^0$, $\kappa_2^0$, and determine $\tilde a(k)$ and $\tilde b(k)$ via  \eqref{s_via_til_s}; having $A(k)$ and $B(k)$ compute $\kappa_i^0$, $\kappa_i^T$, $i=1,2$, and determine $\tilde A(k)$ and $\tilde B(k)$ via \eqref{S_via_til_S}.
    With these data, formulate the Riemann–Hilbert problem \textbf{$\widecheck{\text{RH}}^{(xt)}$}, see Section \ref{subsec:5.2}. ;
    \item 
    Solve this RH problem for all $x\leq 0$ and $t\in[0,T]$
    and 
evaluate its  solution $\check M^{(xt)}(x,t,k)$  at $k=\infty$:
    \begin{equation}\label{hatM_xt__inf__}
        \check M^{(xt)}(x,t,k)= I+\frac{1}{4\ii k}\begin{pmatrix}
           -\eta(x,t)&\xi(x,t)\\\xi(x,t)&\eta(x,t)
       \end{pmatrix}+O(\frac{1}{k^2}). 
    \end{equation}
    \item 
    Define $ u_x(x,t)$, $x\leq 0$ and $t\in[0,T]$, from this expansion as follows (cf. \eqref{inf_M_(xt)__}):
    \begin{equation}\label{sol_via_RH_xt__}
      u_x(x,t)=\xi(x,t).
    \end{equation}
 \item  Having $u_{x}(x,t)$, construct $u(x,t)$ by
\begin{equation}\label{Fund_u_xt__}
   u(x,t) = \int^x_{
   -\infty} u_{x}(y,t)\dd y+2\pi k=u(0,t)+ \int_0^x u_{x}(y,t)\dd y,
\end{equation}
where $k$ is defined by the behavior of $u(x,0)$ as $x\to-\infty$.

\end{enumerate}

\end{proof}

\subsection{Behavior as $x\to-\infty$}\label{sec:9.2}

In Section~\ref{sec:9.1}, the Riemann--Hilbert problem was derived under the assumption that a solution to Problem~II for the sine--Gordon equation exists and satisfies prescribed asymptotic conditions as $x \to -\infty$. We now reverse this perspective: instead of starting from a solution of the sine--Gordon equation, we take initial data from an appropriate class and study the corresponding Riemann--Hilbert problem directly. The regularity properties of the quantities involved in the reconstruction procedure were established in Section~\ref{sec:7.3}. The aim of the present section is to analyze the behavior of the reconstructed solution as $x \to -\infty$.
To fix ideas, we assume that $d^*(k)$ has no zeros in $\overline{\mathbb{C}^+}$. 

First, we construct $h$ and $\tilde h$ to improve the behavior at $k=\infty$ and $k=0$, and to achieve the matching near $k=\epsilon$ (c.f. \cite{L12}).

Let
\[
h_0(k)=\frac{\rho_1}{\ii k}, \qquad \tilde h_0(k)=\ii \tilde \rho_1 k-\tilde \rho_2k^2,
\]
where
\[
\rho(k)=\frac{\rho_1}{\ii k}+O\left(\frac{1}{k^2}\right), \qquad 
\tilde \rho(k)=\ii \tilde \rho_1 k-\tilde \rho_2k^2+O(k^3).
\]
Define
\[
F(k)=\Xi(k)+h_0(k)-\tilde h_0(k),
\]
which is analytic in \(\mathbb{C}^+\cap\{|k-\epsilon|<\epsilon/2\}\). For $n\ge1$, set
\[
G(k)=(k+\ii)^{n+2}F(k), \qquad 
T_nG(k)=\sum_{j=0}^n \frac{G^{(j)}(\epsilon)}{j!}(k-\epsilon)^j,
\]
and
\[
H_0(k)=\frac{T_nG(k)}{(k+\ii)^{n+2}}, \qquad 
H(k)=\tfrac12\big(H_0(k)+\overline{H_0(-\bar k)}\big).
\]
Then
\[
h(k)=\frac{\rho_1}{\ii k}+H(k), \qquad 
\tilde h(k)=\ii \tilde \rho_1 k-\tilde \rho_2k^2
\]
satisfy $f(k)=\overline{f(-\bar k)}$, are continuous on $|k|=\epsilon$. Moreover,
\[
h(k)=O\left(\frac{1}{k}\right),\ k\to\infty, \qquad 
\tilde h(k)=O(k),\ k\to 0,
\]
\[
h(k)-\rho(k)=O\left(\frac{1}{k^2}\right),\ k\to\infty, \qquad 
\tilde h(k)-\tilde \rho(k)=O(k^3),\ k\to 0,
\]
and
\[
h(k)-\tilde h(k)+\Xi(k)
=O\big((k-\epsilon)^n\big), \quad k\to\epsilon.
\]

Furthermore, since $\rho(k)-\tilde \rho(k)+ \Xi(k)=0$ for $k\in\mathbb{R}$, we also have
\[
h(k)-\rho(k)=-\tilde \rho(k) + \tilde h(k)+O\big((k-\epsilon)^n\big), \quad k\to\epsilon.
\]

Introduce a matrix-valued function $\check{\check M}(x,t,k)$ related to
$\check M(x,t,k)$ by
\begin{equation}\label{hat_M_}
   \check{\check M}(x,t,k)=\begin{cases}
        \check M(x,t,k)\eul^{-\ii k p(x,t,k)\sigma_3}\begin{pmatrix}
            1&-h(k)\\
        0&1
   \end{pmatrix}\eul^{\ii k p(x,t,k)\sigma_3},\quad \{|k|>\epsilon\}\cap\mathbb{C}_+,\\
        \check  M(x,t,k)\eul^{-\ii k p(x,t,k)\sigma_3}\begin{pmatrix}
            1&0\\
        h^*(k)&1
\end{pmatrix}\eul^{\ii k p(x,t,k)\sigma_3},\quad \{|k|>\epsilon\}\cap\mathbb{C}_-,\\
     \check  M(x,t,k)\eul^{-\ii k p(x,t,k)\sigma_3}\begin{pmatrix}
            1&-\tilde h(k)\\
       0 &1
   \end{pmatrix}\eul^{\ii k p(x,t,k)\sigma_3},\quad \{|k|<\epsilon\}\cap\mathbb{C}_+,\\
       \check   M(x,t,k)\eul^{-\ii k p(x,t,k)\sigma_3}\begin{pmatrix}
            1&0\\
        \tilde h^*(k)&1
\end{pmatrix}\eul^{\ii k p(x,t,k)\sigma_3},\quad \{|k|<\epsilon\}\cap\mathbb{C}_-
    \end{cases}  
    \end{equation}
and conclude that it satisfies the following RH problem:
\begin{enumerate}
    \item Jump condition across $\mathbb{R}\cup\{|k|=\epsilon\}$
    \begin{subequations}
    \label{jump_M_(xt)_hat_}
        \begin{equation}
           \check{\check M}_-^{(xt)}(x,t,k)= \check{\check M}_+^{(xt)}(x,t,k)\check{\check J}^{(xt)}(x,t,k),\quad k\in \mathbb{R}\cup\{|k|=\epsilon\} 
        \end{equation}
        where
         \begin{equation} \label{jump_M_(xt)_matr_hat_}
        \check{\check J}^{(xt)}(x,t,k)=\eul^{-\ii kp(x,t,k)\sigma_3} \check{\check J}^{(xt)}_0(k)\eul^{\ii kp(x,t,k)\sigma_3}
        \end{equation} 
        with $p(x,t,k)=x-\frac{t}{4 k^2}$ and
\begin{equation}\label{jump_M_(xt)_matr_0_hat_}
\check{\check J}^{(xt)}_0(k)=\begin{cases}
      \begin{pmatrix}
       1&h(k)-\rho(k)\\0&1
   \end{pmatrix} \begin{pmatrix}
       1&0\\h^*(k)-\rho^*(k)&1
   \end{pmatrix},\quad k\in\mathbb{R}\cap\{|k| >\epsilon\},\\
   \begin{pmatrix}
       1&0\\\tilde \rho^*(k)-\tilde h^*(k)&1
\end{pmatrix}\begin{pmatrix}
       1&\tilde \rho(k)-\tilde h(k)\\0&1
   \end{pmatrix},\quad k\in\mathbb{R}\cap\{|k| <\epsilon\},\\
   \begin{pmatrix}
       1&h(k)-\tilde h(k)+\Xi(k)\\
       0&1
   \end{pmatrix},\quad k\in\mathbb{C}^+\cap\{|k| =\epsilon\},\\
   \begin{pmatrix}
       1&h^*(k)-\tilde h^*(k)+\Xi^*(k)\\
       0&1
\end{pmatrix},\quad k\in\mathbb{C}^-\cap\{|k| =\epsilon\}.
   \end{cases} 
\end{equation}
    \end{subequations}
Recall that 
\[
\rho(k)-\tilde \rho(k)+ \Xi(k)=0,\quad k\in\mathbb{R}.
\]
\item Behavior at $\infty$:
\begin{equation}\label{inf_hatM_(xt)_hat_}
       \check{\check M}^{(xt)}(x, t,k)=I+O\left(\frac{1}{ k}\right),\quad k\to\infty.
\end{equation}
\end{enumerate}
Notice that
\begin{equation}\label{hatM_via_mu}
     \check{\check M}(x,t,k)=I+\frac{1}{2\pi\ii}\int_\Sigma \frac{\mu(x,t,z)(\check w^+ + \check w^-)(x,t,z)}{z-k}dz.
\end{equation}
In particular, the properties of $h$ and $\tilde h$ together with dominated convergence theorem (limits are not-tangential) imply that
\begin{equation}\label{xi_via_mu_}
\begin{aligned}
 \xi(x,t)=&\frac{2}{\pi}\int_{\mathbb{R}\setminus(-\epsilon,\epsilon)} \mu_{11}(x,t,z) (h(z)-\rho(z))\eul^{-2\ii z x+\frac{\ii t}{2z}} \dd z-\\
 &\frac{2}{\pi}\int_{(-\epsilon,\epsilon)} \mu_{11}(x,t,z) (\tilde \rho(z)-\tilde h(z))\eul^{-2\ii z x+\frac{\ii t}{2z}} \dd z -\\
  &\frac{2}{\pi}\int_{|z|=\epsilon, z\in \mathbb{C}^+} \mu_{11}(x,t,z) (h(k)-\tilde h(k)+ \Xi(k))\eul^{-2\ii z x+\frac{\ii t}{2z}} \dd z.
\end{aligned}   
\end{equation}
\begin{equation}\label{beta_via_mu_}
\begin{aligned}
 \beta(x,t)=&\frac{1}{2\pi\ii}\int_{\mathbb{R}\setminus(-\epsilon,\epsilon)} \mu_{11}(x,t,z) \frac{(h(z)-\rho(z))}{z}\eul^{-2\ii z x+\frac{\ii t}{2z}} \dd z-\\
 &\frac{2}{\pi}\int_{(-\epsilon,\epsilon)} \mu_{11}(x,t,z) \frac{(\tilde \rho(z)-\tilde h(z))}{z}\eul^{-2\ii z x+\frac{\ii t}{2z}} \dd z -\\
  &\frac{2}{\pi}\int_{|z|=\epsilon, z\in \mathbb{C}^+} \mu_{11}(x,t,z) \frac{(h(k)-\tilde h(k)+ \Xi(k))}{z}\eul^{-2\ii z x+\frac{\ii t}{2z}} \dd z.
\end{aligned}   
\end{equation}

Again, following the ideas of 
\cite{Zhou98} (c.f. Lemma 2.9), we conclude that the first two terms in \eqref{xi_via_mu_}--\eqref{beta_via_mu_} can be treated in the same way as in \ref{lem_est} and \ref{lem_est_2}, and yield the desired $L^2$-bounds. As for the third term, since the integration is over the upper semicircle and $x\leq 0$, the exponential factor $\eul^{(-2ikx+\tfrac{it}{2k})}$ decays exponentially fast as $x\to-\infty$ on the contour away from the real line. In addition, the choice of $h$ and $\tilde h$ ensures that $h-\tilde h+\Xi$ vanishes to the required order at the points where the arc intersects the real axis. It follows that this contribution is also in $L^2$. Therefore, we obtain the following result
\begin{corollary}
    Assume that $u_0(x)\in H^{2,2}((-\infty,0))+2\pi k$. Then for all fixed $t\geq 0$
    \begin{equation}\label{beh_xi_alp_beta_2_}
         \xi(x,t),~\beta(x,t)\in L^2((1+x^{2})\dd x,(-\infty,0)).
    \end{equation}
\end{corollary}

\subsection{Existence result for Problem II}\label{sec:9.3}

\begin{proposition}
    Let $\{u_0(x)\in H^{2,2}((-\infty,0))+2\pi k; ~v_0(t)\in  H^{1}(0, T)\}$ be a set of functions satisfying internal compatibility  condition  
     $u_0(0)=v_0(0)$. 
Then the IBVP  
\begin{subequations}\label{SP_IBVP__}
\begin{align}\label{SP-2_IBVP__}
&u_{xt} = \sin u,\\
\label{boundary_IBVP__}
    &u(0,t) = v_0(t), \quad
     0\leq t < T\leq\infty;\\
     \label{ic_IBVP__}
     &u(x,0) = u_0(x), \quad x \leq 0;\\
     &u(x,t)\in \tilde H^{1,2}((-\infty,0)), \quad x \to-\infty, \quad
     0\leq t < T<\infty.
\end{align}
\end{subequations}
has a unique solution, $u(x,t)$, which can be represented in terms of the solution of the associated RH problem \textbf{$\widecheck{\text{RH}}^{(xt)}$} via \eqref{sol_via_RH_xt__}--\eqref{Fund_u_xt__}.

\end{proposition} 

\begin{proof}
The existence of a solution of the RH problem \textbf{$\widecheck{\text{RH}}^{(xt)}$} follows from Zhou's vanishing lemma. Furthermore, Section \ref{sec:6} shows that $u(x,t)$ satisfies the sG equation, while Section \ref{sec:7.3} implies that $u(x,t)$ is $C^1$ in both $x$ and $t$, and Section~\ref{sec:9.2} establishes that $u(x,t)\in \tilde H^{1,2}((-\infty,0))$.

Therefore, it remains to prove that:
\begin{enumerate}[1.]

    \item  $u(x,0)=u_0(x)$.

    \item $ u(0,t) = v_0(t)$.
\end{enumerate}

\textbf{1.} Following the approach used for the right half-line problem, we establish a relation between the
 solution $\check{ M}^{(xt)}(x,t,k)$ of the Riemann--Hilbert problem \textbf{$\widecheck{\text{RH}}^{(xt)}$} problem evaluated at $t = 0$ and the solution $\check M^{(x)}(x,k)$ of the \textbf{$\widecheck{\text{RH}}^{(x)}$}  in the form of the multiplication by an matrix factor:

     \begin{equation*}
         P^{(x)}(x,k)=\eul^{-\ii kx\sigma_3}P_0^{(x)}(k)\eul^{\ii kx\sigma_3}
     \end{equation*}
        with
        \begin{equation*}
          P_0^{(x)}(k)=\begin{cases}
        \begin{pmatrix}
           1&-\frac{B(k)}{a^*(k)d^*(k)}\\
           0&1
        \end{pmatrix},\quad k\in\mathbb{C}^+\setminus\{|k|<\epsilon\}\\
\begin{pmatrix}
           1&-\frac{\tilde B(k)}{\tilde a^*(k)\tilde d^*(k)}\\
           0&1
        \end{pmatrix},\quad k\in\mathbb{C}^+\cap\{|k|<\epsilon\}\\
        \begin{pmatrix}
           1&0\\
           \frac{B^*(k)}{a(k)d(k)}&1
        \end{pmatrix},\quad k\in\mathbb{C}^-\setminus\{|k|<\epsilon\},\\
       \begin{pmatrix}
           1&0\\
           \frac{\tilde B^*(k)}{\tilde a(k)\tilde d(k)}&1
        \end{pmatrix},\quad k\in\mathbb{C}^-\cap\{|k|<\epsilon\}.
     \end{cases}
       \end{equation*}

Thus, the problem reduces to show that  ${ N}^{(x)}(x,k)$ defined by
 \begin{equation}
   { N}^{(x)}(x,k)=\check{ M}^{(xt)}(x,0,k)P^{(x)}(x,k)
 \end{equation}
 satisfies the Riemann--Hilbert problem \textbf{$\widecheck{\text{RH}}^{(x)}$}.

 \begin{itemize}
     \item The jump conditions match by construction.

     \item Notice that $-\frac{B(k)}{a^*(k)d^*(k)}\eul^{-2\ii k x}=O\left(\frac{\eul^{-2\ii k x}}{k}\right)$. In particular, we have $-\frac{B(k)}{a^*(k)d^*(k)}\eul^{-2\ii k x}=O\left(\frac{1}{k}\right)$ for $x\leq0$ in $\mathbb{C}^+$, and thus the normalization condition is satisfied.

     \item Let's check that the residue conditions match.

    For $k\in \mathbb{C}^+$, we have

    \begin{align}
     { N}^{(xt)(1)}(x,k)&=\check{M}^{(xt)(1)}(x,0,k),\\
     { N}^{(xt)(2)}(x,k)&=-\frac{B(k)\eul^{-2\ii k x}}{a^*(k)d^*(k)}\check{M}^{(xt)(1)}(x,0,k)+\check{M}^{(xt)(2)}(x,0,k).
 \end{align}
Taking into account the residue condition \eqref{res-M-_(xt)__} and the symmetries \eqref{sym_a_} and \eqref{sym_A_}, the singularities of the second column at $k=-\eta_j$ cancel.

On the other hand, at $k=- k_j$, the second column is singular due to the singularity of $\frac{1}{a^*(k)}$, and, using $d^*(-k_j)=b^*(- k_j)B(- k_j)$, the corresponding residue condition takes the form \eqref{res-M-_(x)__}. 

The singularities in $\mathbb{C}^-$  can be treated in a similar way.

 \end{itemize}

 Thus, by uniqueness of solution of RH problem \textbf{$\widecheck{\text{RH}}^{(x)}$}, we conclude that $\check{ M}^{(xt)}(x,0,k)P^{(x)}(x,k)=\check{ M}^{(x)}(x,k)$.

 Recall that $u_{0x}(x)$ is expressed in terms of the second coefficient in the expansion of $\check{ M}^{(x)}(x,k)$ near $k=\infty$. In this regard, notice that $-\frac{B(k)}{a^*(k)d^*(k)}\eul^{-2\ii k x}$ decays exponentially fast for $x<0$ in $\mathbb{C}^+$,
which implies that
\[
 {u}_{x}(x,0)={u}_{0x}(x).
\]

\textbf{2.}
Introduce
\[
     P^{(t)}_0(k)=\begin{cases}
         \begin{pmatrix}
             \frac{A(k)}{d^*(k)}&0\\
             b^*(k)&\frac{d^*(k)}{A(k)}
         \end{pmatrix},\quad k\in\mathbb{C}^+\cap\{|k|>\epsilon\},\\
                \begin{pmatrix}
             \frac{\tilde A(k)}{\tilde d^*(k)}&0\\
             \tilde b^*(k)&\frac{\tilde d^*(k)}{\tilde A(k)}
         \end{pmatrix},\quad k\in\mathbb{C}^+\cap\{|k|<\epsilon\},\\
                \begin{pmatrix}
             \frac{d(k)}{A^*(k)}&-b(k)\\
            0&\frac{A^*(k)}{d(k)}
         \end{pmatrix},\quad k\in\mathbb{C}^-\cap\{|k|>\epsilon\},\\
                  \begin{pmatrix}
             \frac{\tilde d(k)}{\tilde A^*(k)}&-\tilde b(k)\\
            0&\frac{\tilde A^*(k)}{\tilde d(k)}
         \end{pmatrix},\quad k\in\mathbb{C}^-\cap\{|k|<\epsilon\},
     \end{cases}
 \]
and define  
\[
   \hat{{ N}}^{(xt)}(t,k):=\check{ M}^{(xt)}(0,t,k)P^{(t)}(t,k)
 \]
with $P^{(t)}(t,k)=\eul^{\frac{\ii t}{4 k}\sigma_3}P^{(t)}_0(k)\eul^{-\frac{\ii t}{4 k}\sigma_3}$.

 \begin{enumerate}
     \item The jump condition 
     for $\hat{{ N}}^{(xt)}(t,k)$ coincides with that for $\check{ M}^{(t)}(t,k)$
    by construction.

    \item Noticing that $\frac{A(k)}{d^*(k)}=1+O(\frac{1}{k})$ and $b^*(k)\eul^{-\frac{\ii t}{2 k}}=O(\frac{1}{k})$ in $\mathbb{C}^+$ for all $t\geq0$, the normalization condition is satisfied.

    \item Let's check that the residue conditions match.

    For $k\in \mathbb{C}^+$, we have

      \begin{align}
     \hat{{ N}}^{(xt)(1)}(t,k)&=\frac{A(k)}{d^*(k)}\check{M}^{(xt)(1)}(0,t,k)+b^*(k)\eul^{-\frac{\ii t}{2 k}}\check{M}^{(xt)(2)}(0,t,k),\\
     \hat{{ N}}^{(xt)(2)}(t,k)&=\frac{d^*(k)}{A(k)}\check{M}^{(xt)(2)}(0,t,k).
 \end{align}
Taking into account the residue condition \eqref{res-M-_(xt)__} and  $A(-\eta_j)=-\frac{b^*(-\eta_j) B(-\eta_j)}{a^*(-\eta_j)}$ together with the symmetries \eqref{sym_a_} and \eqref{sym_A_}, the singularities of the first column at $k=-\eta_j$ cancel.

On the other hand, at $k= \mu_j$, the second column is singular due to the singularity of $\frac{1}{A(k)}$, and, using $d^*(\mu_j)=b^*(\mu_j)B(\mu_j)$, the corresponding residue condition takes the form \eqref{res-M+_(t)}. 

The singularities in $\mathbb{C}^-$  can be treated in a similar way.
    \end{enumerate}

    By the uniqueness of solution of the RH problem \textbf{$\widecheck{\text{RH}}^{(t)}$}, it follows  that
 $\hat{{ N}}^{(xt)}(t,k)=\check{ M}^{(t)}(t,k)$.

Since $v_{0t}(t)$ is expressed in terms of the first coefficient (via $\sin\tfrac{v_0}{2}$ and $\cos\tfrac{v_0}{2}$) in the expansion of $\check{ M}^{(t)}(0,t,k)$ near $k=0$, it is important to control $P^{(t)}(t,k)$ near this point. 
In this regard, notice that in $\mathbb{C}^+$ as $k\to 0$
\[
\frac{\tilde A(k)}{\tilde d^*(k)}=1+O(k) 
\]
and
\[
\tilde b^*(k)\eul^{2\ii k\left(-\frac{t}{4 k^2}\right)}=O(k\eul^{- \frac{\ii t}{2k}}).
\]
In particular, $\tilde b^*(k)\eul^{-\frac{\ii t}{2 k}}$ decays exponentially fast as $k\to 0$ in $\mathbb{C}^+$ for all $t>0$. Therefore,
 \[
 P^{(t)}(t,k)= I+\begin{pmatrix}
         O(k)& 0\\
         O(k\eul^{- \frac{\ii t}{2k}}) &O(k)
     \end{pmatrix}, \quad k\to 0,
 \]
 and the expressions for $u_t(0,t)$,  given in \eqref{hmbbgg}
 remains unchanged.
    
\end{proof}

\section{Concluding remarks} Throughout this paper, we consider the generic case, in the sense that the initial data for Problem I and the initial and boundary data for Problem II are assumed to generate the corresponding spectral functions with finitely many zeros, all of which are simple. Under this assumption, the associated RH problems can be formulated by imposing residue conditions at the zeros of the spectral functions. The restriction to simple zeros is not essential, however. Indeed, the general case can be treated by replacing the residue conditions with a jump condition on a sufficiently large circle enclosing all the zeros; see \cites{BFS04,BFS06,BS08}.

\section{Acknowledgment}

This research was funded in part by the Austrian Science Fund (FWF), grant no.
10.55776/ESP691, and the Research Council of Norway, project 361083 (Ukraina-NASTRAN cooperation).

\bibliographystyle{RS}
\bibliography{shepelsky_etal}

\end{document}